\documentclass{article}
\usepackage{amssymb,amsmath,amsthm,ascmac}
\usepackage{enumerate}
\newtheorem{thm}{Theorem}

\newtheorem{lem}{Lemma}[section]

\newtheorem{rem}{Remark}[section]

\newcommand{\dis}{\displaystyle}

\newcommand{\R}{{\Bbb R}}
\newcommand{\N}{{\Bbb N}}

\newcommand{\pa}{\partial}

\makeatletter

\@addtoreset{equation}{section}
\makeatother

\usepackage{fancyhdr}
\title{\Large\sf Existence of blowup solutions to the semilinear heat equation
 with double power nonlinearity}
\author{Junichi Harada
\\{\small Faculty of Education and Human Studies, Akita University}
\\[1mm]{\small email: harada-j@math.akita-u.ac.jp}}

\begin{document}
\maketitle
\thispagestyle{empty}

 \begin{abstract}
 We consider
 the semilinear heat equation
 $u_t=\Delta u+|u|^{p-1}u-|u|^{q-1}u$ in $\R^n\times(0,T)$,
 where $n=5$, $p=\frac{n+2}{n-2}$ and $q\in(0,1)$.
 By the presence of $-|u|^{q-1}u$,
 this equation has a finite time extinction property.
 We show the existence of a new type of blowup solutions
 by using this property.
 In fact,
 we obtain such blowup solutions
 by connecting
 a specific blowup solution of $u_t=\Delta u+|u|^{p-1}u$
 and
 a specific solution of $u_t=\Delta u-|u|^{q-1}u$,
 and
 by adding correction terms.
 \end{abstract}

 \tableofcontents

 \section{Introduction}
 \label{sec_1}
 We are concerned with blowup solutions to the semilinear heat equation.
 \begin{equation}\label{eq1.1}
 \begin{cases}
 u_t = \Delta u+|u|^{p-1}u-|u|^{q-1}u
 &
 \text{in } \R^n\times(0,T),\\
 u|_{t=0}=u_0(x)
 &
 \text{on } \R^n.
 \end{cases}
 \end{equation}
 The exponents $p$ and $q$ are given by
 \[
 p=\tfrac{n+2}{n-2},
 \qquad
 q\in(0,1).
 \]
 From the sign of the last term $|u|^{q-1}u$,
 \eqref{eq1.1} admits a unique local classical solution
 for any bounded continuous initial data $u_0$.
 If there exits $T>0$ such that $\limsup_{t\to T}\|u(t)\|_\infty=\infty$,
 we say that a solution $u(x,t)$ blows up in a finite time $T$.
 Blowup solutions are often classified into two cases
 in terms of the blowup rate.
 As in the case of $u_t=\Delta u+|u|^{p-1}u$,
 we define
 \begin{align*}
 \limsup_{t\to T}(T-t)^\frac{1}{p-1}\|u(t)\|_\infty
 &<\infty
 \qquad (\text{type I}),
 \\
 \limsup_{t\to T}(T-t)^\frac{1}{p-1}\|u(t)\|_\infty
 &=\infty
 \qquad (\text{type II}).
 \end{align*}
 The factor $(T-t)^{-\frac{1}{p-1}}$ in this definition
 comes from
 the blowup rate of ODE solutions of
 $u_t=|u|^{p-1}u-|u|^{q-1}u$.
 On the other hand,
 by the presence of the absorption term $-|u|^{q-1}u$ with $q\in(0,1)$,
 this equation has a finite time extinction property.
 In fact,
 if $|u_0(x)|<A<1$,
 the solution $u(x,t)$ becomes identically zero at $t=T$,
 where $T$ is a certain positive time.
 By the uniqueness of solutions to \eqref{eq1.1},
 the solution $u(x,t)$ must be zero for $t>T$.
 In this paper,
 we try to understand the role of the extinction properties
 in blowup phenomena.

 Our motivation of this paper
 comes from the work
 by Le Coz, Martel and Rapa\'el \cite{Coz} (see also \cite{Matsui}).
 They consider
 \begin{equation}\label{eq1.2}
 i\pa_tu
 +
 \Delta u
 +
 |u|^{p-1}u
 +
 \epsilon
 |u|^{q-1}
 u
 =0
 \end{equation}
 with $p=1+\frac{4}{n}$, $\epsilon=\pm1$ and $q\in(1,p)$.
 For the case $\epsilon=1$,
 they show the existence of minimal blowup solutions of \eqref{eq1.2}
 satisfying
 \[ 
 \|\nabla u(t)\|_2=c_q(T-t)^{-\sigma_q}
 \]
 with
 $\sigma_q=\frac{4}{4+n(q-1)}$.
 In \cite{Coz},
 they mention that
 this blowup speed $\sigma_q$ is discontinuous at $q=p$,
 since the blowup speed for $q=p$ is given by
 $\|\nabla u(t)\|_2=c(T-t)^{-1}$,
 which corresponds to a standard minimal blowup solution of
 $i\pa_tu+\Delta u+2|u|^{p-1}u=0$
 (on the other hand, $\sigma_q$ is continuous at $q=1$).
 The problem \eqref{eq1.1} is the heat equation version of \eqref{eq1.2}.
 We are also interested in
 the blowup dynamics of \eqref{eq1.1} and $q$ dependence of them.
 Since the range of $q$ in our case is quite different from their case,
 blowup solutions of \eqref{eq1.1} are expected to exhibit
 a new type of asymptotic formula.

 The study of type II blowup solutions to the nonlinear heat equation
 are started from Herrero and Vel\'azquez
 \cite{HerreroV2,HerreroV3} (see also \cite{Mizoguchi}).
 In the pioneering work \cite{HerreroV2,HerreroV3},
 they construct positive radial type II blowup solutions to
 the Fujita equation
 \begin{equation}\label{eq1.3}
 u_t=\Delta u+|u|^{p-1}u
 \end{equation}
 for the case $p>p_{\rm JL}$.
 Here we do not give the definition of the exponent
 $p_{\rm JL}$,
 which is defined only for $n\geq10$ and satisfies $p_{\rm JL}>\frac{n+2}{n-2}$.
 After their work,
 Seki obtains
 a new type of type II blowup solutions to \eqref{eq1.3}
 for a certain range of $p\geq p_{\rm JL}$ \cite{Seki2, Seki3}.
 For the critical case $p=\frac{n+2}{n-2}$,
 Filippas, Herrero and Vel\'azquez \cite{Filippas}
 find a quite different type of
 type II blowup solutions to \eqref{eq1.3}.
 They formally obtain type II blowup solutions
 by using the matched asymptotic expansion technique.
 The first rigorous proof for the existence of
 type II blowup solutions to the critical problem
 is given by Schweyer \cite{Schweyer}.
 He constructs a type II blowup solution for $n=4$.
 Recently
 del Pino, Musso and Wei
 \cite{del_Pino, del_Pino2, del_Pino3}
 construct type II blowup solutions for the critical case.
 They
 develop a new method so-called the inner-outer gluing method
 and obtain type II blowup solutions for the case $n=3,4,5$.
 The author \cite{Harada,Harada_2}
 also constructs type II blowup solutions
 for the critical case with $n=5,6$
 applying their gluing method.

 We now go back to \eqref{eq1.1}.
 Since a blowup solution $u(x,t)$ behaves like
 $u(x,t)\to\infty$ near the singular point,
 the behavior of the solution near singular point
 is dominated by $u_t=\Delta u+|u|^{p-1}u$.
 We recall that
 a blowup solution of \eqref{eq1.3}
 constructed in \cite{del_Pino, del_Pino2, del_Pino3}
 behaves like
 \[
 u(x,t)
 =
 \lambda(t)^{-\frac{2}{p-1}}
 {\sf Q}(y)
 \qquad
 \text{with }
 x=\lambda(t) y
 \]
 near the singular point.
 Here
 ${\sf Q}(y)$ is the standard Talenti function
 and
 $\lambda(t)$ is a scaling function satisfying $\lim_{t\to T}\lambda(t)=0$.
 Under this assumption,
 the solution is expected to be
 $\lim_{t\to T}u(x,t)=0$
 for $|y|\to\infty$, $|x|\ll1$.
 For such a region,
 the solution is governed by
 \begin{equation}\label{eq1.4}
 u_t=\Delta u-|u|^{q-1}u.
 \end{equation}
 The asymptotic formula for solutions to \eqref{eq1.4}
 satisfying $\lim_{t\to T}u(x,t)=0$ is well understood.
 There are two possibilities.
 \begin{enumerate}[(I)]
 \item 
 One is given by
 $u(x,t)=(1-q)^\frac{1}{1-q}(T-t)^\frac{1}{1-q}$,
 \item
 the other is more delicate and complicated,
 which is a kind of type II behavior (see \cite{Guo,Seki}).
 \end{enumerate}
 From this observation,
 we look for solutions of the form
 \begin{align}\label{eq1.5}
 u(x,t)=
 \begin{cases}
 \lambda(t)^{-\frac{2}{p-1}}
 {\sf Q}(y)
 &
 \text{near the singular point},
 \\
 \text{(I) or (II)}
 &
 \text{for }
 |y|\to\infty,\
 |x|\ll1.
 \end{cases}
 \end{align}
 Therefore
 our problem is reduced to the following question:
 ``can we construct blowup solutions of \eqref{eq1.1}
 by connecting
 a specific blowup solution of \eqref{eq1.3}
 and
 a specific solution of \eqref{eq1.4} satisfying $\lim_{t\to T}u(x,t)=0$?''
 This question is a concrete motivation of this paper,
 which seems to provide a new perspective on blowup problems.
 This paper gives an affirmative answer to this question.
 In fact,
 such a solution will be constructed
 by the addition of several correction terms.

 \section{Main result}
 \label{sec_2}
 To state our main theorem,
 we briefly prepare several notations.
 Let ${\sf Q}(y)$ be the Talenti solution
 defined by
 \[
 {\sf Q}(y)
 =
 \left( 1+\tfrac{1}{n(n-2)}|y|^2 \right)^{-\frac{n-2}{2}}.
 \]
 This gives a positive solution of
 $\Delta_y{\sf Q}+{\sf Q}^p=0$.
 Furthermore
 let
 ${\sf U}(\xi)$
 be a positive radially symmetric solution of
 $\Delta_\xi{\sf U}-{\sf U}^q=0$,
 and
 let
 ${\sf U}_\infty(x)$
 be a nonnegative radially symmetric solution of
 $\Delta_x{\sf U}-{\sf U}^q=0$
 with ${\sf U}_\infty|_{x=0}=0$.
 Let us consider a linearized problem around ${\sf U}_\infty(x)$.
 \begin{equation}\label{eq2.1}
 \pa_tw
 =
 \Delta_x
 w
 -
 q{\sf U}_\infty^{q-1}
 w.
 \end{equation}
 Fix $J\in\N$
 and
 let $\Theta_J(x,t)$ be the explicit solution of \eqref{eq2.1}
 defined in Section \ref{eq4.2}.
 Furthermore
 we introduce
 \begin{itemize}
 \item
 $\eta(t)=(T-t)^{J(\frac{2}{1-q}-\gamma)^{-1}}$,
 \item
 $y=\lambda(t)^{-1}x$,
 \item
 $\xi=\eta(t)^{-1}x$,
 \item
 $z=(T-t)^{-\frac{1}{2}}x$.
 \end{itemize}
 The cut off functions $\chi_i$ ($i=1,2,3,4$) are defined in Section \ref{sec_5.1}.
 \begin{thm}\label{Thm1}
 Let $n=5$ and $J\in\N$.
 There exist $T>0$ and a radially symmetric solution
 $u(x,t)\in C(\R^5\times[0,T))\cap C^{2,1}(\R^5\times(0,T))$
 of \eqref{eq1.1} satisfying the following properties\,{\rm$:$}
 \begin{enumerate}[\rm(i)]
 \item the solution $u(x,t)$ is written as
 \begin{align*}
 u(x,t)
 &=
 \dis
 \lambda^{-\frac{n-2}{2}}
 {\sf Q}(y)
 \chi_2
 +
 \lambda^{-\frac{n-2}{2}}\sigma
 T_1(y)
 \chi_1
 \\
 & \quad
 -
 (
 \eta^\frac{2}{1-q}{\sf U}(\xi)
 \chi_2
 +
 {\sf U}_\infty(x)
 (1-\chi_2)
 \chi_4
 )
 (1-\chi_1)
 \\
 & \quad
 -
 (
 \theta(x,t)
 +
 \Theta_J(x,t)
 )
 (1-\chi_2)
 \chi_3
 +
 u_1(x,t),
 \end{align*}
 where  $T_1(y)$ is a bounded smooth function defined in Section
 {\rm\ref{sec_4.1}},
 $\theta(x,t)$ is a correction term defined in Section {\rm\ref{sec_4.5}}
 and
 $u_1(x)$ is a remainder term,
 \item
 $\lambda(t)=(1+o(1))(T-t)^{\frac{2}{6-n}\Gamma_J}$
 with
 $\Gamma_J=
 (\frac{2J}{1-q}+\frac{2}{1-q}-\gamma)
 (\frac{2}{1-q}-\gamma)^{-1}$,
 \item
 $\sigma(t)=(1+o(1))\lambda\dot\lambda$,
 \item
 there exist
 ${\sf d}_1,{\sf c}_1>0$
 and
 $\kappa>2$
 such that
 \begin{align*}
 |u_1(x,t)|
 &<
 \begin{cases}
 (T-t)^{\frac{1}{2}{\sf d}_1}
 \lambda(t)^{-\frac{n-2}{2}}
 \sigma
 & {\rm for}\ |y|<1,
 \\[1mm] \dis
 (T-t)^{{\sf d}_1}
 \eta(t)^\frac{2}{1-q}
 {\sf U}(\xi)
 & {\rm for}\ |y|>1,\ |\xi|<1,
 \\
 (T-t)^{{\sf d}_1}
 (1+|z|^2)^{\frac{3}{2}{\sf d}_1}
 \Theta_J(x,t)
 & {\rm for}\ |\xi|>1,\ |z|<(T-t)^{-\frac{1}{\kappa}},
 \\
 \tfrac{1}{8}
 {\sf U}_\infty(x)
 &
 {\rm for}\ |z|>(T-t)^{-\frac{1}{\kappa}},\ |x|<1,
 \\
 {\sf c}_1
 &
 {\rm for}\ |x|>1.
 \end{cases}
 \end{align*}
 \end{enumerate}
 \end{thm}
 \begin{rem}
 \label{Rem2.1}
 The blowup rate of this solution is given by
 \[
 \|u(t)\|_\infty
 =
 \|\lambda(t)^{-\frac{n-2}{2}}{\sf Q}(y)\|_\infty
 =
 \lambda(t)^{-\frac{n-2}{2}}
 =
 (T-t)^{-\frac{n-2}{6-n}\Gamma_J}
 \]
 with
 $\Gamma_J=
 (\frac{2J}{1-q}+\frac{2}{1-q}-\gamma)
 (\frac{2}{1-q}-\gamma)^{-1}$.
 Since $0<\frac{2}{1-q}-\gamma<2$ {\rm(}see {\rm\eqref{eq4.9}}{\rm)},
 the blowup rate
 $\sigma_q=\frac{n-2}{6-n}\Gamma_J$
 diverges to infinity as $q\to 1$
 for any $J\in\N$.
 This discontinuity phenomenon at $q=1$
 is the same as that of the problem \eqref{eq1.2}
 discussed in {\rm\cite{Coz}}.
 \end{rem}
 \begin{rem}\label{Rem2.2}
 Since the solution in Theorem {\rm\ref{Thm1}}
 satisfies $\lim_{t\to T}u(x,t)=0$ in the region $|z|\sim1$,
 the solution is approximately dominated by $u_t=\Delta u-|u|^{q-1}u$
 in this region.
 However,
 it can be seen that the effect from $|u|^{p-1}u$ is not so small
 in this region.
 Therefore
 the correction term $\theta(x,t)$ is needed to compensate this effect.
 \end{rem}
 \begin{rem}\label{Rem2.3}
 We can expect that
 three types of type II blowup solutions exist in \eqref{eq1.1}.
 \begin{enumerate}[\rm(i)]
 \item 
 Blowup solutions described in Theorem {\rm\ref{Thm1}}
 give the first one,
 which corresponds to {\rm (II)} in \eqref{eq1.5}.
 \item
 The second one is obtained from {\rm (I)} in \eqref{eq1.5}.
 A formal computation will be given in Section {\rm\ref{sec_4}}.
 Since its asymptotic formula is simpler,
 we omit a rigorous proof.
 \item
 The third one is given by
 blowup solutions of \eqref{eq1.3} obtained in {\rm\cite{Filippas,del_Pino}}.
 We recall from {\rm\cite{Filippas,del_Pino,Harada}} that
 \eqref{eq1.3} with $n=5$
 admits infinite many blowup solutions $\{u_k(x,t)\}_{k\in\N}$ satisfying
 \begin{align*}
 \|u_k(t)\|_\infty
 =
 (T-t)^{-3k}
 \qquad
 (k\in\N).
 \end{align*}
 We can verify that
 the additional term $-|u|^{q-1}u$ in \eqref{eq1.1}
 does not have much effect on $u_k(x,t)$
 only for the case $k=1$
 {\rm(}see the asymptotic formula for $u_k(x,t)$ in {\rm\cite{Filippas}}
 and the proof in {\rm\cite{del_Pino,Harada}}{\rm)}.
 Hence
 we can construct a blowup solution of \eqref{eq1.1}
 asymptotically behaves like $u_k(x,t)$ with $k=1$.
 \end{enumerate}
 \end{rem}
 \begin{rem}\label{Rem2.4}
 From the asymptotic formula for blowup solutions to \eqref{eq1.3}
 {\rm(}see {\rm\cite{Filippas}}{\rm)},
 we believe that
 a type II blowup exists only for the case $n=5$.
 \end{rem}
 \begin{rem}\label{Rem2.5}
 For the case $1\leq q<p$,
 \eqref{eq1.1} is similar to \eqref{eq1.2}
 discussed in {\rm\cite{Coz}}.
 When $q$ is close to $1$,
 \eqref{eq1.1} admits
 blowup solutions
 with the same asymptotic form as
 $u_k(x,t)$ {\rm(}$k\in\N${\rm)}
 describe in Remark {\rm\ref{Rem2.3}} {\rm(iii)}
 {\rm(}see proof in {\rm\cite{del_Pino,Harada}}{\rm)}.
 However,
 it is not clear what happens when $q\to p$.
 \end{rem}

 We explain the strategy of the proof.
 Our proof is a combination of
 the matched asymptotic expansion technique
 (see {\rm\cite{Filippas,Seki}})
 and
 the inner-outer gluing method developed in
 (see {\rm\cite{Cortazar,del_Pino,del_Pino2,del_Pino3}}).
 We divide the whole space $x\in\R^n$ into four parts.
 \begin{enumerate}[\rm(i)]
 \item inner region $|y|\sim1$\quad ($x=\lambda(t)y$),
 \item semiinner region $|\xi|\sim1$ \quad ($x=\eta(t)\xi$),
 \item selfsimilar region $|z|\sim1$ \quad ($x=z\sqrt{T-t}$),
 \item outer region $|x|\sim1$,
 \end{enumerate}
 where $0<\lambda(t)<\eta(t)<\sqrt{T-t}$.
 We first construct approximate solutions in each region separately.
 In the region (i),
 \eqref{eq1.1} is approximately written as
 $u_t=\Delta u+|u|^{p-1}u$.
 Hence
 we can use
 $u(x,t)=\lambda(t)^{-\frac{n-2}{2}}{\sf Q}(y)
 +\lambda(t)^{-\frac{n-2}{2}}\sigma(t)T_1(y)$
 as a approximate solution in this region,
 which is the same one as in \cite{Filippas,del_Pino,Harada}
 (Section \ref{sec_4.1}).
 In the region (ii)-(iii),
 we first look for solutions of $u_t=\Delta u-|u|^{p-1}u$.
 The function
 $u(x,t)=\eta(t)^\frac{2}{1-q}{\sf U}(\xi)$ gives an approximate solution in (ii)
 (Section \ref{sec_4.2}).
 This solution is borrowed from \cite{Seki}.
 In the region (iii),
 we need a correction term $\theta(x,t)$.
 We will see that
 $u(x,t)={\sf U}_\infty(x)+\theta(x,t)+\Theta_J(x,t)$
 gives an appropriate approximate solution in this region
  (Section \ref{sec_4.2}, Section\ref{sec_4.5}).
 The region (iv) is not important in our analysis.
 We next investigate
 two matching conditions
 between (i) and (ii),
 and between (ii) and (iii)
 to connect each solution obtained in the first step
 (Section \ref{sec_4.3} - Section \ref{sec_4.4}).
 This procedure determines $\lambda(t)$ and $\eta(t)$.
 We finally construct blowup solutions
 near the approximate solutions obtained in the first and the second step
 by a fixed point argument (Section \ref{sec_5}).
 To investigate the behavior of solutions more precisely,
 we introduce a system of three parabolic equations
 from \eqref{eq1.1}.
 This three equations correspond to (i) - (iii) respectively.
 The first equation is treated in the same way as in \cite{Cortazar,del_Pino}
 (Section \ref{sec_6}),
 and the last two equations are analyzed based on \cite{Seki}
 (Section \ref{sec_7}, Section \ref{sec_8} respectively).
 Although our argument is not new,
 there are two remarks on the analysis of this three parabolic equations.
 \begin{enumerate}
 \item
 Due to Lemma \ref{Lem3.3}, 
 our argument in the first equation
 slightly simplifies that of \cite{Cortazar,del_Pino,Harada}
 (Section \ref{sec_6.2}). 

 \item
 Unlike the setting in \cite{Seki},
 we introduces the second equation corresponding to (ii).
 This provides a simpler alternative proof of \cite{Seki}.
 In fact,
 our argument does not require
 the explicit form of the heat kernel of
 $w_t=\Delta w-q{\sf U}_\infty^{q-1}w$
 (Section \ref{sec_8}). 
 \end{enumerate}

 \section{Preliminary}
 \label{sec_3}

 \subsection{Linearized problem around the ground state ${\sf Q}(y)$
 and its eigenvalue problem}
 \label{sec_3.1}
 To describe the asymptotic behavior of solutions around
 the ground state ${\sf Q}(y)={\sf Q}_\lambda(y)|_{\lambda=1}$,
 we study the linearized problem.
 \[
 \epsilon_t=H_y\epsilon,
 \] 
 where the operator $H_y$ is defined by
 \[
 H_y=\Delta_y+V(y),
 \qquad
 V(y)=f'({\sf Q}(y))=p{\sf Q}(y)^{p-1}.
 \]
 Following the idea in \cite{Cortazar},
 we consider this problem on $B_R$ instead of $\R^n$.
 Solutions of $\epsilon_t=H_y\epsilon$ are completely described by
 the following eigenvalue problems.
 \begin{equation}\label{eq3.1}
 \begin{cases}
 -H_y\psi=\mu\psi & \text{in } B_R,
 \\
 \psi=0 & \text{on } \pa B_R,
 \\
 \psi \text{ is radially symmetric}.
 \end{cases}
 \end{equation}
 We denote the $i$th eigenvalue of \eqref{eq3.1} by $\mu_i^{(R)}$
 and the associated eigenfunction by $\psi_i^{(R)}$.
 We normalize $\psi_i^{(R)}(r)$ as  $\psi_i^{(R)}(0)=1$.
 We recall two lemmas obtained in \cite{Cortazar} and \cite{Harada}.
 \begin{lem}[Lemma 3.1-Lemma 3.2 \cite{Harada}]
 \label{Lem3.1}
 There exists $c>0$ such that
 if $R>1$,
 then
 \begin{align*}
 0
 <
 \psi_1^{(R)}(r)
 &<
 c
 \left( 1+r \right)^{-\frac{n-1}{2}}
 e^{-\sqrt{|\mu_1|}\,r}
 \quad\text{\rm for } r\in(0,R),
 \\
 |\psi_2^{(R)}(r)|
 &<
 c
 \left( 1+r \right)^{-(n-2)}
 \qquad\text{\rm for } r\in(0,R).
 \end{align*}
 The constant $\mu_1<0$ is
 the first eigenvalue of \eqref{eq3.1}
 on $\R^n$.
 \end{lem}
 \begin{lem}[Lemma 7.2 \cite{Cortazar}, Lemma 3.3 \cite{Harada}]
 \label{Lem3.2}
 Let $n\geq5$.
 Then
 \begin{itemize}
 \item
 $\dis\lim_{R\to\infty}\mu_1^{(R)}=\mu_1<0$ and
 \item
 there exits $c>0$ such that $\mu_2^{(R)}>cR^{-(n-2)}$ for $R>1$.
 \end{itemize}
 \end{lem}

 We here provide information of the third eigenvalue of \eqref{eq3.1}.
 This lemma simplifies arguments in Section \ref{sec_6}.
 \begin{lem}
 \label{Lem3.3}
 Let $n\geq5$.
 There exists $c>0$ such that
 $\mu_3^{(R)}>cR^{-\frac{n}{2}}$ for $R>1$.
 \end{lem}
 \begin{proof}
 We note that $Z_1(r)={\Lambda_y}{\sf Q}(r)$ gives a radial solution of
 $H_yZ=0$ on $\R^n$.
 Let $Z_2(r)=\Gamma(r)$ be another independent radial solution of $H_yZ=0$ on $\R^n$.
 From a direct computation,
 we see that
 \begin{align}
 \label{eq3.2}
 \begin{cases}
 Z_2(r)
 =
 a_1r^{-(n-2)}
 +
 o(r^{-(n-2)})
 \qquad
 \text{as } r\to0,
 \\
 Z_2(r)
 =
 a_2
 +
 o(1)
 \qquad
 \text{as } r\to\infty
 \end{cases}
 \end{align}
 for some $a_1,a_2\not=0$.
 It is known that
 any radial solution of
 the inhomogeneous problem $H_yu=f$
 can be represented
 by $Z_1,Z_2,f$
 (see proof of Lemma 7.2 in \cite{Cortazar}).
 From this formula,
 $\psi_3^{(R)}(r)$ is written as
 \begin{align}
 \label{eq3.3}
 \psi_3^{(R)}(r)
 &=
 k\mu_3^{(R)}Z_2(r)\int_0^r\psi_3^{(R)}Z_1r_1^{n-1}dr_1
 +
 k\mu_3^{(R)}Z_1(r)\int_r^R\psi_3^{(R)}Z_2r_1^{n-1}dr_1
 \nonumber
 \\
 &\qquad
 -
 k\mu_3^{(R)}
 \frac{Z_2(R)}{Z_1(R)}
 Z_1(r)
 \int_0^R\psi_3^{(R)}Z_1r^{n-1}dr.
 \end{align}
 The constant $k$ depends only on $n$.
 To estimate $\mu_3^{(R)}$,
 we compute $L^2$ norm of both sides on \eqref{eq3.3}.
 Throughout this proof,
 we denote by $c$ a general positive constant
 which depends only on $n$.
 For simplicity,
 we write 
 \[
 \|\psi\|_{L_{\text{rad}}^2}^2
 =
 \int_0^R
 |\psi(r)|^k
 r^{n-1}
 dr.
 \]
 We first compute the first term of \eqref{eq3.3}.
 \begin{align*}
 \|
 Z_2(r)
 \int_0^r\psi_3^{(R)}Z_1
 &
 r_1^{n-1}
 dr_1
 \|_{L_{\text{rad}}^2}
 =
 \left\|
 \cdots
 \right\|_{L_{\text{rad}}^2(r<1)}
 +
 \left\|
 \cdots
 \right\|_{L_{\text{rad}}^2(1<r<R)}
 \\
 &\leq
 \|\psi_3^{(R)}Z_1\|_{L^\infty(r<1)}
 \|
 Z_2
 r^n
 \|_{L_{\text{rad}}^2(r<1)}
 +
 \|
 Z_2
 \|_{L_{\text{rad}}^2(1<r<R)}
 \|
 \psi_3^{(R)}
 \|_{L_{\text{rad}}^2}^2
 \|
 Z_1
 \|_{L_{\text{rad}}^2}^2.
 \end{align*}
 We note that
 $Z_1(r)\in L_{\text{rad}}^2(\R^n)$ if $n\geq5$
 and
 $\|Z_2\|_{L_\text{rad}^2(1<r<R)}\lesssim R^\frac{n}{2}$
 (see \eqref{eq3.2}).
 Hence
 we get
 \begin{align*}
 \|
 Z_2(r)
 \int_0^r
 &
 \psi_3^{(R)}Z_1r_1^{n-1}dr_1
 \|_{L_{\text{rad}}^2}
 \leq
 \|
 \cdots
 \|_{L_{\text{rad}}^2(r<1)}
 +
 \|
 \cdots
 \|_{L_{\text{rad}}^2(1<r<R)}
 \\
 &\leq
 \|\psi_3^{(R)}Z_1\|_{L_{\text{rad}}^\infty(r<1)}
 \|Z_2r^n\|_{L_{\text{rad}}^2(r<1)}
 +
 \|Z_2\|_{L_{\text{rad}}^2(1<r<R)}
 \|\psi_3^{(R)}\|_{L_{\text{rad}}^2}
 \|Z_1\|_{L_{\text{rad}}^2}
 \\
 &\leq
 c
 (
 \|\psi_3^{(R)}\|_{L_{\text{rad}}^\infty}
 +
 R^\frac{n}{2}
 \|\psi_3^{(R)}\|_{L_{\text{rad}}^2}
 ).
 \end{align*}
 From the order property of eigenvalues:
 $0<\mu_3^{(R_2)}<\mu_3^{(R_1)}$ if $R_1<R_2$
 and
 a local parabolic estimates,
 it holds that
 $\|\psi_3^{(R)}\|_{L_{\text{rad}}^\infty}
 <c\|\psi_3^{(R)}\|_{L_{\text{rad}}^2}$.
 Therefore
 it follows that
 \begin{align}\label{eq3.4}
 \|
 Z_2(r)
 \int_0^r
 \psi_3^{(R)}Z_1r_1^{n-1}dr_1
 \|_{L_{\text{rad}}^2}
 \leq
 c
 R^\frac{n}{2}
 \|\psi_3^{(R)}\|_{L_{\text{rad}}^2}.
 \end{align}
 The second term of \eqref{eq3.3} can be computed in the same manner.
 \begin{align}\label{eq3.5}
 \|
 Z_1(r)
 \int_r^R
 \psi_3^{(R)}Z_2r_1^{n-1}dr_1
 \|_{L_{\text{rad}}^2}
 &\leq
 \|Z_1\|_{L_{\text{rad}}^2}
 (
 \int_0^1
 |\psi_3^{(R)}Z_2|r_1^{n-1}dr_1
 +
 \int_1^R
 |\psi_3^{(R)}Z_2|r_1^{n-1}dr_1
 )
 \nonumber
 \\
 &\leq
 c
 R^2
 \|\psi_3^{(R)}\|_{L_{\text{rad}}^2}.
 \end{align}
 Since $(\psi_2^{(R)},\psi_3^{(R)})_{L_y^2(B_R)}=0$,
 the last term of \eqref{eq3.3} is rewritten as
 \begin{align}\label{eq3.6}
 \int_0^R
 \psi_3^{(R)}
 Z_1r^{n-1}
 dr
 =
 (\psi_3^{(R)},Z_1)_{L_{\text{rad}}^2}
 =
 (\psi_3^{(R)},Z_1-\alpha\psi_2^{(R)})_{L_{\text{rad}}^2}
 \end{align}
 for any $\alpha\in\R$.
 We here choose
 $\alpha=
 (k\mu_2^{(R)}\frac{Z_2(R)}{Z_1(R)}(\psi_2^{(R)},Z_1)_{L_{\text{rad}}^2})^{-1}$.
 We express $\psi_2^{(R)}(r)$
 in the same form as \eqref{eq3.3}.
 \begin{align*}
 \nonumber
 Z_1(r)
 &
 -
 \alpha
 \psi_2^{(R)}(r)
 \\
 &=
 Z_1(r)
 -
 \alpha
 k\mu_2^{(R)}Z_2(r)\int_0^r\psi_2^{(R)}Z_1r^{n-1}dr
 -
 \alpha
 k\mu_2^{(R)}Z_1(r)\int_r^R\psi_2^{(R)}Z_2 r^{n-1}dr
 \nonumber
 \\
 \nonumber
 &\quad
 +
 \alpha
 k\mu_2^{(R)}
 \tfrac{Z_2(R)}{Z_1(R)}
 Z_1(r)
 (\psi_2^{(R)},Z_1)_{L_{\text{rad}}^2}
 \\
 &=
 -
 \alpha
 k\mu_2^{(R)}Z_2(r)\int_0^r\psi_2^{(R)}Z_1r^{n-1}dr
 -
 \alpha
 k\mu_2^{(R)}Z_1(r)\int_r^R\psi_2^{(R)}Z_2 r^{n-1}dr.
 \end{align*}
 We can estimate
 two integrals on the right-hand side
 in the same way as \eqref{eq3.4} - \eqref{eq3.5}.
 \begin{align*}
 \|
 Z_2(r)
 \int_0^r
 \psi_2^{(R)}Z_1r_1^{n-1}dr_1
 \|_{L_{\text{rad}}^2}
 +
 \|
 Z_1(r)
 \int_r^R
 \psi_2^{(R)}Z_2r_1^{n-1}dr_1
 \|_{L_{\text{rad}}^2}
 \leq
 c
 (R^\frac{n}{2}+R^2)
 \|\psi_2^{(R)}\|_{L_{\text{rad}}^2}.
 \end{align*}
 Therefore
 we derive a better estimate of \eqref{eq3.6}.
 \begin{align}\label{eq3.7}
 (\psi_3^{(R)},Z_1-\alpha\psi_2^{(R)})_{L_{\text{rad}}^2}
 \leq
 \|\psi_3^{(R)}\|_{L_{\text{rad}}^2}
 \|Z_1-\alpha\psi_2^{(R)}\|_{L_{\text{rad}}^2}
 \leq
 c
 \alpha
 \mu_2^{(R)}
 R^\frac{n}{2}
 \|\psi_3^{(R)}\|_{L_{\text{rad}}^2}.
 \end{align}
 We note from 
 Lemma \ref{Lem3.1} and \eqref{eq3.2}
 that
 \[
 \alpha
 =
 (
 k\mu_2^{(R)}
 \tfrac{Z_2(R)}{Z_1(R)}(\psi_2^{(R)},Z_1)_{L_{\text{rad}}^2}
 )^{-1}
 \lesssim
 (\mu_2^{(R)})^{-1}
 \tfrac{Z_1(R)}{Z_2(R)}
 \lesssim
 R^{n-2}
 \cdot
 R^{-(n-2)}
 =1.
 \]
 Therefore
 we take $L^2$ norm of both sides on \eqref{eq3.3}
 and
 combine \eqref{eq3.4}-\eqref{eq3.5} and \eqref{eq3.7},
 we conclude
 \begin{align*}
 \|\psi_3^{(R)}\|_{L_\text{rad}^2}
 \leq
 c
 \mu_3^{(R)}
 R^\frac{n}{2}
 \|
 \psi_3^{(R)}
 \|_{L_{\text{rad}}^2}.
 \end{align*}
 This completes the proof.
 \end{proof}


 \subsection{Local $L^\infty$ bound and gradient estimates for parabolic equations}
 \label{sec_3.2}
In this subsection,
we consider
 \begin{equation}\label{eq_sec_3.2}
 u_t=\Delta_xu+{\bf b}(x,t)\cdot\nabla_xu+V(x,t)u+f(x,t)
 \qquad\text{in } Q,
 \end{equation}
where $Q=B_2\times(0,1)$, $B_r=\{x\in\R^n;\ |x|<r\}$.
The coefficients are assumed to be
 \begin{equation}
 \tag{A1}
 {\bf b}(x,t)\in(L^\infty(Q))^n,
 \quad
 V(x,t)\in L^\infty(Q)
 \qquad\text{with }
 \|{\bf b}\|_{L^\infty(Q)}+\|V\|_{L^\infty(Q)}<M.
 \end{equation}
For $p,q\in[1,\infty]$,
we define
 \[
 \|f\|_{L^{p,q}(Q)}
 =
 \begin{cases}
 \dis
 \left( \int_0^1\|f(t)\|_{L^p(B_2)}^qdt \right)^\frac{1}{q}
 & \text{if } q\in[1,\infty),
 \\
 \dis
 \sup_{t\in(0,1)}\|f(t)\|_{L^p(B_2)}
 & \text{if } q=\infty.
 \end{cases}
 \]

\begin{lem}[Exercise 6.5 p. 154 \cite{Lieberman}, Theorem 8.1 p. 192 \cite{Ladyzenskaja}]
 \label{Lem3.4}
 Let $p,q\in(1,\infty]$ satisfy $\frac{n}{2p}+\frac{1}{q}<1$ and assume {\rm(A1)}.
 There exists $c>0$ depending on $p$, $q$, $n$ and $M$ such that
 \begin{enumerate}[{\rm (i)}]
 \item
 if $u(x,t)$ is a weak solution of \eqref{eq_sec_3.2},
 then
 \[
 \|u\|_{L^\infty(B_1\times(\frac{1}{2},1))}
 <
 c\left( \|u\|_{L^2(Q)}+\|f\|_{L^{p,q}(Q)} \right),
 \]
 \item
 if $u(x,t)\in C(B_2\times[0,1))$ is a weak solution of \eqref{eq_sec_3.2} with $u(x,t)|_{t=0}=0$,
 then
 \[
 \|u\|_{L^\infty(B_1\times(0,1))}
 <
 c\left( \|u\|_{L^2(Q)}+\|f\|_{L^{p,q}(Q)} \right).
 \]
 \end{enumerate}
 \end{lem}

 \begin{lem}[Theorem 4.8 p. 56 \cite{Lieberman}, Theorem 11.1 p. 211 \cite{Ladyzenskaja}]
 \label{Lem3.5}
 Let $p,q\in(1,\infty]$ satisfy $\frac{n}{p}+\frac{2}{q}<1$ and assume {\rm(A1)}.
 There exists $c>0$ depending on $p$, $q$, $n$ and $M$ such that
 \begin{enumerate}[{\rm (i)}]
 \item
 if $u(x,t)\in C^{2,1}(Q)$ is a solution of \eqref{eq_sec_3.2},
 then
 \[
 \|\nabla u\|_{L^\infty(B_1\times(\frac{1}{2},1))}
 <
 c\left( \|u\|_{L^\infty(Q)}+\|f\|_{L^{p,q}(Q)} \right),
 \]
 \item
 if $u(x,t)\in C^{2,1}(Q)\cap C^{2,1}(B_2\times[0,1))$ is a solution of \eqref{eq_sec_3.2}
 with $u(x,t)|_{t=0}=0$,
 then
 \[
 \|\nabla u\|_{L^\infty(B_1\times(0,1))}
 <
 c\left( \|u\|_{L^\infty(Q)}+\|f\|_{L^{p,q}(Q)} \right),
 \]
 \item
 if $u(x,t)\in C^{2,1}(\bar Q)$ is a solution of \eqref{eq_sec_3.2} with
 $u(x,t)|_{t=0}=0$ and $u(x,t)|_{\pa B_2}=0$,
 then
 \[
 \|\nabla u\|_{L^\infty(Q)}
 <
 c\left(
 \|u\|_{L^\infty(Q)}+\|f\|_{L^{p,q}(Q)}
 \right).
 \]
 \end{enumerate}
 \end{lem}

\section{Formal derivation of blowup speed}
\label{sec_4}
 In this section,
 we derive the asymptotic behavior of solutions described in Theorem \ref{Thm1}
 by using a matched asymptotic expansion technique.
 Since our blowup solution has four characteristic lengths,
 we denote them by
\begin{itemize}
\item inner region \quad $|x|\sim\lambda(t)$,
\item semiinner region \quad $|x|\sim\eta(t)$,
\item selfsimilar region \quad $|x|\sim\sqrt{T-t}$,
\item outer region \quad $|x|\sim1$,
\end{itemize}
 where $\lambda(t)$ and $\eta(t)$ are unknown functions satisfying 
 \[
 \lambda(t)\ll\eta(t)\ll\sqrt{T-t}.
 \]

\subsection{Inner region $|x|\sim\lambda(t)$}
\label{sec_4.1}
 We first investigate the asymptotic behavior of solutions in the inner region.
 We assume that
 the solution $u(x,t)$ behaves like
 \[
 u(x,t)
 =
 \lambda(t)^{-\frac{n-2}{2}}
 {\sf Q}(y)
 +
 o(\lambda^{-\frac{n-2}{2}})
 \qquad
 \text{in the inner region},
 \]
 where $x=\lambda(t)y$
 and
 $\lambda(t)$ is an unknown function satisfying $\lambda(t)\to0$ as $t\to T$.
 Under this assumption,
 this solution blows up at $t=T$.
 Therefore
 we can expect that
 the solution approximately solves $u_t=\Delta_xu+|u|^{p-1}u$ in this region.
 Following the argument in Section 4 of \cite{Harada_2}
 (the original idea is in \cite{Filippas}),
 we put
 \[
 u(x,t)
 =
 \lambda^{-\frac{n-2}{2}}{\sf Q}(y)
 +
 \lambda^{-\frac{n-2}{2}}\sigma(t)T_1(y)
 +
 \lambda^{-\frac{n-2}{2}}\epsilon_1(y,t),
 \]
 where $0<\sigma(t)\ll1$.
 The function $\epsilon_1(y,t)$ solves
 \[
 \lambda^2\pa_t\epsilon_1
 +
 \lambda^2\dot\sigma T_1
 =
 H_y\epsilon_1
 +
 \lambda\dot\lambda
 (\Lambda_y{\sf Q})
 +
 \sigma
 H_yT_1
 +
 \lambda\dot\lambda
 \sigma
 \Lambda_yT_1
 +
 \lambda\dot\lambda
 \Lambda_y\epsilon_1
 +
 N,
 \]
 where
 \begin{itemize}
 \item 
 $H_y=\Delta_y+f'({\sf Q}(y))$,
 \item
 $\Lambda_y=\frac{n-2}{2}+y\cdot\nabla_y$,
 \item
 $N=f(u)-f(\lambda^{-\frac{n-2}{2}}{\sf Q})
 -f'(\lambda^{-\frac{n-2}{2}}{\sf Q})(u-\lambda^{-\frac{n-2}{2}}{\sf Q})$.
 \end{itemize}
 We choose
 $\sigma=\lambda\dot\lambda$
 and
 $T_1$ as a solution of $H_yT_1+\Lambda_y{\sf Q}=0$
 to cancel the lower order terms.
 Then
 $\epsilon_1(y,t)$ solves
 \[
 \lambda^2\pa_t\epsilon_1
 +
 \lambda^2\dot\sigma T_1
 =
 H_y\epsilon_1
 +
 \lambda\dot\lambda
 \sigma
 \Lambda_yT_1
 +
 \lambda\dot\lambda
 \Lambda_y\epsilon_1
 +
 N.
 \]
 From this equation,
 we can expect $|\epsilon_1(y,t)|\ll|\sigma|$.
 This implies
 \[
 u(x,t)
 =
 \lambda^{-\frac{n-2}{2}}{\sf Q}(y)
 +
 \lambda^{-\frac{n-2}{2}}\sigma
 T_1(y)
 (1+o)
 \qquad
 \text{as }
 t\to T.
 \]
 We recall that
 there exits ${\sf A}_1>0$ such that
 (see p. 12 - p. 13 in \cite{Harada_2})
 \begin{align}
 \label{eq4.1}
 T_1(y)
 &=
 {\sf A}_1
 +
 O(|y|^{-(n-4)})
 +
 O(|y|^{-2})
 \qquad
 \text{as }
 |y|\to\infty
 \qquad
 (n\geq5),
 \\
 \label{eq4.2}
 |\nabla_yT_1(y)|
 &=
 O(|y|^{-(n-3)})
 +
 O(|y|^{-3})
 \qquad
 \text{as }
 |y|\to\infty
 \qquad
 (n\geq5).
 \end{align}
 Therefore
 the asymptotic behavior of the solution is given by
 \begin{equation}\label{eq4.3}
 u(x,t)
 =
 \lambda^{-\frac{n-2}{2}}{\sf Q}(y)
 +
 \lambda^{-\frac{n-2}{2}}\sigma
 {\sf A}_1
 (1+o)
 \qquad
 \text{as }
 |y|\to\infty.
 \end{equation}

\subsection{Semiinner region $|x|\sim\eta(t)$ and Selfsimilar region $|x|\sim\sqrt{T-t}$}
\label{sec_4.2}
 Let $\eta(t)$ be an unknown function which represents the characteristic length
 of the semiinner region.
 To describe the asymptotic behavior of solutions,
 we introduce two new variables
 \[
 \xi=\eta(t)^{-1}x,
 \qquad
 z=(T-t)^{-\frac{1}{2}}x.
 \]
 We now assume
 \begin{equation}\label{eq4.4}
 \lim_{t\to T}
 u(x,t)
 =
 0
 \qquad
 \text{in } k^{-1}\eta(t)<|x|<k\sqrt{T-t}
 \end{equation}
 for arbitrary $k>1$.
 Under this assumption,
 the equation \eqref{eq1.1} is approximated by
 \begin{equation}\label{eq4.5}
 u_t=\Delta_xu-|u|^{q-1}u.
 \end{equation}
 It is known that
 there are two types of solutions of \eqref{eq4.5} satisfying \eqref{eq4.4}.
 \begin{enumerate}[(I)]
 \item One is given by $u(x,t)=\pm(1-q)^\frac{1}{1-q}(T-t)^\frac{1}{1-q}$.
 \item The other solution behaves like
 \begin{align}
 \label{eq4.6}
 u(x,t)
 &=
 \pm\eta(t)^\frac{2}{1-q}{\sf U}(\xi)+o(\eta(t)^\frac{2}{1-q})
 \qquad \text{for } |x|\sim\eta(t),
 \\
 \label{eq4.7}
 u(x,t)
 &=
 \pm{\sf U}_\infty(x)+\Theta_J(x,t)+o(\Theta_J(x,t))
 \qquad \text{for } |x|\sim\sqrt{T-t}
 \qquad
 (J\in\N).
 \end{align}
 \end{enumerate}
 We explain the case (II) along \cite{Seki}
 and 
 give definition of
 ${\sf U}(\xi)$, ${\sf U}_\infty(x)$ and $\Theta_J(x,t)$.
 Let ${\sf U}(\xi)$ be a unique solution of
 \[
 \begin{cases}
 \Delta_\xi{\sf U}-f_2({\sf U})=0
 \qquad
 \text{in }
 \R^n,
 \\
 {\sf U}(\xi)|_{\xi=0}=1,
 \\
 {\sf U}(\xi) \text{ is radially symmetric}.
 \end{cases}
 \]
 As in \cite{Seki},
 we assume
 that
 the solution
 in the semiinner region
 moves along the scale of the steady state ${\sf U}(\xi)$ like
 \begin{equation*}
 u(x,t)
 =
 \pm
 \eta(t)^\frac{2}{1-q}
 {\sf U}(\xi)^\frac{2}{1-q}
 (1+o)
 \qquad
 \text{in }
 |x|\sim\eta(t).
 \end{equation*}
 This implies \eqref{eq4.6}.
 To investigate the asymptotic behavior of solutions in the selfsimilar region,
 we introduce a singular solution ${\sf U}_\infty(x)=L_1|x|^\frac{2}{1-q}$.
 The constant $L_1$ is given by
 \begin{align}\label{eq4.8}
 L_1^{q-1}
 =
 \beta_0
 (
 \beta_0+n-2
 ),
 \quad
 \beta_0
 =
 \tfrac{2}{1-q}.
 \end{align}
 This gives a unique nonnegative radial solution of
 $\Delta_x{\sf U}-{\sf U}^q=0$ with ${\sf U}|_{x=0}=0$.
 Furthermore
 let $\gamma\in(0,\frac{2}{1-q})$ be a unique real number satisfying
 \[
 (\Delta_x-qL_1^{q-1}|x|^{-2})|x|^\gamma=0.
 \]
 It is known that
 (see Remark 1.2 (i) in \cite{Seki} p. 4)
 \begin{equation}\label{eq4.9}
 \tfrac{2}{1-q}-2<\gamma<\tfrac{2}{1-q}.
 \end{equation}
 To investigate the behavior of the solution
 on the border
 between the semiinner region and the selfsimilar region,
 we need the following property of ${\sf U}(\xi)$.
 \begin{equation}\label{eq4.10}
 {\sf U}(\xi)
 =
 {\sf U}_\infty(\xi)
 +
 {\sf B}_1
 |\xi|^\gamma
 +
 O(|\xi|^{\gamma-{\sf k}_1})
 \qquad
 \text{as }
 |\xi|\to\infty,
 \end{equation}
 where ${\sf B}_1>0$ and ${\sf k}_1>0$ (see (2.3) in \cite{Seki}).
 From \eqref{eq4.10},
 we expect that
 the solution $u(x,t)$ behaves like
 \begin{align*}
 u(x,t)
 &=
 \pm
 \eta(t)^\frac{2}{1-q}
 {\sf U}(\xi)
 (1+o)
 =
 \pm
 \eta(t)^\frac{2}{1-q}
 {\sf U}_\infty(\xi)
 (1+o)
 \\
 &=
 \pm
 {\sf U}_\infty(x)
 (1+o)
 \qquad
 \text{as }
 |\xi|\to\infty.
 \end{align*}
 From this relation,
 we assume that
 the solution in the selfsimilar region behaves like
 \[
 u(x,t)
 =
 \pm
 {\sf U}_\infty(x)
 (1+o)
 \qquad
 \text{for }
 |x|\sim\sqrt{T-t}.
 \]
 To obtain more precise asymptotic behavior of solutions,
 we consider a linearized problem around $\pm{\sf U}_\infty(x)$.
 \[
 w_t
 =
 \Delta_x
 w
 -
 f_2'(\pm{\sf U}_\infty)
 w
 =
 \Delta_x
 w
 -
 q|{\sf U}_\infty|^{q-1}
 w
 =
 \Delta_x
 w
 -
 qL_1^{q-1}
 |x|^{-2}
 w.
 \]
 Here the nonlinear term $f(u)$ is neglected,
 since $u(x,t)\to0$ for $|x|\sim\sqrt{T-t}$.
 To describe a local behavior of $w(x,t)$,
 we perform a change of variable:
 $z=(T-t)^{-\frac{1}{2}}x$,
 $T-t=e^{-\tau}$.
 The linearized equation can be written in the new variable as
 \begin{equation*}
 w_\tau
 =
 \Delta_zw
 -
 \tfrac{z}{2}\cdot\nabla_zw
 -
 qL_1^{q-1}
 |z|^{-2}
 w
 \qquad
 \text{for }
 z\in\R^n,\
 \tau\in(-\log T,\infty).
 \end{equation*}
 We define a weighted $L^2$ space.
 \begin{align*}
 L_\rho^2(\R^n)
 :=
 \{f\in L_\text{loc}^2(\R^n);\ \|f\|_\rho<\infty\},
 \qquad
 \|f\|_\rho^2=\int_{\R^n}f(z)^2\rho(z)dz
 \quad
 \text{with }
 \rho(z)=e^{-\frac{|z|^2}{4}}.
 \end{align*}
 The inner product is defined by
 \[
 (f_1,f_2)_\rho=\int_{\R^n}f_1(z)f_2(z)\rho(z)dz.
 \]
 A corresponding eigenvalue problem is given by
 \begin{equation*}
 -
 (
 \Delta_z-\tfrac{z}{2}\cdot\nabla_z-qL_1^{q-1}|z|^{-2}
 )
 e_j=\mu_ie_j
 \qquad\text{in } L_{\rho,\text{rad}}^2(\R^n).
 \end{equation*}
 It is known that
 the eigenvalue $\mu_j$ and the eigenfunction $e_j(z)$
 are explicitly given by
 \begin{itemize}
 \item $\mu_j=\frac{\gamma}{2}+j$ \quad ($j=0,1,2,\cdots$),
 \item $e_j(z)\in L_{\rho,\text{rad}}^2(\R^n)$ \quad ($j=0,1,2,\cdots$),
 \item $e_0(z)={\sf D}_0|z|^\gamma$,
 \item $e_j(z)={\sf D}_j|z|^\gamma+O(|z|^{\gamma+2})$ \quad as $|z|\to0$
 \quad ($j=1,2,\cdots$),
 \item $e_j(z)={\sf E}_j|z|^{2j+\gamma}+O(|z|^{2j-2+\gamma})$ \quad as $|z|\to\infty$
 \quad ($j=1,2,\cdots$),
 \end{itemize}
 and
 $L_{\rho,\text{rad}}^2(\R^n)$ is spanned by
 $\{e_j(z)\}_{j=0}^\infty$
 (see Lemma 2.2 p. 8 in \cite{Seki}, Corollary 2.3 p. 10 in \cite{Seki}).
 This eigenfunction is normalized as $\|e_j\|_\rho=1$ ($j=0,1,2,\cdots$).
 From this fact,
 we choose
 \[
 w(x,t)
 =
 \Theta_J(z,\tau)
 =
 K
 e^{-\mu_j\tau}e_J(z)
 =
 K
 (T-t)^{\frac{\gamma}{2}+J}
 e_J(z)
 \qquad (J=0,1,2,\cdots),
 \]
 where $K$ is an arbitrary constant.
 This implies
 \begin{align*}
 u(x,t)
 &=
 \pm
 {\sf U}_\infty(x)
 +
 \Theta_J(x,t)
 \nonumber
 \\
 &=
 \pm
 {\sf U}_\infty(x)
 +
 K(T-t)^{\frac{\gamma}{2}+J}
 e_J(z)
 \qquad
 \text{for }
 |x|\sim\sqrt{T-t}.
 \end{align*}
 This gives the formula \eqref{eq4.7}.
 The asymptotic behavior of this solution is given by
 \begin{align}\label{eq4.11}
 \nonumber
 u(x,t)
 &=
 \pm
 {\sf U}_\infty(x)
 +
 K{\sf D}_J
 (T-t)^{\frac{\gamma}{2}+J}
 |z|^\gamma
 (1+o(1))
 \\
 &=
 \pm
 {\sf U}_\infty(x)
 +
 K{\sf D}_J
 (T-t)^J
 \eta^\gamma
 |\xi|^\gamma
 (1+o(1))
 \qquad
 \text{as }
 |z|\to0.
 \end{align}

\subsection{Matching condition for the case (I)}
 \label{sec_4.3}
 We firsts consider the case (I) in Section \ref{sec_4.2}.
 Since ${\sf Q}(y)\to0$ as $|y|\to\infty$,
 \eqref{eq4.3} implies
 \begin{align*}
 u(x,t)
 &=
 \lambda^{-\frac{n-2}{2}}
 {\sf Q}(y)
 +
 \lambda(t)^{-\frac{n-2}{2}}
 \sigma
 {\sf A}_1
 (1+o)
 =
 \lambda^{-\frac{n-2}{2}}
 \sigma
 {\sf A}_1
 (1+o)
 \qquad
 \text{as } |y|\to\infty.
 \end{align*}
 This relation
 and
 the asymptotic formula (I)
 give the following matching condition.
 \[
 \lambda^{-\frac{n-2}{2}}
 \sigma
 {\sf A}_1
 =
 \lambda^{-\frac{n-2}{2}}
 \lambda
 \dot\lambda
 {\sf A}_1
 =
 \pm
 (1-q)^\frac{1}{1-q}
 (T-t)^\frac{1}{1-q}
 \qquad
 \text{as } |y|\to\infty.
 \]
 Since ${\sf A}_1>0$ (see \eqref{eq4.1}) and $\dot\lambda<0$,
 we choose a minus sign.
 Therefore
 we obtain
 \[
 \lambda(t)
 =
 (\tfrac{6-n}{2(2-q){\sf A}_1})^{\frac{2}{6-n}}
 (1-q)^{\frac{2-q}{1-q}\frac{2}{6-n}}
 (T-t)^{\frac{2-q}{1-q}\frac{2}{6-n}}.
 \]

\subsection{Matching condition for the case (II)}
 \label{sec_4.4}
 We next consider the case (II).
 We determine $\lambda(t)$ and $\eta(t)$ by the matching procedure.
 From \eqref{eq4.3} and \eqref{eq4.6}-\eqref{eq4.7},
 we obtain matching conditions ($n=5$).
 \begin{align*}
 \lambda^{-\frac{n-2}{2}}{\sf Q}(y)
 +
 \lambda^{-\frac{n-2}{2}}\sigma
 {\sf A}_1
 &=
 \pm
 \eta^\frac{2}{1-q}{\sf U}(\xi)
 \qquad
 \text{for } |y|\to\infty,\ |\xi|\to0,
 \\
 \pm
 \eta^\frac{2}{1-q}{\sf U}(\xi)
 &=
 \pm
 {\sf U}_\infty(x)
 +
 \Theta_J(x,t)
 \qquad
 \text{for } |\xi|\to\infty,\ |z|\to0.
 \end{align*}
 Due to \eqref{eq4.10} and \eqref{eq4.11},
 these conditions can be written as
 \begin{align*}
 \lambda^{-\frac{n-2}{2}}
 \sigma
 {\sf A}_1
 &=
 \pm
 \eta^\frac{2}{1-q}
 \qquad
 \text{for } |y|\to\infty,\ |\xi|\to0,
 \\
 \pm
 \eta^\frac{2}{1-q}
 (
 {\sf U}_\infty(\xi)
 +
 \eta^\frac{2}{1-q}
 {\sf B}_1|\xi|^\gamma
 )
 &=
 \pm
 (
 {\sf U}_\infty(x)
 +
 \eta^\frac{2}{1-q}
 {\sf B}_1|\xi|^\gamma
 )
 \\
 &=
 \pm
 {\sf U}_\infty(x)
 +
 K{\sf D}_J
 (T-t)^J
 \eta^\gamma
 |\xi|^\gamma
 \qquad
 \text{for } |\xi|\to\infty,\ |z|\to0.
 \end{align*}
 We now take a minus sign and obtain
 \begin{align*}
 \lambda^{-\frac{n-2}{2}}
 \sigma
 {\sf A}_1
 &=
 -
 \eta^\frac{2}{1-q}
 \qquad
 \text{for } |y|\to\infty,\ |\xi|\to0,
 \\
 -
 \eta^\frac{2}{1-q}
 {\sf B}_1
 &=
 K{\sf D}_J
 (T-t)^J
 \eta^\gamma
 \qquad
 \text{for } |\xi|\to\infty,\ |z|\to0.
 \end{align*}
 From these relations,
 we can derive
 \begin{itemize}
 \item
 $K=-{\sf D}_J^{-1}{\sf B}_1$,
 \item
 $\eta(t)=(T-t)^{\gamma_J}$
 \quad
 with
 $\gamma_J=J(\frac{2}{1-q}-\gamma)^{-1}$,
 \item
 $\lambda(t)
 =
 (\frac{6-n}{2{\sf A}_1\Gamma_J})^\frac{2}{6-n}
 (T-t)^{\frac{2}{6-n}\Gamma_J}$
 \quad
 with
 $\Gamma_J
 =(\frac{2J}{1-q}+\frac{2}{1-q}-\gamma)
 (\frac{2}{1-q}-\gamma)^{-1}$.
 \end{itemize}
 Since $\eta(t)\to0$ as $t\to T$,
 $J$ must be $J\in\N$.
 The blowup rats of the solution is given by
 \begin{align*}
 \|u(t)\|_\infty
 =
 \lambda(t)^{-\frac{n-2}{2}}
 =
 (T-t)^{-\frac{n-2}{6-n}\Gamma_J}
 \qquad
 (J\in\N).
 \end{align*}

\subsection{Corrections in the selfsimilar region $|x|\sim\sqrt{T-t}$}
\label{sec_4.5}
 Unfortunately
 solutions obtained in Section \ref{sec_4.4} does not give
 appropriate approximate solutions of \eqref{eq1.1},
 since the contribution of $f(u)$ is not negligible.
 We now reconstruct approximate solutions along Section \ref{sec_4.4}.
 As in Section \ref{sec_4.4},
 we assume
 \[
 u(x,t)
 =
 -{\sf U}_\infty(x)
 +
 o({\sf U}_\infty(x))
 \qquad
 \text{for }
 |x|\sim\sqrt{T-t}.
 \]
 To take in the effect of $f(u)$,
 we construct a solution of the form
 \begin{align}
 \label{eq4.12}
 u
 =
 -{\sf U}_\infty(x)-\theta
 =
 -{\sf U}_\infty(x)-(\theta_0+\theta_1+\cdots+\theta_L)
 \qquad
 (L\gg1)
 \end{align}
 satisfying
 the following conditions.
 \begin{enumerate}[(i) ]
 \item
 $\theta_0(x)=a_0{\sf U}_\infty^{p+1-q}
 ={a}_0L_1^{p+1-q}|x|^{\frac{2p}{1-q}+2}$
 \qquad (${a}_0\not=0$),
 
 \item $\dis\theta_i(x)={a}_i|x|^{\frac{2(p-q)i}{1-q}}\theta_0(x)
 +|x|^{\frac{2(p-q)i}{1-q}}\theta_0(x)\sum_{l=1}^Mb_i|x|^\frac{2(p-q)i}{1-q}$
 \qquad (${a}_i\not=0$)
 \qquad ($i=1,2,\cdots,L$),
 \item
 $|\Delta({\sf U}_\infty+\theta)+f({\sf U}_\infty+\theta)
 -f_2({\sf U}_\infty+\theta)|
 \ll\Theta_J(x,t)
 =(T-t)^{\frac{\gamma}{2}+J}e_J(z)$
 \qquad
 in $|x|\sim\sqrt{T-t}$.
 \end{enumerate}
 Once
 such a function
 $\theta(x)=\theta_0(x)+\theta_1(x)+\cdots+\theta_L(x)$
 is constructed,
 we can verify that
 \begin{align*}
 u(x,t)
 =
 -
 {\sf U}_\infty(x)
 -
 \theta(x)
 -
 \Theta_J(x)
 \end{align*}
 gives a new approximate solution in the selfsimilar region
 instead of \eqref{eq4.7}.
 From conditions (i)-(ii),
 the matching condition derived in Section \ref{sec_4.4} does not change.
 Therefore
 we obtain the the same $\lambda(t)$ and $\eta(t)$ as in Section \ref{sec_4.4}.
 To construct functions
 $\theta(x)=\theta_0(x)+\theta_1(x)+\cdots+\theta_L(x)$,
 we substitute \eqref{eq4.12} to \eqref{eq1.1} and
 apply the Taylor expansion.
 \begin{align*}
 0
 &=
 \Delta({\sf U}_\infty+\theta)
 +
 f({\sf U}_\infty+\theta)
 -
 f_2({\sf U}_\infty+\theta)
 \\
 &=
 \Delta{\sf U}_\infty
 +
 \Delta \theta
 +
 f({\sf U}_\infty)
 +
 \sum_{i=1}^N
 \tfrac{f^{(i)}({\sf U}_\infty)}{i!}
 \theta^i
 -
 f_2({\sf U}_\infty)
 -
 f_2'({\sf U}_\infty)
 \theta
 -
 \sum_{i=2}^N
 \tfrac{f_2^{(i)}({\sf U}_\infty)}{i!}
 \theta^i
 \\
 &=
 \Delta(\theta_0+\cdots+\theta_L)
 +
 f({\sf U}_\infty)
 +
 \sum_{i=1}^N
 \tfrac{f^{(i)}({\sf U}_\infty)}{i!}
 (\theta_0+\cdots+\theta_L)^i
 \\
 & \quad
 -
 f_2'({\sf U}_\infty)
 (\theta_0+\cdots+\theta_L)
 -
 \sum_{i=2}^N
 \tfrac{f_2^{(i)}({\sf U}_\infty)}{i!}
 (\theta_0+\cdots+\theta_L)^i.
 \end{align*}
 We determine $\theta_0(x)$, $\theta_1(x)$ and $\theta_2(x)$
 as follows.
 \begin{align}
 &
 \label{eq4.13}
 (\Delta -f_2'({\sf U}_\infty))
 \theta_0
 +
 f({\sf U}_\infty)
 =
 0,
 \\
 &
 \label{eq4.14}
 (\Delta -f_2'({\sf U}_\infty))
 \theta_1
 +
 \sum_{i=1}^N
 \tfrac{f^{(i)}({\sf U}_\infty)}{i!}
 \theta_0^i
 -
 \sum_{i=2}^N
 \tfrac{f_2^{(i)}({\sf U}_\infty)}{i!}
 \theta_0^i
 =
 0,
 \\
 &
 \label{eq4.15}
 (\Delta -f_2'({\sf U}_\infty))
 \theta_2
 +
 \sum_{i=1}^N
 \tfrac{f^{(i)}({\sf U}_\infty)}{i!}
 ((\theta_0+\theta_1)^i-\theta_0^i)
 -
 \sum_{i=2}^N
 \tfrac{f_2^{(i)}({\sf U}_\infty)}{i!}
 ((\theta_0+\theta_1)^i-\theta_0^i)
 =
 0.
 \end{align}
 In the same manner,
 we define $\theta_k(x)$ ($k=3,\cdots,L$) by
 \begin{align*}
 (\Delta -f_2'({\sf U}_\infty))
 \theta_k
 &+
 \sum_{i=1}^N
 \tfrac{f^{(i)}({\sf U}_\infty)}{i!}
 ((\theta_0+\cdots+\theta_{k-1})^i-(\theta_0+\cdots+\theta_{k-2})^i)
 \\
 &
 -
 \sum_{i=2}^N
 \tfrac{f_2^{(i)}({\sf U}_\infty)}{i!}
 ((\theta_0+\cdots+\theta_{k-1})^i-(\theta_0+\cdots+\theta_{k-2})^i)
 =0.
 \end{align*}
 From \eqref{eq4.13} - \eqref{eq4.15},
 $\theta(x)=\theta_0(x)+\theta_1(x)+\cdots+\theta_L(x)$ satisfies
 \begin{align*}
 (\Delta-f_2'({\sf U}_\infty))\theta
 &=
 -f({\sf U}_\infty)
 -
 \sum_{i=0}^N
 \tfrac{f^{(i)}({\sf U}_\infty)}{i!}
 (\theta_0+\cdots+\theta_{L-1})^i
 +
 \sum_{i=2}^N
 \tfrac{f_2^{(i)}({\sf U}_\infty)}{i!}
 (\theta_0+\cdots+\theta_{L-1})^i.
 \end{align*}
 Since $\Delta{\sf U}_\infty=f_2({\sf U}_\infty)$,
 this relation implies
 \begin{align*}
 \Delta({\sf U}_\infty+\theta)
 &=
 -
 f({\sf U}_\infty)
 -
 \sum_{i=1}^L
 \tfrac{f^{(i)}({\sf U}_\infty)}{i!}
 (\theta-\theta_L)^i
 +
 f_2({\sf U}_\infty)
 +
 f_2'({\sf U}_\infty)
 \theta
 +
 \sum_{i=2}^L
 \tfrac{f_2^{(i)}({\sf U}_\infty)}{i!}
 (\theta-\theta_L)^i
 \\
 &=
 -
 \sum_{i=0}^L
 \tfrac{f^{(i)}({\sf U}_\infty)}{i!}
 (\theta-\theta_L)^i
 +
 \sum_{i=0}^L
 \tfrac{f_2^{(i)}({\sf U}_\infty)}{i!}
 (\theta-\theta_L)^i.
 \end{align*}
 Therefore
 we get
 \begin{align}\label{eq4.16}
 \nonumber
 |
 \Delta
 ({\sf U}_\infty+\theta)
 &
 +
 f
 ({\sf U}_\infty+\theta)
 -
 f_2({\sf U}_\infty+\theta)
 |
 \\
 \nonumber
 &<
 |
 f({\sf U}_\infty+\theta)
 -
 \sum_{i=0}^L
 \tfrac{f^{(i)}({\sf U}_\infty)}{i!}
 \theta^i
 |
 +
 |
 f_2({\sf U}_\infty+\theta)
 -
 \sum_{i=0}^L
 \tfrac{f_2^{(i)}({\sf U}_\infty)}{i!}
 \theta^i
 |
 \\
 & \quad
 +
 \sum_{i=0}^L
 \tfrac{f^{(i)}({\sf U}_\infty)}{i!}
 |(\theta-\theta_L)^i-\theta^i|
 +
 \sum_{i=0}^L
 \tfrac{f_2^{(i)}({\sf U}_\infty)}{i!}
 |(\theta-\theta_L)^i-\theta^i|.
 \end{align}
 We note that
 the conditions (i) - (ii) imply
 $\theta(x)=O(|x|^\frac{2(p-q)}{1-q}{\sf U}_\infty(x))$
 and
 $\theta_L(x)=O(|x|^{\frac{2(p-q)}{1-q}(L+1)}{\sf U}_\infty(x))$.
 Therefore
 if the conditions  (i)-(ii) are verified,
 the condition (iii) follows from \eqref{eq4.16}.
 We now construct $\theta(x)=\theta_0(x)+\cdots+\theta_L(x)$.
 Equation \eqref{eq4.13} can be written as
 \begin{equation}\label{eq4.17}
 0
 =
 (
 \Delta
 -
 f_2'({\sf U}_\infty)
 )
 \theta_0
 +
 f({\sf U}_\infty)
 =
 (
 \Delta
 -
 qL_1^{q-1}
 |x|^{-2}
 )
 \theta_0
 +
 L_1^p|x|^{\frac{2p}{1-q}}.
 \end{equation}
 We easily check that
 \begin{equation}\label{eq4.18}
 \theta_0(x)
 =
 {a}_0
 {\sf U}_\infty^{p+1-q}
 =
 {a}_0
 L_1^{p+1-q}|x|^{\frac{2p}{1-q}+2}
 \end{equation}
 gives a solution of \eqref{eq4.17}.
 We now claim that
 \begin{equation}\label{eq4.19}
 -(1-q)^{-1}<{a}_0<0.
 \end{equation}
 We write
 $\theta_0(x)={a}_0L_1^{p+1-q}|x|^\beta$
 with
 $\beta=\frac{2p}{1-q}+2$.
 From \eqref{eq4.17},
 it holds that
 \begin{align*}
 {a}_0
 L_1^{p+1-q}
 \{
 \beta(\beta+n-2)
 -
 qL_1^{q-1}
 \}
 &=
 -L_1^p,
 \\
 {a}_0
 L_1^{1-q}
 \{
 \beta(\beta+n-2)
 -
 qL_1^{q-1}
 \}
 &=
 -1.
 \end{align*}
 This implies
 $-
 {a}_0^{-1}
 =
 L_1^{1-q}
 \beta(\beta+n-2)
 -
 q
 $.
 Since $L_1^{q-1}=\beta_0(\beta_0+n-2)$
 with $\beta_0=\frac{2}{1-q}$
 (see \eqref{eq4.8})
 and $\beta>\beta_0$,
 we get
 \begin{align*}
 -{a}_0^{-1}
 =
 L_1^{1-q}
 \beta(\beta+n-2)
 -
 q
 =
 \tfrac{
 \beta(\beta+n-2)
 }{\beta_0(\beta_0+n-2)}
 -q
 >
 1-q.
 \end{align*}
 This concludes \eqref{eq4.19}.
 We next construct $\theta_1(x)$ in the same way.
 We note from \eqref{eq4.18} that
 \begin{equation}\label{eq4.20}
 {\sf U}_\infty^q
 \theta_0
 =
 {a}_0
 {\sf U}_\infty^{p+1}
 \quad
 \Leftrightarrow
 \quad
 {\sf U}_\infty^{-1}
 \theta_0
 =
 {a}_0
 {\sf U}_\infty^{p-q}.
 \end{equation}
 By using this relation,
 we compute the second and the third terms of \eqref{eq4.14}. 
 \begin{align*}
 \sum_{i=1}^N
 \tfrac{f^{(i)}({\sf U}_\infty)}{i!}
 \theta_0^i
 &=
 f'({\sf U}_\infty)
 \theta_0
 +
 \sum_{i=2}^N
 \tfrac{f^{(i)}({\sf U}_\infty)}{i!}
 \theta_0^i
 \\
 &=
 p
 {\sf U}_\infty^{p-1}
 \theta_0
 +
 \sum_{i=2}^N
 c_i
 {\sf U}_\infty^{p-i}
 \theta_0^i
 \\
 &=
 p
 {\sf U}_\infty^{p-1}
 \theta_0
 +
 {\sf U}_\infty^{p-1}
 \theta_0
 \sum_{i=2}^N
 c_i
 ({\sf U}_\infty^{-1}\theta_0)^{i-1}
 \\
 &=
 p
 {\sf U}_\infty^{p-1}
 \theta_0
 +
 {\sf U}_\infty^{p-1}
 \theta_0
 \sum_{i=2}^N
 c_i
 |x|^\frac{2(p-q)(i-1)}{1-q},
 \\
 \sum_{i=2}^N
 \tfrac{f_2^{(i)}({\sf U}_\infty)}{i!}
 \theta_0^i
 &=
 \tfrac{1}{2}
 f_2''({\sf U}_\infty)
 \theta_0^2
 +
 \sum_{i=3}^N
 \tfrac{f_2^{(i)}({\sf U}_\infty)}{i!}
 \theta_0^i
 \\
 &=
 \tfrac{q(q-1)}{2}
 {\sf U}_\infty^{q-2}
 \theta_0^2
 +
 \sum_{i=3}^N
 d_i
 {\sf U}_\infty^{q-i}
 \theta_0^i
 \\
 &=
 \tfrac{q(q-1)}{2}
 {a}_0
 {\sf U}_\infty^{p-1}
 \theta_0
 +
 {\sf U}_\infty^{q-2}
 \theta_0^2
 \sum_{i=3}^N
 d_i
 (
 {\sf U}_\infty^{-1}
 \theta_0
 )^{i-2}
 \\
 &=
 \tfrac{q(q-1)}{2}
 {a}_0
 {\sf U}_\infty^{p-1}
 \theta_0
 +
 {\sf U}_\infty^{p-1}
 \theta_0
 \sum_{i=3}^N
 d_i
 |x|^\frac{2(p-q)(i-2)}{1-q}
 \\
 &=
 \tfrac{q(q-1)}{2}
 {a}_0
 {\sf U}_\infty^{p-1}
 \theta_0
 +
 {\sf U}_\infty^{p-1}
 \theta_0
 \sum_{i=2}^{N-1}
 d_i
 |x|^\frac{2(p-q)(i-1)}{1-q}.
 \end{align*}
 From \eqref{eq4.19},
 we obtain
 \begin{align*}
 \sum_{i=1}^N
 \tfrac{f^{(i)}({\sf U}_\infty)}{i!}
 \theta_0^i
 -
 \sum_{i=2}^N
 \tfrac{f_2^{(i)}({\sf U}_\infty)}{i!}
 \theta_0^i
 &=
 p
 {\sf U}_\infty^{p-1}
 \theta_0
 -
 \tfrac{q(q-1)}{2}
 {a}_0
 {\sf U}_\infty^{p-1}
 \theta_0
 +
 {\sf U}_\infty^{p-1}
 \theta_0
 \sum_{i=2}^{N}
 C_i
 |x|^{\frac{2(p-q)(i-1)}{1-q}}
 \\
 &=
 (
 \underbrace{
 p+\tfrac{q(1-q)}{2}
 {a}_0
 }_{\not=0}
 )
 {\sf U}_\infty^{p-1}
 \theta_0
 +
 {\sf U}_\infty^{p-1}
 \theta_0
 \sum_{i=1}^{N-1}
 C_i
 |x|^\frac{2(p-q)i}{1-q}.
 \end{align*}
 Therefore
 we obtain a solution $\theta_1(x)$ of \eqref{eq4.14} satisfying
 \begin{align}
 \label{eq4.21}
 \theta_1(x)
 &=
 {\sf c}_1|x|^2{\sf U}_\infty^{p-1}\theta_0
 +
 {\sf h}_1(x)
 =
 {a}_1
 |x|^\frac{2(p-q)}{1-q}
 \theta_0
 +
 {\sf h}_1(x)
 \qquad
 ({a}_1\not=0),
 \\
 \nonumber
 {\sf h}_1(x)
 &=
 |x|^\frac{2(p-q)}{1-q}
 \theta_0
 \sum_{i=1}^{N-1}
 b_i|x|^\frac{2(p-q)i}{1-q}.
 \end{align}
 A solution $\theta_2(x)$ of \eqref{eq4.15} is constructed in the same manner.
 We write
 \begin{align*}
 \sum_{i=1}^N
 \tfrac{f^{(i)}({\sf U}_\infty)}{i!}
 ((\theta_0+\theta_1)^i-\theta_0^i)
 &=
 f'({\sf U}_\infty)
 \theta_1
 +
 \sum_{i=2}^N
 \tfrac{f^{(i)}({\sf U}_\infty)}{i!}
 ((\theta_0+\theta_1)^i-\theta_0^i),
 \\
 \sum_{i=2}^N
 \tfrac{f_2^{(i)}({\sf U}_\infty)}{i!}
 ((\theta_0+\theta_1)^i-\theta_0^i)
 &=
 \tfrac{f''({\sf U}_\infty)}{2}
 (2\theta_0\theta_1+\theta_1^2)
 +
 \sum_{i=3}^N
 d_i
 {\sf U}_\infty^{q-i}
 ((\theta_0+\theta_1)^i-\theta_0^i).
 \end{align*}
 From \eqref{eq4.20} and \eqref{eq4.21},
 we see that
 \begin{align*}
 f'({\sf U}_\infty)
 \theta_1
 &=
 p{\sf U}_\infty^{p-1}
 \theta_1
 \\
 &=
 p{\sf U}_\infty^{p-1}
 {a}_1
 |x|^\frac{2(p-q)}{1-q}
 \theta_0
 +
 p{\sf U}_\infty^{p-1}
 {\sf h}_1,
 \\
 \sum_{i=2}^N
 \tfrac{f^{(i)}({\sf U}_\infty)}{i!}
 ((\theta_0+\theta_1)^i-\theta_0^i)
 &=
 \sum_{i=2}^N
 \sum_{l=1}^i
 c_{i,l}
 f^{(i)}({\sf U}_\infty)
 \theta_0^{i-l}\theta_1^l
 \\
 &=
 \sum_{i=2}^N
 \sum_{l=1}^i
 c_{i,l}
 {\sf U}_\infty^{p-i}
 \theta_0^{i-l}\theta_1^l
 \\
 &=
 ({\sf U}_\infty^{p-1}\theta_1)
 ({\sf U}_\infty^{-1}\theta_0)
 \sum_{i=2}^N
 \sum_{l=1}^i
 c_{i,l}
 ({\sf U}_\infty^{-1}\theta_0)^{i-2}
 (\theta_0^{-1}\theta_1)^{l-1},
 \\
 \tfrac{f''({\sf U}_\infty)}{2}
 (2\theta_0\theta_1+\theta_1^2)
 &=
 q(q-1)
 {\sf U}_\infty^{q-2}
 \theta_0
 \theta_1
 +
 c{\sf U}_\infty^{q-2}
 \theta_1^2
 \\
 &=
 q(q-1)
 {a}_0
 {\sf U}_\infty^{p-1}
 \theta_1
 +
 c
 {\sf U}_\infty^{p-1}
 \theta_0^{-1}
 \theta_1^2
 \\
 &=
 q(q-1)
 {a}_0
 {\sf U}_\infty^{p-1}
 {a}_1
 |x|^\frac{2(p-q)}{1-q}
 \theta_0
 +
 q(q-1)
 {a}_0
 {\sf U}_\infty^{p-1}
 {\sf h}_1
 +
 c
 {\sf U}_\infty^{p-1}
 \theta_1
 (\theta_0^{-1}\theta_1),
 \\
 \sum_{i=3}^N
 d_i
 {\sf U}_\infty^{q-i}
 ((\theta_0+\theta_1)^i-\theta_0^i)
 &=
 \sum_{i=3}^N
 \sum_{l=1}^i
 d_{i,l}
 {\sf U}_\infty^{q-i}
 \theta_0^{i-l}\theta_1^l
 \\
 &=
 {\sf U}_\infty^{q-2}
 \theta_0
 \theta_1
 \sum_{i=3}^N
 \sum_{l=1}^i
 d_{i,l}
 {\sf U}_\infty^{2-i}
 \theta_0^{i-l-1}\theta_1^{l-1}
 \\
 &=
 {\sf U}_\infty^{p-1}
 \theta_1
 \sum_{i=3}^N
 \sum_{l=1}^i
 d_{i,l}
 ({\sf U}_\infty^{-1}\theta_0)^{i-2}
 (\theta_0^{-1}\theta_1)^{l-1}
 \\
 &=
 ({\sf U}_\infty^{p-1}\theta_1)
 ({\sf U}_\infty^{-1}\theta_0)
 \sum_{i=3}^N
 \sum_{l=1}^i
 d_{i,l}
 ({\sf U}_\infty^{-1}\theta_0)^{i-3}
 (\theta_0^{-1}\theta_1)^{l-1}.
 \end{align*}
 Furthermore
 we note from
 \eqref{eq4.18} and \eqref{eq4.21}
 that
 $\theta_0^{-1}\theta_1
 =
 |x|^\frac{2(p-q)}{1-q}
 (a_1
 +
 \sum_{i=1}^{N-1}
 b_i|x|^\frac{2(p-q)i}{1-q})$.
 This relation together with \eqref{eq4.20} implies
 \begin{align*}
 \sum_{i=2}^N
 \sum_{l=1}^i
 c_{i,l}
 ({\sf U}_\infty^{-1}\theta_0)^{i-2}
 (\theta_0^{-1}\theta_1)^{l-1}
 &=
 \sum_{i=0}^{N'}
 c_i
 |x|^\frac{2(p-q)i}{1-q},
 \\
 \sum_{i=3}^N
 \sum_{l=1}^i
 d_{i,l}
 ({\sf U}_\infty^{-1}\theta_0)^{i-3}
 (\theta_0^{-1}\theta_1)^{l-1}
 &=
 \sum_{i=0}^{N'}
 d_i
 |x|^\frac{2(p-q)i}{1-q}.
 \end{align*}
 Combining all computations above,
 we obtain
 \begin{align*}
 \sum_{i=1}^N
 \tfrac{f^{(i)}({\sf U}_\infty)}{i!}
 &
 ((\theta_0+\theta_1)^i-\theta_0^i)
 -
 \sum_{i=2}^N
 \tfrac{f_2^{(i)}({\sf U}_\infty)}{i!}
 ((\theta_0+\theta_1)^i-\theta_0^i)
 \\
 &=
 {a}_1
 (
 p-q(q-1){a}_0
 )
 {\sf U}_\infty^{p-1}
 |x|^\frac{2(p-q)}{1-q}
 \theta_0
 +
 {\sf U}_\infty^{p-1}
 |x|^\frac{2(p-q)}{1-q}
 \theta_0
 \sum_{i=1}^{N'}
 c_i|x|^\frac{2(p-q)}{1-q}.
 \end{align*}
 Therefore
 since $p-q(q-1){a}_0>0$ (see \eqref{eq4.19}),
 we conclude
 \begin{align*}
 \theta_2(x)
 &=
 {\sf c}_2
 |x|^2
 {\sf U}_\infty^{p-1}
 |x|^\frac{2(p-q)}{1-q}
 \theta_0
 +
 {\sf h}_2(x)
 =
 {a}_2
 |x|^\frac{4(p-q)}{1-q}
 \theta_0
 +
 {\sf h}_2(x)
 \qquad
 ({a}_2\not=0),
 \\
 {\sf h}_2(x)
 &=
 |x|^\frac{4(p-q)}{1-q}
 \theta_0
 \sum_{i=1}^{N'}
 b_i
 |x|^\frac{2(p-q)i}{1-q}.
 \end{align*}
 In the same manner,
 we can construct solutions $\theta_k(x)$ ($k=3,4,\cdots,L$) satisfying
 \begin{align*}
 \theta_k(x)
 &=
 {a}_k
 |x|^\frac{2k(p-q)}{1-q}
 \theta_0
 +
 {\sf h}_k(x)
 \qquad
 ({a}_k\not=0),
 \\
 {\sf h}_k(x)
 &=
 |x|^\frac{2k(p-q)}{1-q}
 \theta_0
 \sum_{i=1}^{N'}
 b_i
 |x|^\frac{2(p-q)i}{1-q}.
 \end{align*}

\section{Formulation}
\label{sec_5}
 In this section,
 we set up our problem.
 We recall that
 two types of blowup solutions are predicted 
 in Section \ref{sec_4.2} - Section \ref{sec_4.4}.
 However,
 we here discuss the case (II) only,
 since the case (I) is easier.
 Our strategy in this paper is basically the same as in \cite{Harada_2}
 (the original idea comes from \cite{Cortazar,del_Pino}).
 Let ${\sf M}(t)$ be a unique solution of
 \[
 \begin{cases}
 \frac{d}{dt}{\sf M}=f({\sf M})-f_2({\sf M})
 \qquad
 \text{for }
 t>0,
 \\
 {\sf M}(t)|_{t=0}
 =
 {\sf M}_0,
 \qquad
 \text{where }
 {\sf M}_0={\sf U}_\infty(x)|_{|x|=1}.
 \end{cases}
 \]

 \subsection{Setting}
 \label{sec_5.1}
 From the observation in Section \ref{sec_4},
 we look for solutions of the form
 \begin{align}
 \label{eq5.1}
 u(x,t)
 &=
 \dis
 \lambda^{-\frac{n-2}{2}}
 {\sf Q}(y)
 \chi_2
 +
 \lambda^{-\frac{n-2}{2}}\sigma
 T_1(y)
 \chi_1
 -
 {\sf U}_{\sf c}(x,t)
 (1-\chi_1)
 \\
 & \quad
 -
 (
 \theta(x)
 +
 \Theta_J(x,t)
 )
 (1-\chi_2)
 \chi_3
 +
 u_1(x,t),
 \nonumber
 \end{align}
 where
 \begin{align*}
 {\sf U}_{\sf c}(x,t)
 &=
 \eta^\frac{2}{1-q}{\sf U}(\xi)
 \chi_2
 +
 {\sf U}_\infty(x)
 (1-\chi_2)
 \chi_4
 +
 {\sf M}(t)(1-\chi_4),
 \\
 \theta(x)
 &=
 \theta_0(x)
 +
 \cdots
 +
 \theta_L(x)
 \qquad
 (\text{see Section \ref{sec_4.5}}),
 \\
 \Theta_J(x,t)
 &=
 {\sf D}_J^{-1}
 {\sf B}_1
 (T-t)^{\frac{\gamma}{2}+J}
 e_J(z)
 \qquad
 (\text{see Section \ref{sec_4.2} and Section \ref{sec_4.4}}),
 \\
 u_1(x,t)
 &=
 \lambda(t)^{-\frac{n-2}{2}}
 \epsilon(y,t)
 \chi_{\sf in}
 +
 \eta^\frac{2}{1-q}
 v(\xi,t)
 \chi_{\sf mid}
 +
 w(x,t).
 \end{align*}
 The function $\lambda(t)$ is unknown at this moment which is determined
 in Section \ref{sec_6.1}.
 Other parameters $\eta(t)$ and $\sigma(t)$ are defined by
 \begin{itemize}
 \item
 $\eta(t)=(T-t)^{\gamma_J}$
 \quad
 ($J\in\N$)
 \quad
 with
 $\gamma_J=J(\frac{2}{1-q}-\gamma)^{-1}$,
 \item
 $\sigma(t)=
 -{\sf A}_1^{-1}
 \eta(t)^\frac{2}{1-q}
 \lambda(t)^{\frac{n-2}{2}}$.
 \end{itemize}
 New variables $y$, $\xi$ and $z$ are defined by
 \begin{itemize}
 \item $y=\lambda(t)^{-1}x$,
 \item $\xi=\eta(t)^{-1}x$,
 \item $z=(T-t)^{-\frac{1}{2}}x$.
 \end{itemize}
 Let $\chi(r)$ be a smooth cut off function satisfying
 $\chi(r)
 =
 \begin{cases}
 1 & \text{for } 0\leq r<1 \\
 0 & \text{for } r>2
 \end{cases}$
 and put
 \begin{itemize}
 \item
 $\dis\chi_{\sf in}=\chi(|y|/{\sf R}_{\sf in})$,
 where ${\sf R}_{\sf in}>0$ is a large constant,
 \item
 $\chi_{\sf mid}=\chi(|\xi|/{\sf R}_{\sf mid})$,
 where ${\sf R}_{\sf mid}>0$ is a large constant,
 \item
 $\dis\chi_0=\chi(|z|/{\sf r}_0)$,
 where ${\sf r}_0>0$ is a small constant,
 \item
 $\dis\chi_1=\chi(|y|/{\sf l}_1(t))$,
 where ${\sf l}_1(t)=|\sigma(t)|^{-\frac{1}{n-2}}$,
 \item
 $\dis\chi_2=\chi(|\xi|/{\sf l}_2(t))$,
 where ${\sf l}_2(t)=(T-t)^{-{\sf b}}$
 and
 ${\sf b}>0$ is a small constant,
 \item
 $\chi_3=\chi(|x|/{\sf r}_3)$,
 where ${\sf r}_3>0$ is a small constant,
 \item
 $\chi_4=\chi(|x|)$.
 \end{itemize}
 Constants
 ${\sf R}_{\sf in},{\sf R}_{\sf mid},{\sf r}_0,{\sf r}_3,{\sf b}$
 will be chosen
 in Section \ref{sec_6} - Section \ref{sec_8},
 which depend only on $q,n,J$.
 We remark that
 we cut off $\theta(x)$, $\Theta(x,t)$ and ${\sf U}_\infty(x)$
 by using $\chi_3,\chi_4$
 in \eqref{eq5.1},
 since they are not bounded for $|x|\to\infty$.

 \subsection{Derivation of equations for $\epsilon(y,t)$, $v(\xi,t)$, $w(x,t)$}
 \label{sec_5.2}
 We here derive
 equations for $\epsilon(y,t)$, $v(\xi,t)$ and $w(x,t)$.
 To do that,
 we substitute \eqref{eq5.1} into \eqref{eq1.1}
 and compute them.
 We first give computations for $u_t$.
 \begin{align*}
 \pa_t
 (
 \lambda^{-\frac{n-2}{2}}
 {\sf Q}
 \chi_2
 )
 &=
 -\lambda^{-\frac{n+2}{2}}
 \lambda
 \dot\lambda
 (\Lambda_y{\sf Q})
 \chi_2
 +
 \lambda^{-\frac{n-2}{2}}
 {\sf Q}
 \dot\chi_2
 \\
 &=
 -\lambda^{-\frac{n+2}{2}}
 \lambda
 \dot\lambda
 (\Lambda_y{\sf Q})
 \chi_1
 -
 \underbrace{
 \lambda^{-\frac{n+2}{2}}
 \lambda
 \dot\lambda
 (\Lambda_y{\sf Q})
 (\chi_2-\chi_1)
 }_{=g_0}
 +
 \underbrace{
 \lambda^{-\frac{n-2}{2}}
 {\sf Q}
 \dot\chi_2
 }_{=g_1},
 \\
 \pa_t
 (
 \lambda^{-\frac{n-2}{2}}\sigma
 T_1
 \chi_1
 )
 &=
 \underbrace{
 (
 \pa_t(
 \lambda^{-\frac{n-2}{2}}\sigma
 T_1
 )
 )
 \chi_1
 }_{=g_2}
 +
 \lambda^{-\frac{n-2}{2}}\sigma
 T_1
 \dot\chi_1,
 \\
 \pa_t
 ({\sf U}_{\sf c}(x,t)(1-\chi_1))
 &=
 (
 \pa_t{\sf U}_{\sf c}
 )
 (1-\chi_1)
 -
 {\sf U}_{\sf c}
 \dot\chi_1
 \\
 &=
 \{
 \underbrace{
 (
 \pa_t
 (
 \eta^\frac{2}{1-q}
 {\sf U}
 )
 )
 \chi_2
 }_{=g_3}
 +
 \eta^\frac{2}{1-q}
 {\sf U}
 \dot\chi_2
 -
 {\sf U}_\infty
 \dot\chi_2
 \chi_4
 \}
 (1-\chi_1)
 \\
 & \quad
 +
 \dot {\sf M}(t)
 (1-\chi_4)
 (1-\chi_1)
 -
 {\sf U}_{\sf c}
 \dot\chi_1
 \\
 &=
 g_3
 +
 \eta^\frac{2}{1-q}
 {\sf U}
 \dot\chi_2
 -
 {\sf U}_\infty
 \dot\chi_2
 +
 \dot {\sf M}(t)
 (1-\chi_4)
 -
 {\sf U}_{\sf c}
 \dot\chi_1,
 \\
 \pa_t(\theta(1-\chi_2)\chi_3)
 &=
 -
 \underbrace{
 \theta
 \dot\chi_2
 \chi_3
 }_{=g_4},
 \\
 \pa_t(\Theta_J(1-\chi_2)\chi_3)
 &=
 \dot\Theta_J
 (1-\chi_2)
 \chi_3
 -
 \Theta_J
 \dot\chi_2
 \chi_3,
 \\
 \pa_tu_1
 &=
 \lambda^{-\frac{n-2}{2}}
 \epsilon_t\chi_{\sf in}
 -
 \underbrace{
 \lambda^{-\frac{n+2}{2}}
 \lambda\dot\lambda
 (\Lambda_y\epsilon)\chi_{\sf in}
 +
 \lambda^{-\frac{n-2}{2}}
 \epsilon
 \dot\chi_{\sf in}
 }_{=h_1[\epsilon]}
 +
 \eta^\frac{2}{1-q}
 v_t
 \chi_{\sf mid}
 \\
 & \qquad
 +
 \eta^{\frac{2}{1-q}-1}
 \dot\eta
 (\Lambda_\xi v)
 \chi_{\sf mid}
 +
 \underbrace{
 \eta^\frac{2}{1-q}
 v
 \dot\chi_{\sf mid}
 }_{=k_1[v]}
 +
 w_t,
 \end{align*}
 where $\Lambda_\xi=\frac{2}{1-q}-\xi\cdot\nabla_\xi$.
 Therefore
 we get
 \begin{align*}
 u_t
 &=
 -
 \lambda^{-\frac{n+2}{2}}
 \lambda\dot\lambda
 (\Lambda_y{\sf Q})
 \chi_1
 +
 \underbrace{
 (
 \lambda^{-\frac{n-2}{2}}\sigma T_1
 +
 {\sf U}_{\sf c}
 )
 \dot\chi_1
 }_{=g_5}
 +
 \underbrace{
 (
 -
 \eta^\frac{2}{1-q}{\sf U}
 +
 {\sf U}_\infty
 \chi_4
 +
 \Theta_J
 \chi_3
 )
 \dot\chi_2
 }_{=g_6}
 \\
 &\quad
 -
 \dot{\sf M}
 (1-\chi_4)
 -
 \dot\Theta_J
 (1-\chi_2)\chi_3
 +
 \lambda^{-\frac{n-2}{2}}
 \epsilon_t\chi_{\sf in}
 +
 \eta^\frac{2}{1-q}
 v_t
 \chi_{\sf mid}
 \\
 & \quad
 +
 \eta^{\frac{2}{1-q}-1}
 \dot\eta
 (\Lambda_\xi v)
 \chi_{\sf mid}
 +
 w_t
 +
 g_0+g_1+g_2+g_3+g_4
 +
 h_1+k_1
 \\
 &=
 \lambda^{-\frac{n-2}{2}}
 \epsilon_t\chi_{\sf in}
 +
 \eta^\frac{2}{1-q}
 v_t
 \chi_{\sf mid}
 +
 \eta^{\frac{2}{1-q}-1}
 \dot\eta
 (\Lambda_\xi v)
 \chi_{\sf mid}
 +
 w_t
 \\
 & \quad
 -
 \lambda^{-\frac{n+2}{2}}
 \lambda\dot\lambda
 (\Lambda_y{\sf Q})
 \chi_1
 -
 \dot\Theta_J
 (1-\chi_2)
 \chi_3
 -
 \dot{\sf M}
 (1-\chi_4)
 +
 g_{0\sim6}
 +
 h_1
 +
 k_1.
 \end{align*}
 Next
 we compute $\Delta_xu$.
 The first two terms in \eqref{eq5.1} are estimated as
 \begin{align*}
 \Delta_x
 &(
 \lambda^{-\frac{n-2}{2}}{\sf Q}(y)
 \chi_2
 +
 \lambda^{-\frac{n-2}{2}}\sigma
 T_1(y)\chi_1
 )
 \\
 &=
 \lambda^{-\frac{n+2}{2}}
 (\Delta_y{\sf Q})
 \chi_2
 +
 \underbrace{
 \eta^{-1}
 \lambda^{-\frac{n}{2}}
 (\nabla_y{\sf Q}\cdot\nabla_\xi\chi_2)
 +
 \eta^{-2}
 \lambda^{-\frac{n-2}{2}}
 {\sf Q}
 (\Delta_\xi\chi_2)
 }_{=g_0'}
 \\
 & \quad
 +
 \lambda^{-\frac{n+2}{2}}
 \sigma
 (\Delta_yT_1)
 \chi_1
 +
 \underbrace{
 2\lambda^{-\frac{n+2}{2}}
 \sigma
 \nabla_yT_1
 \cdot\nabla_y\chi_1
 }_{=g_1'}
 +
 \lambda^{-\frac{n+2}{2}}
 \sigma
 T_1(\Delta_y\chi_1).
 \end{align*}
 The rest of terms are computed as
 \begin{align*}
 \Delta_x
 {\sf U}_{\sf c}
 &=
 \Delta_x
 \{
 \eta^\frac{2}{1-q}{\sf U}
 \chi_2
 +
 {\sf U}_\infty
 (1-\chi_2)
 \chi_4
 +
 {\sf M}(t)(1-\chi_4)
 \}
 \\
 &=
 \eta^\frac{2q}{1-q}
 (\Delta_\xi{\sf U})
 \chi_2
 +
 (\Delta_x{\sf U}_\infty)
 (1-\chi_2)
 \chi_4
 \\
 & \quad
 +
 2\eta^{-1}
 \nabla_x(
 \eta^\frac{2}{1-q}{\sf U}
 -
 {\sf U}_\infty
 \chi_4
 )
 \cdot\nabla_\xi\chi_2
 \\
 & \quad
 +
 \eta^{-2}
 (
 \eta^\frac{2}{1-q}
 {\sf U}
 -
 {\sf U}_\infty
 \chi_4
 )
 (\Delta_\xi\chi_2)
 +
 \underbrace{
 2
 (\nabla_x{\sf U}_\infty\cdot\nabla_x\chi_4)
 (1-\chi_2)
 }_{=g_{{\sf out},1}'}
 \\
 & \quad
 -
 \underbrace{
 2{\sf U}_\infty
 (\nabla_x\chi_2\cdot\nabla_x\chi_4)
 }_{=0}
 +
 \underbrace{
 {\sf U}_\infty
 (1-\chi_2)
 (\Delta_x\chi_4)
 }_{=g_{{\sf out},2}'}
 -
 \underbrace{
 {\sf M}(t)
 \Delta_x\chi_4
 }_{=g_{{\sf out},3}'},
 \end{align*}
 \begin{align*}
 \Delta_x
 \{
 {\sf U}_{\sf c}
 (1-\chi_1)
 \}
 &=
 (\Delta_x{\sf U}_{\sf c})
 (1-\chi_1)
 -
 \underbrace{
 2\lambda^{-1}
 \nabla_x{\sf U}_{\sf c}\cdot\nabla_y\chi_1
 }_{=g_2'}
 -
 \lambda^{-2}
 {\sf U}_{\sf c}
 (\Delta_y\chi_1)
 \\
 &=
 \eta^\frac{2q}{1-q}
 (\Delta_\xi{\sf U})
 (1-\chi_1)
 \chi_2
 +
 \underbrace{
 (\Delta_x{\sf U}_\infty)
 (1-\chi_1)
 (1-\chi_2)
 \chi_4
 }_{=(\Delta_x{\sf U}_\infty)(1-\chi_2)\chi_4}
 \\
 & \quad
 +
 2\eta^{-1}
 \nabla_x(
 \eta^\frac{2}{1-q}{\sf U}
 -
 {\sf U}_\infty
 \chi_4
 )
 \cdot\nabla_\xi\chi_2
 +
 \eta^{-2}
 (
 \eta^\frac{2}{1-q}
 {\sf U}
 -
 {\sf U}_\infty
 \chi_4
 )
 (\Delta_\xi\chi_2)
 \\
 & \quad
 -
 \lambda^{-2}
 {\sf U}_{\sf c}
 (\Delta_y\chi_1)
 +
 g_2'
 +
 g_{{\sf out},1}'+g_{{\sf out},2}'+g_{{\sf out},3}',
 \\
 \Delta_x
 \{
 \theta
 (1-\chi_2)
 \chi_3
 \}
 &=
 (\Delta_x\theta)
 (1-\chi_2)
 \chi_3
 -
 \underbrace{
 2
 (\nabla_x\theta\cdot\nabla_x\chi_2)
 \chi_3
 }_{=g_3'}
 -
 \underbrace{
 \theta(\Delta_x\chi_2)
 \chi_3
 }_{=g_4'}
 \\
 & \quad
 +
 \underbrace{
 2
 (\nabla_x\theta\cdot\nabla_x\chi_3)
 (1-\chi_2)
 }_{=g_{{\sf out},4}'}
 -
 \underbrace{
 2\theta
 (\nabla_x\chi_2\cdot\nabla_x\chi_3)
 }_{=0}
 +
 \underbrace{
 \theta
 (1-\chi_2)
 (\Delta_x\chi_3)
 }_{=g_{{\sf out},5}'},
 \\
 \Delta_x
 \{
 \Theta_J
 (1-\chi_2)
 \chi_3
 \}
 &=
 (\Delta_x\Theta_J)
 (1-\chi_2)
 \chi_3
 -
 \underbrace{
 2\eta^{-1}
 (\nabla_x\Theta_J\cdot\nabla_\xi\chi_2)
 \chi_3
 }_{=
 2\eta^{-1}
 \nabla_x(\Theta_J\chi_3)\cdot\nabla_\xi\chi_2
 }
 -
 \eta^{-2}
 \Theta_J(\Delta_\xi\chi_2)
 \chi_3
 \\
 & \quad
 +
 \underbrace{
 2
 (\nabla_x\Theta_J\cdot\nabla_x\chi_3)
 (1-\chi_2)
 }_{=g_{{\sf out},6}'}
 -
 \underbrace{
 2\Theta_J
 (\nabla_x\chi_2\cdot\nabla_x\chi_3)
 }_{=0}
 +
 \underbrace{
 \Theta_J
 (1-\chi_2)
 (\Delta_x\chi_3)
 }_{=g_{{\sf out},7}'}.
 \end{align*}
 Combining these estimates,
 we get
 \begin{align*}
 \Delta_x
 &
 (u-u_1)
 =
 \Delta_x
 (
 \lambda^{-\frac{n-2}{2}}{\sf Q}(y)
 \chi_2
 +
 \lambda^{-\frac{n-2}{2}}\sigma
 T_1(y)\chi_1
 -
 {\sf U}_{\sf c}
 (1-\chi_1)
 -
 (\Theta+\theta_J)
 (1-\chi_2)
 \chi_3
 )
 \\
 &=
 \lambda^{-\frac{n+2}{2}}
 (\Delta_y{\sf Q})
 \chi_2
 +
 \lambda^{-\frac{n+2}{2}}
 \sigma
 (\Delta_yT_1)
 \chi_1
 +
 \underbrace{
 \lambda^{-2}
 (
 \lambda^{-\frac{n-2}{2}}
 \sigma T_1
 +
 {\sf U}_{\sf c}
 )
 (\Delta_y\chi_1)
 }_{=g_5'}
 \\
 & \quad
 -\eta^\frac{2q}{1-q}
 (\Delta_\xi{\sf U})
 (1-\chi_1)
 \chi_2
 -
 (\Delta_x{\sf U}_\infty)
 (1-\chi_2)
 \chi_4
 -
 (\Delta_x\theta)
 (1-\chi_2)
 \chi_3
 -
 (\Delta_x\Theta_J)
 (1-\chi_2)
 \chi_3
 \\
 & \quad
 -
 \underbrace{
 2\eta^{-1}
 \nabla_x(
 \eta^\frac{2}{1-q}{\sf U}
 -
 {\sf U}_\infty
 \chi_4
 -
 \Theta_J
 \chi_3
 )
 \cdot\nabla_\xi\chi_2
 }_{=g_6'}
 -
 \underbrace{
 \eta^{-2}
 (
 \eta^\frac{2}{1-q}
 {\sf U}
 -
 {\sf U}_\infty
 \chi_4
 -
 \Theta_J
 \chi_3
 )
 (\Delta_\xi\chi_2)
 }_{=g_7'}
 \\
 & \quad
 +
 g_0'+\cdots+g_4'
 +
 g_{{\sf out},1}'+\cdots+g_{{\sf out},7}'
 \\
 &=
 -
 \underbrace{
 f(\lambda^{-\frac{n-2}{2}}{\sf Q})
 \chi_2
 }_{={\sf t}_1}
 -
 \lambda^{-\frac{n+2}{2}}
 \sigma
 (VT_1)
 \chi_1
 +
 \lambda^{-\frac{n+2}{2}}
 \sigma
 (H_yT_1)
 \chi_1
 -
 \underbrace{
 f_2(\eta^\frac{2}{1-q}{\sf U})
 (1-\chi_1)
 \chi_2
 }_{={\sf t}_2}
 \\
 & \quad
 -
 \underbrace{
 \{
 f_2
 ({\sf U}_\infty)
 +
 q{\sf U}_\infty^{q-1}
 (\theta+\Theta_J)
 \chi_3
 +
 (\Delta_x\theta-q{\sf U}_\infty^{q-1}\theta)
 \chi_3
 \}
 (1-\chi_2)
 \chi_4
 }_{={\sf t}_3}
 \\
 & \quad
 -
 (\Delta_x\Theta_J-q{\sf U}_\infty^{q-1}\Theta_J)
 (1-\chi_2)
 \chi_3
 +
 g_0'+\cdots+g_7'
 +
 g_{{\sf out},1}'+\cdots+g_{{\sf out},7}'.
 \end{align*}
 We write ${\sf t}_i$ ($i=1,2,3$) as
 \begin{align*}
 {\sf t}_1
 &=
 f(\lambda^{-\frac{n-2}{2}}{\sf Q})
 \chi_2
 \\
 &=
 -
 \underbrace{
 \{
 f(u)
 -
 f(\lambda^{-\frac{n-2}{2}}{\sf Q})
 -
 \lambda^{-2}
 V
 (
 u
 -
 \lambda^{-\frac{n-2}{2}}
 {\sf Q}
 )
 \}
 \chi_2
 }_{={\sf N}_1[\epsilon,v,w]}
 +
 f(u)
 \chi_{2}
 -
 \lambda^{-2}
 V
 (
 u
 -
 \lambda^{-\frac{n-2}{2}}
 {\sf Q}
 )
 \chi_{2}
 \\
 &=
 f(u)
 \chi_{2}
 -
 \lambda^{-2}
 V
 (
 \lambda^{-\frac{n-2}{2}}
 \sigma
 T_1
 \chi_1
 +
 u_1 
 )
 \chi_{2}
 -
 \underbrace{
 \lambda^{-2}
 V
 (
 {\sf U}_{\sf c}
 (1-\chi_1)
 +
 (\theta+\Theta_J)
 (1-\chi_2)
 )
 \chi_{2}
 }_{=g_8'}
 +
 N_1,
 \\
 {\sf t}_2
 &=
 f_2
 (\eta^\frac{2}{1-q}{\sf U})
 (1-\chi_1)
 \chi_2
 \\
 &=
 \underbrace{
 \{
 f_2(u)
 +
 f_2(\eta^\frac{2}{1-q}{\sf U})
 -
 q
 \eta^{-2}
 {\sf U}^{q-1}
 (u+\eta^\frac{2}{1-q}{\sf U})
 \}
 (1-\chi_{1})
 \chi_{2}
 }_{={\sf N}_2[\epsilon,v,w]}
 \\
 & \quad
 -
 f_2(u)
 (1-\chi_{1})
 \chi_{2}
 +
 \underbrace{
 q
 \eta^{-2}
 {\sf U}^{q-1}
 (u+\eta^\frac{2}{1-q}{\sf U}-u_1)
 (1-\chi_{1})
 \chi_{2}
 }_{=g_9'}
 +
 q
 \eta^{-2}
 {\sf U}^{q-1}
 u_1
 (1-\chi_{1})
 \chi_{2},
 \\
 {\sf t}_3
 &=
 \{
 f_2
 ({\sf U}_\infty)
 +
 q{\sf U}_\infty^{q-1}
 (\theta+\Theta_J)
 \chi_3
 +
 (\Delta_x\theta-q{\sf U}_\infty^{q-1}\theta)
 \chi_3
 \}
 (1-\chi_2)
 \chi_4
 \\
 &=
 -
 \underbrace{
 \{
 f(u)
 -
 f_2(u)
 -
 f_2({\sf U}_\infty)
 +
 q{\sf U}_\infty^{q-1}
 (u+{\sf U}_\infty)
 -
 (\Delta_x\theta-q{\sf U}_\infty^{q-1}\theta)
 \chi_3
 \}
 (1-\chi_{2})
 \chi_4
 }_{=N_3[\epsilon,v,w]}
 \\
 & \quad
 +
 \{
 f(u)
 -
 f_2(u)
 \}
 (1-\chi_{2})
 \chi_4
 +
 \underbrace{
 q
 {\sf U}_\infty^{q-1}
 (u+{\sf U}_\infty+(\theta+\Theta_J)\chi_3-u_1)
 (1-\chi_{2})
 \chi_4
 }_{=g_{10}'}
 \\
 & \quad
 +
 q{\sf U}_\infty^{q-1}
 u_1
 (1-\chi_{2})
 \chi_4.
 \end{align*}
 This implies
 \begin{align*}
 {\sf t}_1
 +
 {\sf t}_2
 +
 {\sf t}_3
 &=
 f(u)
 \chi_{2}
 -
 \lambda^{-2}
 V
 (
 \lambda^{-\frac{n-2}{2}}
 \sigma
 T_1
 \chi_1
 +
 u_1 
 )
 \chi_{2}
 \\
 & \quad
 -
 f_2(u)
 (1-\chi_{1})
 \chi_{2}
 +
 q
 \eta^{-2}
 {\sf U}^{q-1}
 u_1
 (1-\chi_{1})
 \chi_{2},
 \\
 & \quad
 +
 \{
 f(u)
 -
 f_2(u)
 \}
 (1-\chi_{2})
 \chi_4
 +
 q{\sf U}_\infty^{q-1}
 u_1
 (1-\chi_{2})
 \chi_4
 \\
 & \quad
 +
 g_8'+g_9'+g_{10}'
 +
 N_1+N_2+N_3
 \\
 &=
 f(u)
 \chi_4
 -
 f_2(u)
 \chi_4
 +
 \underbrace{
 f_2(u)
 \chi_1
 }_{=N_4[\epsilon,v,w]}
 -
 \lambda^{-2}
 V
 \lambda^{-\frac{n-2}{2}}
 \sigma
 T_1
 \chi_1
 \\
 & \quad
 -
 \lambda^{-2}
 V
 u_1
 \chi_2
 +
 q
 \eta^{-2}
 {\sf U}^{q-1}
 u_1
 (1-\chi_{1})
 \chi_{2}
 +
 q{\sf U}_\infty^{q-1}
 u_1
 (1-\chi_{2})
 \chi_4
 \\
 & \quad
 +
 g_8'+g_9'+g_{10}'
 +
 N_1
 +
 N_2
 +
 N_3.
 \end{align*}
 Therefore
 it follows that
 \begin{align*}
 \Delta_x
 (u-u_1)
 &=
 -{\sf t}_1
 -
 \lambda^{-\frac{n+2}{2}}
 \sigma
 (VT_1)
 \chi_1
 +
 \lambda^{-\frac{n+2}{2}}
 \sigma
 (H_yT_1)
 \chi_1
 -
 {\sf t}_2
 -
 {\sf t}_3
 \\
 & \quad
 -
 (\Delta_x\Theta_J-q{\sf U}_\infty^{q-1}\Theta_J)
 (1-\chi_2)
 \chi_3
 +
 g_{0\sim7}'
 +
 g_{{\sf out},1\sim7}'
 \\
 &=
 \lambda^{-\frac{n+2}{2}}
 \sigma
 (H_yT_1)
 \chi_1
 -
 (\Delta_x\Theta_J-q{\sf U}_\infty^{q-1}\Theta_J)
 (1-\chi_2)
 \chi_3
 \\
 & \quad
 +
 \lambda^{-2}
 V
 u_1
 \chi_2
 -
 q
 \eta^{-2}
 {\sf U}^{q-1}
 u_1
 (1-\chi_{1})
 \chi_{2}
 -
 q{\sf U}_\infty^{q-1}
 u_1
 (1-\chi_{2})
 \chi_4
 \\
 & \quad
 -
 f(u)
 \chi_4
 +
 f_2(u)
 \chi_4
 +
 g_{0\sim10}'
 +
 g_{{\sf out},1\sim7}'
 +
 N_{1\sim4}.
 \end{align*}
 Furthermore
 we verify that
 \begin{align*}
 \Delta_xu_1
 &=
 \lambda^{-\frac{n+2}{2}}(\Delta_y\epsilon)
 \chi_{\sf in}
 +
 \underbrace{
 2\lambda^{-\frac{n+2}{2}}
 \nabla_y\epsilon\cdot\nabla_y\chi_{\sf in}
 +
 \lambda^{-\frac{n+2}{2}}
 \epsilon\Delta_y\chi_{\sf in}}_{=h_1'[\epsilon]}
 \\
 & \quad
 +
 \eta^\frac{2q}{1-q}
 (\Delta_\xi v)
 \chi_{\sf mid}
 +
 \underbrace{
 2\eta^\frac{2q}{1-q}
 (\nabla_\xi v\cdot\nabla_\xi\chi_{\sf mid})
 +
 \eta^\frac{2q}{1-q}
 v(\Delta_\xi\chi_{\sf mid})
 }_{=k_1'[v]}
 +
 \Delta_xw
 \\
 &=
 \lambda^{-\frac{n+2}{2}}(H_y\epsilon)
 \chi_{\sf in}
 +
 \eta^\frac{2q}{1-q}
 (\Delta_\xi v-q{\sf U}^{q-1}v)
 \chi_{\sf mid}
 +
 (\Delta_xw-q{\sf U}_\infty^{q-1}w)
 \\
 & 
 \quad
 -
 \lambda^{-\frac{n+2}{2}}V\epsilon
 +
 \eta^\frac{2q}{1-q}
 q{\sf U}^{q-1}
 v
 \chi_{\sf mid}
 +
 q{\sf U}_\infty^{q-1}
 w
 +
 h_1'[v]
 +
 k_1'[v].
 \end{align*}
 As a consequence,
 we obtain
 \begin{align*}
 \Delta_xu
 &=
 \Delta_x(u-u_1)
 +
 \Delta_xu_1
 \\
 &=
 \lambda^{-\frac{n+2}{2}}
 \sigma
 (H_yT_1)
 \chi_1
 -
 (\Delta_x\Theta_J-q{\sf U}_\infty^{q-1}\Theta_J)
 (1-\chi_2)
 \chi_3
 \\
 & \quad
 +
 \lambda^{-\frac{n+2}{2}}(H_y\epsilon)\chi_{\sf in}
 +
 \eta^\frac{2q}{1-q}
 (\Delta_\xi v-q{\sf U}^{q-1}(1-\chi_1)v)
 \chi_{\sf mid}
 +
 (\Delta_xw-q{\sf U}_\infty^{q-1}w)
 \\
 & \quad
 +
 \lambda^{-2}
 V(\eta^\frac{2}{1-q}v\chi_{{\sf mid}}+w)\chi_2
 -
 q
 (
 \eta^{-2}
 {\sf U}^{q-1}
 (1-\chi_1)
 -
 {\sf U}_\infty^{q-1}
 )
 w
 \chi_2
 \\
 & \quad
 +
 \underbrace{
 q{\sf U}_\infty^{q-1}
 w
 (1-\chi_4)
 }_{=N_5[w]}
 -
 f(u)
 \chi_4
 +
 f_2(u)
 \chi_4
 +
 g_{0\sim10}'
 +
 g_{{\sf out},1\sim7}'
 +
 N_{1\sim4}.
 \end{align*}
 Since $HT_1=-\Lambda_y{\sf Q}$,
 we conclude
 \begin{align*}
 \lambda^{-\frac{n-2}{2}}
 &
 \epsilon_t
 \chi_{\sf in}
 +
 \eta^\frac{2}{1-q}
 v_t
 \chi_{\sf mid}
 +
 w_t
 =
 \lambda^{-\frac{n+2}{2}}
 (\lambda\dot\lambda-\sigma)
 (\Lambda_y{\sf Q})
 \chi_1
 +
 \lambda^{-\frac{n+2}{2}}
 (H_y\epsilon)
 \chi_{\sf in}
 \\
 &
 +
 \eta^\frac{2q}{1-q}
 (\Delta_\xi v-q{\sf U}^{q-1}(1-\chi_1)v)
 \chi_{\sf mid}
 -
 \eta^{\frac{2}{1-q}-1}
 \dot\eta
 (\Lambda_\xi v)
 \chi_{\sf mid}
 +
 (\Delta_xw-q{\sf U}_\infty^{q-1}w)
 \\
 &
 +
 \underbrace{
 \lambda^{-2}
 V(\eta^\frac{2}{1-q}v\chi_{{\sf mid}}+w)
 \chi_2
 }_{=F_1[v,w]}
 -
 \underbrace{
 q
 (
 \eta^{-2}
 {\sf U}^{q-1}
 (1-\chi_1)
 -
 {\sf U}_\infty^{q-1}
 )
 w
 \chi_2
 }_{=F_2[w]}
 \\
 &
 +
 \underbrace{
 \{
 \dot{\sf M}
 +
 f(u)
 -
 f_2(u)
 \}
 (1-\chi_4)
 }_{=N_6[w]}
 -
 g_{0\sim6}
 +
 g_{0\sim10}'
 +
 g_{{\sf out},1\sim7}'
 +
 N_{1\sim5}
 -
 h_1
 +
 h_1'
 +
 k_1
 -
 k_1'.
 \end{align*}
 For simplicity,
 we write
 \begin{itemize}
 \item
 $g=
 g_{0\sim6}
 +
 g_{0\sim10}'
 +
 g_{{\sf out},1\sim7}'$,
 \item
 $N=N_{1\sim6}$,
 \item
 $h[\epsilon]=-h_1[\epsilon]+h_1'[\epsilon]$,
 \item
 $k[v]=-k_1[v]+k_1'[v]$.
 \end{itemize}
 We introduce the other cut off function.
 \[
 \chi_{{\sf sq}}
 =
 \chi(|\xi|/\sqrt{{\sf R}_{\sf  mid}}).
 \]
 We now decompose this equation into three equations.
 \begin{align}\label{eq5.2}
 \begin{cases}
 \dis
 \lambda^2\epsilon_t
 =
 H_y\epsilon
 +
 (
 \lambda\dot\lambda
 -
 \sigma
 )
 \Lambda_y{\sf Q}
 +
 \lambda^\frac{n+2}{2}
 F_1[v,w]
 &
 \text{in } |y|<{\sf R}_{\sf in},\
 t\in(0,T),
 \\[2mm]
 \eta^2v_t
 =
 \Delta_\xi v
 -
 q{\sf U}^{q-1}
 (1-\chi_1)
 v
 -
 \eta
 \dot\eta
 (\Lambda_\xi v)
 \chi_{\sf mid}
 \\
 \qquad
 \quad
 +
 \eta^{-\frac{2q}{1-q}}
 \lambda^{-\frac{n+2}{2}}
 (\lambda\dot\lambda-\sigma)
 (\Lambda_y{\sf Q})
 \chi_1(1-\chi_{\sf in})
 \\
 \qquad
 \quad
 +
 \eta^{-\frac{2q}{1-q}}
 (
 F_1[v,w]
 (1-\chi_{{\sf in}})
 \chi_{{\sf sq}}
 -
 F_2[w]
 \chi_{{\sf sq}}
 +
 h[\epsilon]
 )
 \\
 \qquad
 \quad
 +
 \eta^{-\frac{2q}{1-q}}
 (
 g
 +
 N[\epsilon,v,w]
 )
 {\bf 1}_{|\xi|<1}
 &
 \text{in } |\xi|<{\sf R}_{\sf mid},\
 t\in(0,T),
 \\[2mm]
 w_t
 =
 \Delta_xw
 -
 q
 {\sf U}_\infty^{q-1}
 w
 +
 (
 F_1[v,w]
 -
 F_2[w]
 )
 (1-\chi_{{\sf sq}})
 \\
 \qquad
 \quad
 +
 k[v]
 +
 (
 g
 +
 N[\epsilon,v,w]
 )
 {\bf 1}_{|\xi|>1}
 &
 \text{in } x\in\R^n,\
 t\in(0,T).
 \end{cases}
 \end{align}
 For convenience,
 we here write the explicit expression of $F_i$ ($i=1,2$).
 \begin{itemize}
 \item 
 $F_1[v,w]=\lambda^{-2}
 V(\eta^\frac{2}{1-q}v\chi_{{\sf mid}}+w)
 \chi_2$,
 \item
 $F_2[w]=
 q
 (
 \eta^{-2}
 {\sf U}^{q-1}
 (1-\chi_1)
 -
 {\sf U}_\infty^{q-1}
 )
 w
 \chi_2$.
 \end{itemize}
 Once a pair of solutions
 $(\lambda(t),\epsilon(y,t),v(\xi,t),w(x,t))$ of \eqref{eq5.2}
 is obtained,
 $u(x,t)$ described in \eqref{eq5.1} gives a solution of \eqref{eq1.1}.
 This decomposition is crucial in this procedure,
 which is the same spirit as in \cite{Cortazar,del_Pino}.
 We appropriately choose the initial data and the boundary condition
 in each equation of \eqref{eq5.2}
 so that solutions of \eqref{eq5.2} decay enough as $t\to T$.

\subsection{Fixed point argument}
\label{sec_5.3}
 Let ${\sf d}_1\in(0,1)$ be a small constant,
 and define
 \[
 {\sf l}_{\sf out}
 =
 L_2
 (T-t)^{-\frac{1}{2}+{\sf b}_{\sf out}}
 \qquad
 \text{with }
 {\sf b}_{\sf out}
 =
 \tfrac{{\sf d}_1}{2(\gamma+2J-\frac{2}{1-q}+3{\sf d}_1)}.
 \]
 The constant $L_2$ is chosen so that
 \[
 (T-t)^{{\sf d}_1}
 (T-t)^{\frac{\gamma}{2}+J}
 |z|^{\gamma+2J+3{\sf d}_1}
 =
 {\sf U}_\infty(x)
 \qquad
 \text{on }
 |z|={\sf l}_{\sf out}.
 \]
 Furthermore
 we define
 \begin{align*}
 {\cal W}(x,t)
 &=
 \begin{cases}
 (T-t)^{{\sf d}_1}
 (T-t)^{\frac{\gamma}{2}+J}
 |z|^\gamma
 & \text{for } |z|<1,
 \\
 (T-t)^{{\sf d}_1}
 (T-t)^{\frac{\gamma}{2}+J}
 |z|^{\gamma+2J+3{\sf d}_1}
 & \text{for } 1<|z|<{\sf l}_{\sf out}(t),
 \\
 {\sf U}_\infty(x)
 & \text{for }
 {\sf l}_{\sf out}(t)\sqrt{T-t}<|x|<1,
 \\
 {\sf M}_0|x|^{-1}
 & \text{for } |x|>1,
 \end{cases}
 \end{align*}
 where ${\sf M}_0={\sf U}_\infty(x)|_{|x|=1}$,
 and
 \begin{align*}
 {\cal V}(\xi,t)
 &=
 (T-t)^{{\sf d}_1}
 (1+|\xi|^2)^\frac{\gamma}{2},
 \qquad
 \text{where }
 \xi=\eta(t)^{-1}x.
 \end{align*}
 From the definition of ${\sf l}_{\sf out}$,
 the function ${\cal W}(x,t)$ is continuous on $\R^n\times[0,T]$.
 We now introduce two inequalities
 to define the functional spaces used in a fixed point argument.
 \begin{align}
 \label{eq5.3}
 |w_1(x,t)|
 &\leq
 {\sf R}_1^{-1}
 {\cal W}(x,t)
 \qquad
 \text{for }
 (x,t)\in\R^n\times[0,T-\delta],
 \\
 \label{eq5.4}
 |v_1(x,t)|
 &\leq
 {\cal V}(x,t)
 \qquad
 \text{for }
 (x,t)\in\overline{B}_{{\sf R}_{\sf mid}}\times[0,T-\delta],
 \end{align}
 where ${\sf R}_1>0$ is a large constant.
 \begin{enumerate}[(i)]
 \item
 Let $X_1^{(\delta)}$ be the set of all continuous functions
 on $\R^n\times[0,T-\delta]$
 satisfying \eqref{eq5.3},

 \item
 let $X_2^{(\delta)}$ be the set of all continuous functions
 on $\overline{B}_{{\sf R}_{\sf mid}}\times[0,T-\delta]$
 satisfying
 \eqref{eq5.4},
 and

 \item
 set
 $X^{(\delta)}=X_1^{(\delta)}\times X_2^{(\delta)}$.

 \item
 For the special case,
 we write
 $X_1=X_1^{(\delta)}|_{\delta=0}$,
 $X_2=X_2^{(\delta)}|_{\delta=0}$,
 $X=X^{(\delta)}|_{\delta=0}$.
 \end{enumerate}
 Since a function $w_1(x,t)\in X_1^{(\delta)}$ is not defined in
 $t\in(T-\delta,T)$,
 we extend it a continuous function on $\R^n\times[0,T]$.
 \[
 w_1^{\sf ext}(x,t)=
 \begin{cases}
 w_1(x,t)
 & \text{if } (x,t)\in\R^n\times[0,T-\delta],
 \\
 w_1(x,T-\delta)
 & \text{if } (x,t)\in\R^n\times[T-\delta,T].
 \end{cases}
 \]
 Similarly
 we define
 \[
 v_1^{\sf ext}(\xi,t)
 =
 \begin{cases}
 v_1(\xi,t)
 & \text{if }
 (\xi,t)\in
 \overline{B}_{{\sf R}_{\sf mid}}\times[0,T-\delta],
 \\
 v_1(\xi,T-\delta)
 & \text{if }
 (\xi,t)\in
 \overline{B}_{{\sf R}_{\sf mid}}\times[T-\delta,T],
 \\
 v_1(B_{{\sf R}_{\sf mid}}\frac{\xi}{|\xi|},t)
 & \text{if }
 (\xi,t)\in(\R^n\setminus\overline{B}_{{\sf R}_{\sf mid}})\times[T-\delta,T],
 \\
 v_1(B_{{\sf R}_{\sf mid}}\frac{\xi}{|\xi|},T-\delta)
 & \text{if }
 (\xi,t)\in(\R^n\setminus\overline{B}_{{\sf R}_{\sf mid}})\times[T-\delta,T].
 \end{cases}
 \]
 For our purpose,
 these extended functions must be
 $v_1^{\sf ext}(\xi,t)=w_1^{\sf ext}(x,t)=0$ for $t=T$.
 Hence we additionally define
 \begin{align*}
 \bar{w}_1(x,t)
 &=
 \begin{cases}
 w_1^{\sf ext}(x,t) & \text{if } x\in\R^n,\ t\in[0,T-\delta],
 \\
 {\cal W}(x,t) & \text{if }x\in\R^n,\ t\in[T-\delta,T]
 \text{\ \ and\ \ }
 w_1^{\sf ext}(x,t)>{\cal W}(x,t),
 \\
 w_1^{\sf ext}(x,t) & \text{if } x\in\R^n,\ t\in[T-\delta,T]
 \text{\ \ and\  \ }
 |w_1^{\sf ext}(x,t)|\leq{\cal W}(x,t),
 \\
 -{\cal W}(x,t) & \text{if } x\in\R^n,\ t\in[T-\delta,T]
 \text{\ \ and\ \ }
 w_1^{\sf ext}(x,t)<-{\cal W}(x,t),
 \end{cases}
 \\
 \bar{v}_1(\xi,t)
 &=
 \begin{cases}
 v_1^{\sf ext}(\xi,t) & \text{if } x\in\R^n,\ t\in[0,T],
 \\
 {\cal V}(\xi,t) & \text{if }x\in\R^n,\ t\in[T-\delta,T]
 \text{\ \ and\ \ }
 v_1^{\sf ext}(\xi,t)>{\cal V}(\xi,t),
 \\
 v_1^{\sf ext}(\xi,t) & \text{if } x\in\R^n,\ t\in[T-\delta,T]
 \text{\ \ and\  \ }
 |v_1^{\sf ext}(\xi,t)|\leq{\cal V}(\xi,t),
 \\
 -{\cal V}(\xi,t) & \text{if } x\in\R^n,\ t\in[T-\delta,T]
 \text{\ \ and\ \ }
 v_1^{\sf ext}(\xi,t)<-{\cal V}(\xi,t)
 \end{cases}
 \end{align*}
 and put
 \[
 \tilde{w}_1(x,t)
 =
 \bar{w}_1(x,t)
 \chi_\delta(t),
 \qquad
 \tilde{v}_1(\xi,t)
 =
 \bar{v}_1(\xi,t)
 \chi_\delta(t).
 \]
 The cut off function $\chi_\delta(t)$ is defined by
 $\chi_{\delta}(t)=1$ for $t<T-\delta$
 and
 $\chi_{\delta}(t)=0$ for $t>T-\frac{1}{2}\delta$.
 From this definition,
 we see that
 $\tilde{w}_1(x,t)\in C(\R^n\times[0,T])$ if $w_1\in X_1^{(\delta)}$
 and
 \begin{align}
 \nonumber
 \tilde{w}_1(x,t)
 &=
 w_1(x,t)
 \qquad
 \text{for } x\in\R^n\times[0,T-\delta],
 \\
 \nonumber
 |\tilde{w}_1(x,t)|
 &\leq
 {\sf R}_1^{-1}
 \begin{cases}
 (T-t)^{{\sf d}_1}
 (T-t)^{\frac{\gamma}{2}+J}
 |z|^\gamma
 & \text{for } |z|<1,\ t\in[0,T],
 \\
 (T-t)^{{\sf d}_1}
 (T-t)^{\frac{\gamma}{2}+J}
 |z|^{\gamma+2J+3{\sf d}_1}
 & \text{for } 1<|z|<{\sf l}_{\sf out}(t),\ t\in[0,T],
 \\
 {\sf U}_\infty(x)
 & \text{for }
 \sqrt{T-t}\cdot{\sf l}_{\sf out}(t)<|x|<1,\ t\in[0,T],
 \\
 {\sf M}_0|x|^{-1}
 & \text{for } |x|>1,\ t\in[0,T],
 \end{cases}
 \\
 \label{eq5.5}
 \tilde w_1(x,t)
 &=
 0
 \qquad \text{for } x\in\R^n,\ t\in[T-\tfrac{1}{2}\delta,T].
 \end{align}
 Similarly
 $\tilde{v}_1(\xi,t)\in C(\R^n\times[0,T])$ if $v_1\in X_2^{(\delta)}$
 and
 \begin{align}
 \nonumber
 \tilde{v}_1(\xi,t)
 &=
 v_1(\xi,t)
 \qquad
 \text{for } x\in\overline{B}_{{\sf R}_{\sf mid}}\times[0,T-\delta],
 \\
 \nonumber
 |\tilde{v}_1(x,t)|
 &\leq
 (T-t)^{{\sf d}_1}(1+|\xi|^2)^\frac{\gamma}{2}
 \qquad
 \text{for } (x,t)\in\R^n\times[0,T],
 \\
 \label{eq5.6}
 \tilde{v}_1(x,t)
 &=
 0 \qquad
 \text{for } (x,t)\in\R^n\times[T-\tfrac{1}{2}\delta,T].
 \end{align}
 We introduce the metrics on $X^{(\delta )}$ and $X$.
 \begin{align*}
 d_{X^{(\delta)}}((w_1,v_1),(w_2,v_2))
 &=
 \sup_{(x,t)\in\R^n\times[0,T-\delta]}
 |w_1-w_2|
 +
 \sup_{(\xi,t)\in\overline{B}_{{\sf R}_{\sf mid}}\times[0,T-\delta]}
 |v_1-v_2|,
 \\
 d_{X}((\tilde w_1,\tilde v_1),(\tilde w_2,\tilde v_2))
 &=
 \sup_{(x,t)\in\R^n\times[0,T]}
 |\tilde w_1-\tilde w_2|
 +
 \sup_{(\xi,t)\in\overline{B}_{{\sf R}_{\sf mid}}\times[0,T]}
 |\tilde v_1-\tilde v_2|.
 \end{align*}
 From this construction,
 we easily see that
 the mapping
 ${\mathcal T}_1(w_1,v_1)=(\tilde w_1,\tilde v_1)$ is continuous
 from
 $(X^{(\delta)},d_{X^{(\delta)}})$
 to
 $(X,d_{X})$.
 To formulate our problem as a fixed point problem,
 we now define the mapping ${\mathcal T}_2:X\to X\subset X^{(\delta )}$.
 For given $(w_1,v_1)\in X^{(\delta)}$,
 we solve the first equation of \eqref{eq5.2}.
 \[
 \lambda^2\epsilon_t
 =
 H_y\epsilon
 +
 (
 \lambda\dot\lambda
 -
 \sigma
 )
 \Lambda_y{\sf Q}
 +
 \lambda^\frac{n+2}{2}
 F_1[\tilde v_1,\tilde w_1]
 \qquad
 \text{in } |y|<{\sf R}_{\sf in},\
 t\in(0,T),
 \]
 where
 $(\tilde w_1,\tilde v_1)={\mathcal T}_1(w_1,v_1)$.
 In this step,
 we determine $\lambda(t)$ and
 obtain a solution $\epsilon(y,t)$ satisfying $\epsilon(y,t)\to0$ as $t\to T$.
 Next
 we construct $v(\xi,t)$ as a solution of the second equation of \eqref{eq5.2}.
 \begin{align*}
 \eta^2v_t
 &=
 \Delta_\xi v
 -
 q{\sf U}^{q-1}
 (1-\chi_1)
 v
 -
 \eta
 \dot\eta
 (\Lambda_\xi v)
 +
 \eta^{-\frac{2q}{1-q}}
 \lambda^{-\frac{n+2}{2}}
 (\lambda\dot\lambda-\sigma)
 (\Lambda_y{\sf Q})
 \chi_1(1-\chi_{\sf in})
 \\
 & \quad
 +
 \eta^{-\frac{2q}{1-q}}
 (
 F_1[\tilde v_1,\tilde w_1]
 (1-\chi_{{\sf in}})
 \chi_{{\sf sq}}
 -
 F_2[\tilde w_1]
 \chi_{{\sf sq}}
 +
 h[\epsilon]
 )
 \\
 & \quad
 +
 \eta^{-\frac{2q}{1-q}}
 (
 g
 +
 N[\epsilon,\tilde v_1,\tilde w_1]
 )
 {\bf 1}_{|\xi|<1}
 \qquad
 \text{in } |\xi|<{\sf R}_{\sf mid},\
 t\in(0,T).
 \end{align*}
 Here
 $\lambda(t)$ and $\epsilon(y,t)$ are functions obtained in the first step,
 and
 $(\tilde w_1,\tilde v_1)={\mathcal T}_1(w_1,v_1)$ are given functions.
 We will see that $v\in X_2^{(\delta)}$.
 We finally consider
 \begin{align*}
 w_t
 &=
 \Delta_xw
 -
 q
 {\sf U}_\infty^{q-1}
 w
 +
 (
 F_1[\tilde v_1,\tilde w_1]
 -
 F_2[\tilde w_1]
 )
 (1-\chi_{{\sf sq}})
 \\
 & \quad
 +
 k[v]
 +
 (
 g
 +
 N[\epsilon,\tilde v_1,\tilde w_1]
 )
 {\bf 1}_{|\xi|>1}
 \qquad
 \text{in } x\in\R^n,\
 t\in(0,T).
 \end{align*}
 As in the second step,
 $\lambda(t)$ and $\epsilon(y,t)$ are functions obtained in the first step,
 $v(\xi,t)$ is the function constructed in the second step
 and $(\tilde w_1,\tilde v_1)={\mathcal T}_1(w_1,v_1)$.
 We can construct a solution $w(x,t)$ of this equation
 such that $w\in X_1^{(\delta)}$.
 By using this $v(\xi,t)$ and $w(x,t)$,
 we define
 \[
 {\cal T}_2
 (\tilde w_1,\tilde v_1)
 =
 (w,v).
 \]
 This mapping ${\mathcal T}_2:X\to X\subset X^{(\delta)}$ is continuous.
 Hence
 the mapping
 ${\mathcal T}={\mathcal T}_2\circ{\mathcal T}_1:X^{(\delta)}\to X^{(\delta)}$
 is also continuous.
 Furthermore
 we can show that
 the mapping ${\mathcal T}:X^{(\delta)}\to X^{(\delta)}$ is compact
 for any $\delta>0$.
 Therefore
 the Schauder fixed point theorem
 shows that
 ${\mathcal T}:X^{(\delta)}\to X^{(\delta)}$ has a fixed point
 $(w_1,v_1)\in X^{(\delta)}$.
 This gives a solution $u^{(\delta)}(x,t)$ of \eqref{eq1.1}
 with the desired properties.
 However this $u^{(\delta)}(x,t)$ is defined only on $\R^n\times[0,T-\delta]$.
 Finally
 we take $\delta\to0$ and
 obtain a solution $u(x,t)=\lim_{\delta\to0}u^{(\delta)}(x,t)$,
 which proves Theorem \ref{Thm1}.
 For the rest of paper,
 we investigate the mapping ${\cal T}_2$ and prove (i) - (iii).
 \begin{enumerate}[(i)]
 \item ${\cal T}_2:X\to X^{(\delta)}$ is well defined,
 \item ${\cal T}_2:(X,d_X)\to(X^{(\delta)},d_{X^{(\delta)}})$ is continuous and
 \item ${\cal T}_2:(X,d_X)\to(X^{(\delta)},d_{X^{(\delta)}})$ is compact.
 \end{enumerate}
 Section \ref{sec_6} - Section \ref{sec_8} are devoted to the proof of (i).
 We do not give precise proofs of (ii) - (iii),
 since those are standard.
 In fact,
 since our construction $(\tilde w_1,\tilde v_1)\mapsto(w,v)$
 is unique at each step, (ii) follows.
 Furthermore
 since
 the H\"older continuity
 is bounded by the $L^\infty$ bound
 from standard parabolic estimates (see Theorem 6.29 in \cite{Lieberman} p. 131),
 (iii) holds.

 \subsection{Notations}
 We fix
 \begin{itemize}
 \item ${\sf R}_{\sf in}={\sf R}_{\sf mid}=-\log T$,
 \item ${\sf R}_1=\log {\sf R}_{\sf in}$.
 \end{itemize}
 Throughout this paper,
 the symbol $C$ denotes a generic positive constant
 depending only on $q,n,J$
 (independent of
 $\delta$, ${\sf b}$, ${\sf d}_1$
 ${\sf R}_{\sf in}$, ${\sf R}_{\sf mid}$, ${\sf R}_1$, ${\sf r}_0$, ${\sf r}_3$).
 For real numbers $X$ and $Y$,
 we write $X\lesssim Y$
 if there exists a generic positive constant $C$
 such that $|X|\leq C|Y|$.

 \section{In the inner region}
 \label{sec_6}
 In this section,
 we construct a pair of solutions $(\lambda(t),\epsilon(y,t))$
 in the same procedure as in the proof of Proposition 7.1 \cite{Cortazar}
 (see also Lemma 4.1 in \cite{del_Pino}).
 Here we follow a simplified version of
 their method obtained in \cite{Harada,Harada_2}
 (see Section 6 in \cite{Harada_2}).
 By virtue of Lemma \ref{Lem3.3},
 we can skip one procedure in the proof of Proposition 7.1 in \cite{Cortazar},
 which provides a more direct approach.
 Let $\mu_i^{({\sf R}_{\sf in})}$ be the $i$th eigenfunction of
 \[
 \begin{cases}
 -H_ye=\mu e & \text{in } |y|<4{\sf R}_{\sf in},
 \\
 e=0 & \text{on } |y|=4{\sf R}_{\sf in}
 \end{cases}
 \]
 and
 $\psi_i^{({\sf R}_{\sf in})}$ be the associated eigenfunctions ($i\in\N$).
 Here we take $\|\psi_i^{({\sf R}_{\sf in})}\|_{L_y^2(B_{4{\sf R}_{\sf in}})}=1$.
 For simplicity,
 we write
 \begin{itemize}
 \item 
 $\mu_i=\mu_i^{({\sf R}_{\sf in})}$,
 \item
 $\psi_i=\psi_i^{({\sf R}_{\sf in})}$.
 \end{itemize}

 \subsection{Choice of $\lambda(t)$}
 \label{sec_6.1}
 We fix $(v_1,w_1)\in X^{(\delta)}$,
 and write ${\mathcal T}_1(v_1,w_1)=(\tilde v_1,\tilde w_1)\in X$
 (see Section \ref{sec_5.3}).
 In this subsection,
 we determine $\lambda(t)$
 such that
 $\lambda(t)\to0$ as $t\to T$.
 As is explained in Section \ref{sec_5.3},
 we consider
 \begin{align}\label{eq6.1}
 &
 \begin{cases}
 \dis
 \lambda^2\epsilon_t
 =
 H_y\epsilon
 +
 (
 \lambda\dot\lambda
 -
 \sigma
 )
 \Lambda_y{\sf Q}
 +
 \lambda^\frac{n-2}{2}
 V
 (\eta^\frac{2}{1-q}\tilde v_1+\tilde w_1)
 &
 \text{in } |y|<4{\sf R}_{\sf in},\ t\in(0,T),
 \\
 \epsilon=0
 &
 \text{on } |y|=4{\sf R}_{\sf in},\ t\in(0,T),
 \\
 \epsilon=\epsilon_0
 &
 \text{on } |y|<4{\sf R}_{\sf in},\ t=0.
 \end{cases}
 \end{align}
 Initial data $\epsilon_0$ will be chosen in Section \ref{sec_6.2}
 so that $\epsilon(y,t)\to0$ as $t\to T$.
 We define $\lambda(t)$ as a solution of 
 \begin{equation}\label{eq6.2}
 \begin{cases}
 (\lambda\dot\lambda-\sigma)
 (
 \Lambda_y{\sf Q},
 \psi_2
 )_{L_y^2(B_{4{\sf R}_{\sf in}})}
 +
 \lambda^\frac{n-2}{2}
 (
 V
 (\eta^\frac{2}{1-q}\tilde v_1+\tilde w_1),
 \psi_2
 )_{L_y^2(B_{4{\sf R}_{\sf in}})}
 =
 0
 \qquad
 \text{for } t\in(0,T),
 \\
 \lambda(t)=0
 \qquad
 \text{for } t=T.
 \end{cases}
 \end{equation}
 We note that
 the inner product
 $(V
 (\eta^\frac{2}{1-q}\tilde v_1+\tilde w_1),
 \psi_2
 )_{L_y^2(B_{4{\sf R}_{\sf in}})}$ in \eqref{eq6.2}
 is nonlinear as a function of $\lambda$,
 since it is expressed by
 \[
 (
 V(\eta^\frac{2}{1-q}\tilde v_1+\tilde w_1),
 \psi_2
 )_{L_y^2(B_{4{\sf R}_{\sf in}})}
 =
 \int_{B_{4{\sf R}_{\sf in}}}
 V(y)
 \{
 \eta^\frac{2}{1-q}\tilde v_1(\lambda\eta^{-1}y,t)+\tilde w_1(\lambda y,t)
 \}
 \psi_2(y)
 dy.
 \]
 We put
 ${\sf a}=\lambda^\frac{6-n}{2}$.
 Then \eqref{eq6.2} is written as 
 \begin{equation}\label{eq6.3}
 \begin{cases}
 \frac{2}{6-n}
 (\dot{\sf a}+\frac{6-n}{2{\sf A}_1}\eta^\frac{2}{1-q})
 (
 \Lambda_y{\sf Q},
 \psi_2
 )_{L_y^2(B_{4{\sf R}_{\sf in}})}
 +
 (
 V
 \eta^\frac{2}{1-q}\tilde v_1(\eta^{-1}{\sf a}^{\frac{2}{6-n}}y,t),
 \psi_2
 )_{L_y^2(B_{4{\sf R}_{\sf in}})}
 \\[2mm]
 \hspace{45mm}
 +
 (
 V
 \tilde w_1({\sf a}^{\frac{2}{6-n}}y,t),
 \psi_2
 )_{L_y^2(B_{4{\sf R}_{\sf in}})}
 =
 0
 \qquad
 \text{for } t\in(0,T),
 \\
 {\sf a}=0
 \qquad
 \text{for } t=T.
 \end{cases}
 \end{equation}
 To construct a solution of \eqref{eq6.3},
 we introduce
 \begin{align*}
 {\cal A}
 &=
 \{{\sf a}\in C([0,T]),\
 |{\sf a}(t)-{\sf a}_0(t)|\leq(T-t)^{\frac{{\sf d}_1}{2}}{\sf a}_0(t)\},
 \\
 & \quad
 \text{where }
 {\sf a}_0(t)
 =
 \frac{6-n}{2{\sf A}_1}
 \int_t^T
 \eta(t_1)^\frac{2}{1-q}dt_1
 =
 \kappa_1
 (T-t)
 \eta^\frac{2}{1-q}.
 \end{align*}
 We now show that
 for any $(v_1,w_1)\in X^{(\delta)}$,
 there exists the unique solution ${\sf a}(t)\in{\cal A}$ of \eqref{eq6.3}.
 The existence is proved by a fixed point argument.
 In fact,
 for any ${\sf a}_1\in {\cal A}$,
 we define ${\sf a}(t)$ as a unique solution of
 \eqref{eq6.3} with replaced
 $\tilde v_1(\eta^{-1}{\sf a}^\frac{2}{6-n}y,t)$,
 $\tilde w_1({\sf a}^\frac{2}{6-n}y,t)$
 by
 $\tilde v_1(\eta^{-1}{\sf a}_1^\frac{2}{6-n}y,t)$,
 $\tilde w_1({\sf a}_1^\frac{2}{6-n}y,t)$
 respectively.
 From Lemma \ref{Lem3.1},
 we easily see that
 \begin{align*}
 |
 (
 V
 \eta^\frac{2}{1-q}
 \tilde v_1
 (\eta^{-1}{\sf a}_1^{\frac{2}{6-n}}y,t),
 &
 \psi_2
 )_{L_y^2(B_{4{\sf R}_{\sf in}})}
 +
 (
 V
 \tilde w_1({\sf a}_1^{\frac{2}{6-n}}y,t),
 \psi_2
 )_{L_y^2(B_{4{\sf R}_{\sf in}})}
 |
 \\
 &\lesssim
 (T-t)^{{\sf d}_1}
 \eta^\frac{2}{1-q}
 |
 (V,\psi_2
 )_{L_y^2(B_{4{\sf R}_{\sf in}})}
 |
 \\
 &\lesssim
 (T-t)^{{\sf d}_1}
 \eta^\frac{2}{1-q}
 \qquad
 \text{for }
 (\tilde v_1,\tilde w_1)\in X,\
 {\sf a}_1\in{\cal A}.
 \end{align*}
 Hence
 the solution ${\sf a}(t)$ satisfies
 \[
 |{\sf a}(t)-{\sf a}_0(t)|
 \lesssim
 (T-t)^{{\sf d}_1+1}
 \eta^\frac{2}{1-q}
 \lesssim
 (T-t)^{{\sf d}_1}
 {\sf a}_0.
 \]
 The Schauder fixed point theorem shows that
 the mapping ${\sf a}_1\mapsto{\sf a}$ has a fixed point in ${\cal A}$,
 which gives a solution of \eqref{eq6.3}.
 We next discuss the uniqueness of solutions to \eqref{eq6.3}.
 Let ${\sf a}_1(t),{\sf a}_2(t)\in{\cal A}$ be solutions of \eqref{eq6.3}.
 For simplicity,
 we write $\lambda_i(t)={\sf a}_i(t)^\frac{6-n}{2}$
 ($i=1,2$).
 A direct computation shows that
 \begin{align*}
 (
 V\tilde v_1(\eta^{-1}\lambda_iy,t),
 \psi_2
 )_{L_y^2(B_{4{\sf R}_{\sf in}})}
 &=
 \int_{|y|<4{\sf R}_{\sf in}}
 V(y)
 \tilde v_1(\eta^{-1}\lambda_iy,t)
 \psi_2(y)
 dy
 \\
 &=
 \eta^n\lambda_i^{-n}
 \int_{|\xi|<4{\sf R}_{\sf in}\eta\lambda_i^{-1}}
 V(\eta\lambda_i^{-1}\xi)
 \tilde v_1(\xi,t)
 \psi_2(\eta\lambda_i^{-1}\xi)
 d\xi.
 \end{align*}
 By change of variables,
 we get
 \begin{align*}
 (
 V\tilde v_1
 &
 (\eta^{-1}\lambda_2y,t),
 \psi_2
 )_{L_y^2(B_{4{\sf R}_{\sf in}})}
 -
 (
 V\tilde v_1(\eta^{-1}\lambda_2y,t)),
 \psi_2
 )_{L_y^2(B_{4{\sf R}_{\sf in}})}
 \\
 &=
 \eta^n\lambda_1^{-n}
 \int_{|\xi|<4{\sf R}_{\sf in}\eta\lambda_1^{-1}}
 |
 V(\eta\lambda_1^{-1}\xi)
 \tilde v_1(\xi,t)
 \psi_2(\eta\lambda_1^{-1}\xi)
 d\xi
 \\
 & \quad
 -
 \eta^n\lambda_2^{-n}
 \int_{|\xi|<4{\sf R}_{\sf in}\eta\lambda_2^{-1}}
 V(\eta\lambda_2^{-1}\xi)
 \tilde v_1(\xi,t)
 \psi_2(\eta\lambda_2^{-1}\xi)
 d\xi
 \\
 &=
 \eta^n
 (\lambda_1^{-n}-\lambda_2^{-n})
 \int_{|\xi|<4{\sf R}_{\sf in}\eta\lambda_1^{-1}}
 V(\eta\lambda_1^{-1}\xi)
 \tilde v_1(\xi,t)
 \psi_2(\eta\lambda_1^{-1}\xi)
 d\xi
 \\
 & \quad
 +
 \eta^n\lambda_2^{-n}
 \int_{|\xi|<4{\sf R}_{\sf in}\eta\lambda_1^{-1}}
 (
 V(\eta\lambda_1^{-1}\xi)
 \psi_2(\eta\lambda_1^{-1}\xi)
 -
 V(\eta\lambda_2^{-1}\xi)
 \psi_2(\eta\lambda_2^{-1}\xi)
 )
 \tilde v_1(\xi,t)
 d\xi
 \\
 & \quad
 +
 \eta^n\lambda_2^{-n}
 \left(
 \int_{|\xi|<4{\sf R}_{\sf in}\eta\lambda_1^{-1}}
 d\xi
 -
 \int_{|\xi|<4{\sf R}_{\sf in}\eta\lambda_2^{-1}}
 d\xi
 \right)
 V(\eta\lambda_2^{-1}\xi)
 \psi_2(\eta\lambda_2^{-1}\xi)
 \tilde v_1(\xi,t)
 d\xi.
 \end{align*}
 Since $\lambda_1^{-1}\lambda_2=1$ as $t\to T$,
 it follows from Lemma \ref{Lem3.1} that 
 \begin{align*}
 |
 (
 V\tilde v_1
 &
 (\eta^{-1}\lambda_1y,t),
 \psi_2
 )_{L_y^2(B_{4{\sf R}_{\sf in}})}
 -
 (
 V\tilde v_1(\eta^{-1}\lambda_2y,t)),
 \psi_2
 )_{L_y^2(B_{4{\sf R}_{\sf in}})}
 |
 \\
 &\lesssim
 (T-t)^{{\sf d}_1}
 \eta^\frac{2}{1-q}
 \lambda_1^n
 |\lambda_1^{-n}-\lambda_2^{-n}|
 \int_{|y|<4{\sf R}_{\sf in}}
 |
 V(y)
 \psi_2(y)
 |
 dy
 \\
 & \quad
 +
 (T-t)^{{\sf d}_1}
 \eta^\frac{2}{1-q}
 \lambda_1^n\lambda_2^{-n}
 \int_{|y|<4{\sf R}_{\sf in}}
 |
 V(y)
 \psi_2(y)
 -
 V(y_1)
 \psi_2(y_1)
 |
 dy
 \qquad
 (y_1=\lambda_1\lambda_2^{-1}y)
 \\
 & \quad
 +
 (T-t)^{{\sf d}_1}
 \eta^\frac{2}{1-q}
 \lambda_1^n
 \lambda_2^{-n}
 \left|
 \int_{|y|<4{\sf R}_{\sf in}}
 dy
 -
 \int_{|\xi|<4{\sf R}_{\sf in}\lambda_1\lambda_2^{-1}}
 dy
 \right|
 V(\lambda_1\lambda_2^{-1}y)
 \psi_2(\lambda_1\lambda_2^{-1}y)
 dy
 \\
 &\lesssim
 (T-t)^{{\sf d}_1}
 \eta^\frac{2}{1-q}
 \lambda_1^n
 |\lambda_1^{-n}-\lambda_2^{-n}|
 +
 (T-t)^{{\sf d}_1}
 \eta^\frac{2}{1-q}
 \lambda_1^n\lambda_2^{-n}
 |1-\lambda_1\lambda_2^{-1}|
 \\
 & \quad
 +
 (T-t)^{{\sf d}_1}
 \eta^\frac{2}{1-q}
 \lambda_1^n
 \lambda_2^{-n}
 {\sf R}_{\sf in}^{-2}
 |1-\lambda_1\lambda_2^{-1}|
 \\
 &\lesssim
 (T-t)^{{\sf d}_1}
 \eta^\frac{2}{1-q}
 \lambda_2^{-1}
 |\lambda_1-\lambda_2|
 \\
 &\lesssim
 (T-t)^{{\sf d}_1}
 \eta^\frac{2}{1-q}
 \lambda_2^{-1}
 {\sf a}_1^\frac{4-n}{2}
 |{\sf a}_1-{\sf a}_2|.
 \end{align*}
 The same computation shows
 \begin{align*}
 |
 (
 V\tilde w_1
 (\lambda_1y,t)
 ,
 \psi_2
 )_{L_y^2(B_{4{\sf R}_{\sf in}})}
 -
 (
 V
 &
 \tilde w_1
 (\lambda_2y,t)),
 \psi_2
 )_{L_y^2(B_{4{\sf R}_{\sf in}})}
 |
 \lesssim
 (T-t)^{{\sf d}_1}
 \eta^\frac{2}{1-q}
 \lambda_2^{-1}
 {\sf a}_1^\frac{4-n}{2}
 |{\sf a}_1-{\sf a}_2|.
 \end{align*}
 We recall that
 $\tilde v_1=\tilde w_1=0$ for $t\in[T-\frac{\delta}{2},T]$
 (see \eqref{eq5.5}, \eqref{eq5.6}).
 Hence
 ${\sf a}_1(t)={\sf a}_2(t)={\sf a}_0(t)$ for $t\in[T-\frac{\delta}{2},T]$.
 Therefore
 since ${\sf a}_1,{\sf a}_2\in{\cal A}$ are solutions of \eqref{eq6.3},
 there exits $B>0$ such that
 \begin{align*}
 |
 \tfrac{d}{dt}
 ({\sf a}_1-{\sf a}_2)
 |
 &\lesssim
 \begin{cases}
 (T-t)^{{\sf d}_1}
 \eta^\frac{2}{1-q}
 \lambda_2^{-1}
 {\sf a}_1^\frac{4-n}{2}
 |{\sf a}_1-{\sf a}_2|
 & \text{for }
 t\in[0,T-\frac{\delta}{2}]
 \\
 0
 & \text{for }
 t\in[T-\frac{\delta}{2},T] 
 \end{cases}
 \\
 &\lesssim
 (T-t)^{{\sf d}_1}
 \eta^\frac{2}{1-q}
 \lambda_2^{-1}
 {\sf a}_1^\frac{4-n}{2}
 |_{t=T-\frac{\delta}{2}}
 \cdot
 |{\sf a}_1-{\sf a}_2|
 \qquad
 \text{for }
 t\in[0,T]
 \\
 &\lesssim
 \delta^{-B}
 |{\sf a}_1-{\sf a}_2|
 \qquad
 \text{for }
 t\in[0,T].
 \end{align*}
 Since ${\sf a}_1(t)={\sf a}_2(t)={\sf a}_0(t)$ for $t\in[T-\frac{\delta}{2},T]$,
 this proves the uniqueness of solutions to \eqref{eq6.3} in ${\cal A}$.

 \subsection{Construction of $\epsilon(y,t)$}
 \label{sec_6.2}
 We next construct a solution $\epsilon(y,t)$  of \eqref{eq6.1}.
 Let $\lambda(t)$ be a solution of \eqref{eq6.2} constructed
 in Section \ref{sec_6.1} from $(v_1,w_1)\in X^{(\delta)}$.
 As is stated in the beginning of Section \ref{sec_6},
 we directly solve \eqref{eq6.1}
 (the approach in \cite{Cortazar} need to introduce additional equations).
 We choose ${\sf m}_1>1$ such that
 ($n\geq5$)
 \begin{align*}
 -H_y|y|^{-2}
 &>
 -
 \tfrac{1}{2}
 \Delta_y|y|^{-2}
 =
 (n-4)
 |y|^{-4}
 \qquad
 \text{for } |y|>{\sf m}_1,
 \\
 -H_y
 |y|^{-(n-4)}
 &>
 -
 \tfrac{1}{2}
 \Delta_y|y|^{-(n-4)}
 =
 (n-4)
 |y|^{-(n-2)}
 \qquad
 \text{for }
 |y|>{\sf m}_1.
 \end{align*}
 The constant ${\sf m}_1$ depends only on $n$.
 We here introduce three equations.
 \begin{align}\label{eq6.4}
 \nonumber
 &
 \begin{cases}
 \dis
 \lambda^2
 \pa_t\epsilon_1
 =
 H_y\epsilon_1
 +
 (
 \lambda\dot\lambda
 -
 \sigma
 )
 \Lambda_y{\sf Q}
 &
 \text{in } {\sf m}_1<|y|<4{\sf R}_{\sf in},\ t\in(0,T),
 \\
 \epsilon_1=0
 &
 \text{on } |y|={\sf m}_1,\ |y|=4{\sf R}_{\sf in},\ t\in(0,T),
 \\
 \epsilon_1=0
 &
 \text{in } {\sf m}_1<|y|<4{\sf R}_{\sf in},\ t=0,
 \end{cases}
 \\
 \nonumber
 &
 \begin{cases}
 \dis
 \lambda^2
 \pa_t\epsilon_2
 =
 H_y\epsilon_2
 +
 \lambda^\frac{n-2}{2}
 V
 (\eta^\frac{2}{1-q}\tilde v_1+\tilde w_1)
 &
 \text{in } {\sf m}_1<|y|<4{\sf R}_{\sf in},\ t\in(0,T),
 \\
 \epsilon_2=0
 &
 \text{on } |y|={\sf m}_1,\ |y|=4{\sf R}_{\sf in},\ t\in(0,T),
 \\
 \epsilon_2=0
 &
 \text{in } {\sf m}_1<|y|<4{\sf R}_{\sf in},\ t=0,
 \end{cases}
 \\
 &
 \begin{cases}
 \dis
 \lambda^2
 \pa_t\epsilon_3
 =
 H_y\epsilon_3
 +
 \underbrace{
 (
 \lambda\dot\lambda
 -
 \sigma
 )
 (\Lambda_y{\sf Q})
 \chi_{{\sf m}_1}
 }_{=\mathcal{R}_1}
 \\
 \qquad
 \qquad
 +
 \underbrace{
 \lambda^\frac{n-2}{2}
 V
 (\eta^\frac{2}{1-q}\tilde v_1+\tilde w_1)
 \chi_{{\sf m}_1}
 +
 l[\epsilon_1,\epsilon_2]
 }_{=\mathcal{R}_2}
 &
 \text{in } |y|<4{\sf R}_{\sf in},\ t\in(0,T),
 \\
 \epsilon_3=0
 &
 \text{on } |y|=4{\sf R}_{\sf in},\ t\in(0,T),
 \\
 \epsilon_3=\epsilon_0
 &
 \text{in } |y|<4{\sf R}_{\sf in},\ t=0,
 \end{cases}
 \end{align}
 where
 $l[\epsilon_1,\epsilon_2]=
 -
 2\nabla_y
 (\epsilon_1+\epsilon_2)
 \cdot
 \nabla_y\chi_{{\sf m}_1}
 -
 (\epsilon_1+\epsilon_2)
 (\Delta_y\chi_{{\sf m}_1})$.
 We can verify that
 \[
 \epsilon
 =
 \epsilon_1
 (1-\chi_{{\sf m}_1})
 +
 \epsilon_2
 (1-\chi_{{\sf m}_1})
 +
 \epsilon_3
 \qquad
 \text{with }
 \chi_{{\sf m}_1}
 =
 \begin{cases}
 1 & |y|<{\sf m}_1 \\
 0 & |y|>2{\sf m}_1
 \end{cases}
 \]
 gives a solution of \eqref{eq6.1}.
 We now construct $\epsilon_1(y,t)$.
 We recall that
 $(\lambda\dot\lambda-\sigma)\Lambda_y{\sf Q}
 \lesssim(T-t)^{{\sf d}_1}\sigma|y|^{-(n-2)}$
 from \eqref{eq6.2}.
 Therefore
 from the choice of ${\sf m}_1$,
 a comparison argument shows
 \begin{align}\label{eq6.5}
 |\epsilon_1(y,t)|
 &\lesssim
 {\sf m}_1^{n-4}
 (T-t)^{{\sf d}_1}
 \sigma
 |y|^{-(n-4)}
 \qquad
 \text{for } {\sf m}_1<|y|<4{\sf R}_{\sf in},\ t\in(0,T),
 \\
 \label{eq6.6}
 |\nabla_y\epsilon_1(y,t)|
 &\lesssim
 {\sf m}_1^{n-3}
 (T-t)^{{\sf d}_1}
 \sigma
 |y|^{-(n-3)}
 \qquad
 \text{for } {\sf m}_1<|y|<4{\sf R}_{\sf in},\ t\in(0,T).
 \end{align}
 Estimate \eqref{eq6.6} is obtained from
 a comparison argument to the equation of $\pa_r\epsilon_1(y,t)$.
 From the definition of $X^{(\delta)}$ (see Section \ref{sec_5.3}),
 we easily see that
 \[
 \lambda^\frac{n-2}{2}
 V(\eta^\frac{2}{1-q}\tilde v_1+\tilde w_1)
 \lesssim
 (T-t)^{{\sf d}_1}
 \sigma
 |y|^{-4}
 \]
 for
 ${\sf m}_1<|y|<4{\sf R}_{\sf in}$,
 $t\in(0,T)$.
 Hence
 by a comparison argument,
 we can show that
 \begin{align}\label{eq6.7}
 |\epsilon_2(y,t)|
 \lesssim
 {\sf m}_1^2
 (T-t)^{{\sf d}_1}
 \sigma
 |y|^{-2}
 \qquad
 \text{for } {\sf m}_1<|y|<4{\sf R}_{\sf in},\ t\in(0,T).
 \end{align}
 This together with Lemma \ref{Lem3.5} implies
 \begin{align}\label{eq6.8}
 |\nabla_y\epsilon_2(y,t)|
 \lesssim
 {\sf m}_1^2
 (T-t)^{{\sf d}_1}\sigma
 |y|^{-2}
 \qquad
 \text{for } {\sf m}_1<|y|<4{\sf R}_{\sf in},\ t\in(0,T).
 \end{align}
 To construct $\epsilon_3(y,t)$,
 we first investigate the unstable mode of \eqref{eq6.4}.
 By change of variables
 $s=\int_0^t\tfrac{dt}{\lambda^2}$,
 it holds that
 \begin{align*}
 \pa_s
 (\epsilon_3,\psi_1)_{L_y^2(B_{4{\sf R}_{\sf in}})}
 &=
 -
 \mu_1
 (\epsilon_3,\psi_1)_{L_y^2(B_{4{\sf R}_{\sf in}})}
 +
 (\mathcal{R}_1+\mathcal{R}_2,\psi_1)_{L_y^2(B_{4{\sf R}_{\sf in}})}.
 \end{align*}
 Integrating this equation,
 we get
 \begin{align*}
 e^{\mu_1s}
 (\epsilon_3,\psi_1)_{L_y^2(B_{4{\sf R}_{\sf in}})}
 -
 (\epsilon_0,\psi_1)_{L_y^2(B_{4{\sf R}_{\sf in}})}
 &=
 \int_0^s
 e^{\mu_1s_1}
 (\mathcal{R}_1+\mathcal{R}_2,\psi_1)_{L_y^2(B_{4{\sf R}_{\sf in}})}
 ds_1.
 \end{align*}
 From this observation,
 we now take the initial data $\epsilon_0$ such that
 \begin{align*}
 \epsilon_0
 =
 \alpha_1
 \psi_1
 \qquad
 \text{with }
 \alpha_1
 =
 -
 \int_0^\infty
 e^{\mu_1s_1}
 (\mathcal{R}_1+\mathcal{R}_2,\psi_1)_{L_y^2(B_{4{\sf R}_{\sf in}})}
 ds_1.
 \end{align*}
 This integral is finite since $\mu_1$ is negative
 (see Lemma \ref{Lem3.2}).
 Then
 it follows from Lemma \ref{Lem3.1}
 and \eqref{eq6.5} - \eqref{eq6.8}
 that
 \begin{align*}
 \alpha_1
 &\lesssim
 \int_0^\infty
 e^{\mu_1s_1}
 \{
 (\lambda\dot\lambda-\sigma)
 +
 \lambda^\frac{n-2}{2}
 \eta^\frac{2}{1-q}
 (T-t)^{{\sf d}_1}
 +
 (T-t)^{{\sf d}_1}
 \sigma
 \}
 ds_1
 \\
 &\lesssim
 (T-t)^{{\sf d}_1}
 \sigma
 |_{t=0}
 \int_0^\infty
 e^{\mu_1s_1}
 ds_1
 \\
 &\lesssim
 (T-t)^{{\sf d}_1}
 \sigma
 |_{t=0}.
 \end{align*}
 In the last lien,
 we use $\lim_{R\to\infty}\mu_1=\exists\bar\mu_1<0$
 (see Lemma \ref{Lem3.2}).
 The same computation shows
 \begin{align*}
 (\epsilon_3,\psi_1)_{L_y^2(B_{4{\sf R}_{\sf in}})}
 &=
 -
 e^{-\mu_1s}
 \int_s^\infty
 e^{\mu_1s_1}
 (\mathcal{R}_1+\mathcal{R}_2,\psi_1)_{L_y^2(B_{4{\sf R}_{\sf in}})}
 ds_1
 \\
 \nonumber
 &
 \lesssim
 e^{-\mu_1s}
 \int_s^\infty
 e^{\mu_1s_1}
 \{
 (\lambda\dot\lambda-\sigma)
 +
 (T-t)^{{\sf d}_1}
 \lambda^{\frac{n-2}{2}}
 \eta^\frac{2}{1-q}
 +
 (T-t)^{{\sf d}_1}
 \sigma
 \}
 ds_1
 \\
 &
 \lesssim
 (T-t)^{{\sf d}_1}
 \sigma.
 \end{align*}
 We next derive estimates for the stable mode $\epsilon_3^\bot$
 defined by
 \[
 \epsilon_{3}^\bot
 =
 \epsilon_3
 -
 (\epsilon_3,\psi_1)_{L_y^2(B_{4{\sf R}_{\sf in}})}
 \psi_1
 -
 (\epsilon_3,\psi_2)_{L_y^2(B_{4{\sf R}_{\sf in}})}
 \psi_2.
 \]
 The choice of $\lambda(t)$ immediately implies
 $(\epsilon(t),\psi_2)_{L_y^2(B_{4{\sf R}_{\sf in}})}=0$
 for $t\in(0,T)$.
 Hence it follows from \eqref{eq6.5} and \eqref{eq6.7} that
 \begin{align}\label{eq6.9}
 \nonumber
 (\epsilon_3,\psi_2)_{L_y^2(B_{4{\sf R}_{\sf in}})}
 &=
 (
 -(\epsilon_1+\epsilon_2)(1-\chi_{{\sf m}_1}),\psi_2
 )_{L_y^2(B_{4{\sf R}_{\sf in}})}
 \\
 \nonumber
 &\lesssim
 {\sf m}_1^2
 (T-t)^{{\sf d}_1}
 \sigma
 \left(
 1+
 \int_{1<|y|<4{\sf R}_{\sf in}}
 (
 |y|^{-(n-4)}+|y|^{-2}
 )
 |y|^{-(n-2)}
 dy
 \right)
 \\
 &\lesssim
 {\sf m}_1^2
 {\sf R}_{\sf in}^{6-n}
 (T-t)^{{\sf d}_1}
 \sigma
 \qquad
 \text{for }
 t\in(0,T).
 \end{align}
 We multiply \eqref{eq6.4} by $\epsilon_3^\bot$ and integrate by parts.
 Then we get from \eqref{eq6.5} - \eqref{eq6.8} that
 \begin{align*}
 \tfrac{\lambda^2}{2}
 \tfrac{d}{dt}
 \|\epsilon_{3}^{\bot}\|_{L_y^2(B_{4{\sf R}_{\sf in}})}^2
 &=
 (\epsilon_{3}^{\bot},H_y\epsilon_{3})_{L_y^2(B_{4{\sf R}_{\sf in}})}^2
 +
 (
 \mathcal{R}_1+\mathcal{R}_2,
 \epsilon_{3}^{\bot}
 )_{L_y^2(B_{4{\sf R}_{\sf in}})}
 \\
 &<
 -
 \tfrac{\mu_3}{2}
 \|\epsilon_{3}^{\bot}\|_{L_y^2(B_{4{\sf R}_{\sf in}})}^2
 +
 \tfrac{2}{\mu_3}
 \|\mathcal{R}_1+\mathcal{R}_2\|_{L_y^2(B_{4{\sf R}_{\sf in}})}^2
 \\
 &<
 -
 \tfrac{\mu_3}{2}
 \|\epsilon_{3}^{\bot}(s)\|_{L_y^2(B_{4{\sf R}_{\sf in}})}^2
 +
 \tfrac{C}{\mu_3}
 (T-t)^{2{\sf d}_1}
 \sigma^2.
 \end{align*}
 We note that $\epsilon_0^\bot=0$.
 Hence
 by change of variables $s=\int_0^t\frac{dt}{\lambda^2}$,
 we get
 \begin{align*}
 \|\epsilon_{3}^{\bot}\|_{L_y^2(B_{4{\sf R}_{\sf in}})}^2
 \lesssim
 \tfrac{1}{\mu_3}
 e^{-\mu_3s}
 \int_0^s
 e^{\mu_3s_1}
 (T-t)^{2{\sf d}_1}
 \sigma^2
 ds_1.
 \end{align*}
 From integration by parts,
 we observe that
 \begin{align*}
 \int_0^s
 e^{\mu_3s_1}
 (T-t)^{2{\sf d}_1}
 \sigma^2
 ds_1
 &=
 \left[
 \tfrac{1}{\mu_3}
 e^{\mu_3s_1}
 (T-t)^{2{\sf d}_1}
 \sigma^2
 \right]_{s_1=0}^{s_1=s}
 -
 \tfrac{1}{\mu_3}
 \int_0^s
 e^{\mu_3s_1}
 \frac{dt}{ds}
 \frac{d}{dt}
 \{
 (T-t)^{2{\sf d}_1}
 \sigma^2
 \}
 ds_1
 \\
 &<
 \left[
 \tfrac{1}{\mu_3}
 e^{\mu_3s_1}
 (T-t)^{2{\sf d}_1}
 \sigma^2
 \right]_{s_1=0}^{s_1=s}
 +
 C
 \int_0^s
 \tfrac{\lambda^2}{\mu_3(T-t)}
 e^{\mu_3s_1}
 (T-t)^{2{\sf d}_1}
 \sigma^2
 ds_1.
 \end{align*}
 Since $\tfrac{\lambda^2}{\mu_3(T-t)}\ll1$,
 it holds that
 \begin{align*}
 \int_0^s
 e^{\mu_3s_1}
 (T-t)^{2{\sf d}_1}
 \sigma^2
 ds_1
 <
 C
 \left[
 \tfrac{1}{\mu_3}
 e^{\mu_3s_1}
 (T-t)^{2{\sf d}_1}
 \sigma^2
 \right]_{s_1=0}^{s_1=s}
 \lesssim
 \tfrac{1}{\mu_3}
 e^{\mu_3s}
 (T-t)^{2{\sf d}_1}
 \sigma^2.
 \end{align*}
 Therefore
 from Lemma \ref{Lem3.3},
 we obtain
 \begin{align}\label{eq6.10}
 \|\epsilon_{3}^{\bot}\|_{L_y^2(B_{4{\sf R}_{\sf in}})}
 \lesssim
 \tfrac{1}{\mu_3}
 (T-t)^{{\sf d}_1}
 \sigma
 \lesssim
 {\sf R}_{\sf in}^\frac{n}{2}
 (T-t)^{{\sf d}_1}
 \sigma
 \qquad
 \text{for }
 t\in(0,T).
 \end{align}
 We now go back to \eqref{eq6.4}.
 We note from \eqref{eq6.9} - \eqref{eq6.10} and  Lemma \ref{Lem3.5} that
 \begin{align*}
 |\epsilon_3(y,t)|
 +
 |\nabla_y\epsilon_3(y,t)|
 &\lesssim
 ({\sf R}_{\sf in}^{6-n}+{\sf R}_{\sf in}^\frac{n}{2})(T-t)^{{\sf d}_1}
 \qquad
 \text{for }
 |y|<4{\sf R}_{\sf in},\ t\in(0,T).
 \end{align*}
 Furthermore
 we recall that
 (see Lemma \ref{Lem3.1})
 \[
 \epsilon_0
 =
 \alpha_1
 \psi_1
 \lesssim
 (T-t)^{{\sf d}_1}
 \sigma
 |_{t=0}
 |y|^{-\frac{n-1}{2}}
 e^{-\sqrt{|\mu_1|}\cdot|y|}
 \qquad
 \text{for }
 |y|>1.
 \]
 Since
 $\lambda^2\pa_t\epsilon_3=H_y\epsilon_3$ for $|y|>2{\sf m}_1$,
 a comparison argument shows
 \begin{equation}\label{eq6.11}
 |\epsilon_3(y,t)|
 \lesssim
 {\sf m}_1^{n-\frac{9}{4}}
 {\sf R}_{\sf in}^\frac{n}{2}
 |y|^{-(n-\frac{9}{4})}
 \qquad
 \text{for }
 2{\sf m}_1<|y|<4{\sf R}_{\sf in},\
 t\in(0,T).
 \end{equation}
 Furthermore
 we note that
 $\epsilon_{3,r}=\pa_r\epsilon_3$ satisfies
 $\pa_s\epsilon_{3,r}
 =H_y\epsilon_{3,r}-(n-1)|y|^{-2}\epsilon_{3,r}+(\pa_rV)\epsilon_3$
 for $|y|>2{\sf m}_1$.
 Due to \eqref{eq6.11},
 a comparison argument shows
 \begin{equation*}
 |\nabla_y\epsilon_3(y,t)|
 \lesssim
 {\sf m}_1^{n-\frac{5}{4}}
 {\sf R}_{\sf in}^\frac{n}{2}
 |y|^{-(n-\frac{5}{4})}
 \qquad
 \text{for }
 2{\sf m}_1<|y|<4{\sf R}_{\sf in},\
 t\in(0,T).
 \end{equation*}
 Therefore
 since ${\sf m}_1$ depends only on $n$,
 we conclude
 \begin{align}
 \label{eq6.12}
 |\epsilon(y,t)|
 &\lesssim
 {\sf R}_{\sf in}^\frac{n}{2}
 (T-t)^{{\sf d}_1}
 \sigma
 (1+|y|^2)^{-\frac{1}{2}(n-\frac{9}{4})}
 \qquad
 \text{for }
 (y,t)\in\overline{B}_{4{\sf R}_{\sf in}}\times[0,T],
 \\
 \label{eq6.13}
 |\nabla_y\epsilon(y,t)|
 &\lesssim
 {\sf R}_{\sf in}^\frac{n}{2}
 (T-t)^{{\sf d}_1}
 \sigma
 (1+|y|^2)^{-\frac{1}{2}(n-\frac{5}{4})}
 \qquad
 \text{for }
 (y,t)\in\overline{B}_{4{\sf R}_{\sf in}}\times[0,T].
 \end{align}

 \section{In the semiinner region}
 \label{sec_7}
 For simplicity,
 we define
 \begin{align*}
 L_{\xi}v
 =
 \Delta_\xi v
 -
 q{\sf U}^{q-1}
 (1-\chi_1)
 v
 -
 \eta\dot\eta
 (\Lambda_\xi v),
 \qquad
 \text{where }
 \Lambda_\xi=\tfrac{2}{1-q}-\xi\cdot\nabla_\xi.
 \end{align*}
 In order to construct a solution $v(\xi,t)$ of
 the second equation in \eqref{eq5.2},
 we here consider
 \begin{align}\label{eq7.1}
 \begin{cases}
 \eta^2v_t
 =
 L_\xi v
 +
 \eta^{-\frac{2q}{1-q}}
 \lambda^{-\frac{n+2}{2}}
 (\lambda\dot\lambda-\sigma)
 (\Lambda_y{\sf Q})
 \chi_1(1-\chi_{\sf in})
 \\
 \qquad
 \quad
 +
 \eta^{-\frac{2q}{1-q}}
 (
 F_1[\tilde v_1,\tilde w_1]
 (1-\chi_{{\sf in}})
 \chi_{{\sf sq}}
 -
 F_2[\tilde w_1]
 \chi_{{\sf sq}}
 +
 h[\epsilon]
 )
 \\
 \qquad
 \quad
 +
 \eta^{-\frac{2q}{1-q}}
 (
 g
 +
 N[\epsilon,\tilde v_1,\tilde w_1]
 )
 {\bf 1}_{|\xi|<1}
 &
 \text{in } |\xi|<4{\sf R}_{\sf mid},\ t\in(0,T),
 \\
 v=0
 &
 \text{on } |\xi|=4{\sf R}_{\sf mid},\ t\in(0,T),
 \\
 v=0
 &
 \text{in } |\xi|<4{\sf R}_{\sf mid},\ t=0.
 \end{cases}
 \end{align}
 We remark that
 a construction of solutions to \eqref{eq7.1}
 satisfying $v(\xi,t)\to0$ as $t\to T$
 is simpler than that of $\epsilon(y,t)$,
 thanks to the stability of the trivial solution of
 $v_t=\Delta_\xi v-q{\sf U}^{q-1}v$.
 A goal of this section is to construct a solution of \eqref{eq7.1}
 satisfying $v\in X_2\subset X_2^{(\delta)}$.
 To do that,
 we need to compute all terms on the right-hand side of \eqref{eq7.1}.
 We collect them in  Lemma \ref{Lem7.1} - Lemma \ref{Lem7.3}.
 Their proofs are postponed to Appendix.
 Throughout this section,
 \begin{itemize}
 \item
 $(\lambda(t),\epsilon(y,t))$ are solutions obtained
 in Section {\rm\ref{sec_6}} from $(w_1,v_1)\in X^{(\delta)}$,
 \item
 $(\tilde w_1,\tilde v_1)\in X$ are extensions of $(w_1,v_1)\in X^{(\delta)}$,
 namely $(\tilde w_1,\tilde v_1)=\mathcal{T}_1(w_1,v_1)$.
 \end{itemize}
 \begin{lem}\label{Lem7.1}
 There exist positive constants
 ${\sf b}_1,{\sf c}_0,{\sf c}_1,{\sf c}_2$
 depending only on $q,n,J$
 {\rm(}independent of
 $\delta$,
 ${\sf d}_1$,
 ${\sf b}$,
 ${\sf R}_{\sf in}$,
 ${\sf R}_{\sf mid}$,
 ${\sf R}_1$,
 ${\sf r}_0$,
 ${\sf r}_3${\rm)}
 such that
 if
 ${\sf b}\in(0,{\sf b}_1)$,
 then it holds that
 \begin{align*}
 \sum_{i=0}^6
 |g_i|
 {\bf 1}_{|\xi|<1}
 +
 \sum_{i=0}^{10}
 |g_i'|
 {\bf 1}_{|\xi|<1}
 &\leq
 {\sf c}_0
 (T-t)^{{\sf c}_1}
 \lambda^{-\frac{n+2}{2}}
 \sigma
 (1+|y|^2)^{-(1+{\sf c}_2)}
 {\bf 1}_{|\xi|<1}.
 \end{align*}
 \end{lem}

 \begin{lem}\label{Lem7.2}
 There exist positive constants
 ${\sf b}_1,{\sf c}_0,{\sf c}_1,{\sf c}_2$
 depending only on $q,n,J$
 such that
 if
 ${\sf b}\in(0,{\sf b}_1)$,
 then
 \begin{align*}
 N_1[\epsilon,\tilde v_1,\tilde w_1]
 {\bf 1}_{|y|<{\sf l}_1}
 &\leq
 {\sf c}_0
 (T-t)^{{\sf c}_1}
 \lambda^{-\frac{n+2}{2}}
 \sigma
 (1+|y|^2)^{-(1+{\sf c}_2)}
 {\bf 1}_{|y|<{\sf l}_1},
 \\
 N_1[\epsilon,\tilde v_1,\tilde w_1]
 {\bf 1}_{|y|>{\sf l}_1}
 {\bf 1}_{|\xi|<1}
 &\leq
 {\sf c}_0
 (T-t)^{{\sf c}_1}
 \eta^\frac{2q}{1-q}
 {\bf 1}_{|y|>{\sf l}_1}
 {\bf 1}_{|\xi|<1},
 \\
 N_2[\epsilon,v,w]
 {\bf 1}_{|\xi|<(T-t)^{{\sf q}_1}}
 &\leq
 {\sf c}_0
 (T-t)^{{\sf c}_1}
 \lambda^{-\frac{n+2}{2}}
 \sigma
 |y|^{-2(1+{\sf c}_2)}
 {\bf 1}_{|y|>{\sf l}_1}
 {\bf 1}_{|\xi|<(T-t)^{{\sf q}_1}},
 \\
 N_2[\epsilon,v,w]
 {\bf 1}_{(T-t)^{{\sf q}_1}<|\xi|<1}
 &\leq
 {\sf c}_0
 (T-t)^{2{\sf d}_1}
 \eta^\frac{2q}{1-q}
 {\bf 1}_{(T-t)^{{\sf q}_1}<|\xi|<1},
 \\
 N_3[\epsilon,\tilde v_1,\tilde w_1]
 {\bf 1}_{|\xi|<1}
 &=
 0,
 \\
 N_4[\epsilon,\tilde v_1,\tilde w_1]
 &\leq
 {\sf c}_0
 (T-t)^{{\sf c}_1}
 \lambda^{-\frac{n+2}{2}}
 \sigma
 (1+|y|^2)^{-(1+{\sf c}_2)}
 {\bf 1}_{|y|<2{\sf l}_2},
 \\
 N_5[\tilde w_1]
 {\bf 1}_{|\xi|<1}
 &=
 0,
 \\
 N_6[\tilde w_1]
 {\bf 1}_{|\xi|<1}
 &=
 0.
 \end{align*}
 \end{lem}
%
%
 \begin{lem}\label{Lem7.3}
 There exist positive constants
 ${\sf b}_1,{\sf c}_0,{\sf c}_1,{\sf c}_2$
 depending only on $q,n,J$
 such that
 if
 ${\sf b}\in(0,{\sf b}_1)$,
 then
 \begin{align*}
 F_1[\tilde v_1,\tilde w_1]
 (1-\chi_{\sf in})
 \chi_{\sf sq}
 {\bf 1}_{|\xi|<1}
 &\leq
 {\sf c}_0
 (T-t)^{{\sf c}_1}
 \lambda^{-\frac{n+2}{2}}
 \sigma
 |y|^{-2(1+{\sf c}_2)}
 {\bf 1}_{|y|>{\sf R}_{\sf in}}
 {\bf 1}_{|\xi|<1},
 \\
 F_1[\tilde v_1,\tilde w_1]
 (1-\chi_{\sf in})
 \chi_{\sf sq}
 {\bf 1}_{|\xi|>1}
 &\leq
 {\sf c}_0
 (T-t)^{{\sf c}_1}
 \eta^\frac{2q}{1-q}
 |\xi|^{\gamma-2-2{\sf c}_2}
 {\bf 1}_{1<|\xi|<2\sqrt{{\sf R}_{\sf mid}}},
 \\
 F_2[\tilde w_1]
 \chi_{\sf sq}
 {\bf 1}_{|\xi|<1}
 &\leq
 {\sf R}_1^{-1}
 (T-t)^{{\sf d}_1}
 \eta^\frac{2q}{1-q}
 |\xi|^{\gamma-2}
 {\bf 1}_{|\xi|<1},
 \\
 F_2[\tilde w_1]
 \chi_{\sf sq}
 {\bf 1}_{|\xi|>1}
 &\leq
 {\sf R}_1^{-1}
 (T-t)^{{\sf d}_1}
 \eta^\frac{2q}{1-q}
 |\xi|^{\gamma-2-2{\sf c}_2}
 {\bf 1}_{1<|\xi|<2\sqrt{{\sf R}_{\sf mid}}}.
 \end{align*}
 \end{lem}
 To clarify the role of each term on the right-hand side of \eqref{eq7.1},
 we divide them into four parts and introduce
 \begin{align}
 \label{eq7.2}
 &
 \begin{cases}
 \eta^2
 \pa_t
 v_1
 =
 L_\xi
 v_1
 +
 \eta^{-\frac{2q}{1-q}}
 G_1
 &
 \text{in } |\xi|<4{\sf R}_{\sf mid},\ t\in(0,T),
 \\
 v_1
 =
 0
 &
 \text{on } |\xi|=4{\sf R}_{\sf mid},\ t\in(0,T),
 \\
 v_1
 =
 0
 &
 \text{in } |\xi|<4{\sf R}_{\sf mid},\ t=0,
 \end{cases}
 \\
 &
 \label{eq7.3}
 \begin{cases}
 \eta^2
 \pa_t
 v_2
 =
 L_\xi
 v_2
 +
 \eta^{-\frac{2q}{1-q}}
 G_2
 &
 \text{in } |\xi|<2,\ t\in(0,T),
 \\
 v_2=0
 &
 \text{on } |\xi|=2,\ t\in(0,T),
 \\
 v_2
 =
 0
 &
 \text{in } |\xi|<2,\ t=0,
 \end{cases}
 \\
 &
 \label{eq7.4}
 \begin{cases}
 \eta^2
 \pa_t
 v_3
 =
 L_\xi
 v_3
 +
 \eta^{-\frac{2q}{1-q}}
 G_3
 &
 \text{in } |\xi|<4{\sf R}_{\sf mid},\ t\in(0,T),
 \\
 v_3=0
 &
 \text{on } |\xi|=4{\sf R}_{\sf mid},\ t\in(0,T),
 \\
 v_3
 =
 0
 &
 \text{in } |\xi|<4{\sf R}_{\sf mid},\ t=0,
 \end{cases}
 \\
 &
 \label{eq7.5}
 \begin{cases}
 \eta^2
 \pa_t
 v_4
 =
 L_\xi
 v_4
 +
 \eta^{-\frac{2q}{1-q}}
 G_4
 &
 \text{in } \frac{1}{2}{\sf m}_2<|\xi|<4{\sf R}_{\sf mid},\
 t\in(0,T),
 \\
 v_4=0
 &
 \text{on } |\xi|=\frac{1}{2}{\sf m}_2,\ |\xi|=4{\sf R}_{\sf mid},\
 t\in(0,T),
 \\
 v_4
 =
 0
 &
 \text{in } \frac{1}{2}{\sf m}_2<|\xi|<4{\sf R}_{\sf mid},\
 t=0.
 \end{cases}
 \end{align}
 The constant ${\sf m}_2>1$ will be chosen later
 such that $L_\xi\approx\Delta_\xi-q{\sf U}_\infty(\xi)^{q-1}$
 for $|\xi|>{\sf m}_2$,
 and
 $G_i$ ($i=1,2,3,4$) are
 \begin{align*}
 G_1
 &=
 \lambda^{-\frac{n+2}{2}}
 (\lambda\dot\lambda-\sigma)
 (\Lambda_y{\sf Q})
 \chi_1(1-\chi_{\sf in})
 +
 F_1[\tilde v_1,\tilde w_1](1-\chi_{\sf in})
 \chi_{\sf sq}
 {\bf 1}_{|\xi|<1}
 +
 h[\epsilon]
 \\
 & \quad
 +
 g
 {\bf 1}_{|\xi|<1}
 +
 N_1[\epsilon,\tilde v_1,\tilde w_1]
 {\bf 1}_{|y|<{\sf l}_1}
 +
 N_2[\epsilon,\tilde v_1,\tilde w_1]
 {\bf 1}_{|\xi|<(T-t)^{{\sf q}_1}}
 +
 N_4[\epsilon,\tilde v_1,\tilde w_1],
 \\
 G_2
 &=
 -
 F_2[\tilde w_1]
 \chi_{\sf sq}
 {\bf 1}_{|\xi|<1},
 \\
 G_3
 &=
 (
 F_1[\tilde v_1,\tilde w_1]
 (1-\chi_{{\sf in}})
 \chi_{{\sf sq}}
 -
 F_2[\tilde w_1]
 \chi_{{\sf sq}}
 )
 {\bf 1}_{1<|\xi|<{\sf m}_2}
 \\
 & \quad
 +
 N_1[\epsilon,\tilde v_1,\tilde w_1]
 {\bf 1}_{|y|>{\sf l}_1}
 {\bf 1}_{|\xi|<1}
 +
 N_2[\epsilon,\tilde v_1,\tilde w_1]
 {\bf 1}_{(T-t)^{{\sf q}_1}<|\xi|<1},
 \\
 G_4
 &=
 (
 F_1[\tilde v_1,\tilde w_1]
 (1-\chi_{{\sf in}})
 \chi_{{\sf sq}}
 -
 F_2[\tilde w_1]
 \chi_{{\sf sq}}
 )
 {\bf 1}_{|\xi|>{\sf m}_2}.
 \end{align*}
 By using solutions $v_1$, $v_2$, $v_3$ and $v_4$,
 we define $v_5$ as
 \[
 v
 =
 v_1+v_2\chi_A+v_3+v_4\chi_B+v_5,
 \qquad
 \chi_A=
 \begin{cases}
 1 & |\xi|<1 \\
 0 & |\xi|>2,
 \end{cases}
 \qquad
 \chi_B=
 \begin{cases}
 0 & |\xi|<\frac{1}{2}{\sf m}_2 \\
 1 & |\xi|>{\sf m}_2.
 \end{cases}
 \]
 From this definition,
 $v_5$ satisfies
 \begin{align}\label{eq7.6}
 \begin{cases}
 \eta^2
 \pa_t
 v_5
 =
 L_\xi
 v_5
 +
 G_5
 &
 \text{in } |\xi|<4{\sf R}_{\sf mid},\ t\in(0,T),
 \\
 v_5=0
 &
 \text{on } |\xi|=4{\sf R}_{\sf mid},\ t\in(0,T),
 \\
 v_5=0
 &
 \text{in } |\xi|<4{\sf R}_{\sf mid},\ t=0,
 \end{cases}
 \end{align}
 where
 $G_5=
 -
 2\nabla_\xi v_2\cdot\nabla_\xi\chi_A
 -
 v_2\Delta\chi_A
 -
 2\nabla_\xi v_4\cdot\nabla_\xi\chi_B
 -
 v_4\Delta\chi_B$.
 Before constructing solutions $v_i$ ($i=1,2,3,4,5$),
 we prepare the following lemma.
 We now assume
 \begin{itemize}
 \item $T^{\frac{1}{8}{\sf d}_1}<{\sf R}_1^{-1}$.
 \end{itemize}
 Since ${\sf d}_1$ will be determined by constants depending only on $q,n,J$,
 this assumption holds if $T$ is sufficiently small.
 \begin{lem}\label{Lem7.4}
 There exist positive constants
 ${\sf b}_1,{\sf c}_0,{\sf c}_1,{\sf c}_2$
 depending only on $q,n,J$
 such that
 if
 ${\sf b}\in(0,{\sf b}_1)$,
 then
 \begin{align*}
 G_1
 &<
 {\sf c}_0
 {\sf R}_1^{-1}
 (T-t)^{{\sf d}_1}
 \lambda^{-2}
 \eta^\frac{2}{1-q}
 (1+|y|^2)^{-(1+{\sf c}_2)}
 {\bf 1}_{|\xi|<1},
 \\
 G_3
 &<
 {\sf c}_0
 {\sf m}_2^\gamma
 {\sf R}_1^{-1}
 (T-t)^{{\sf d}_1}
 \eta^\frac{2q}{1-q}
 (1+|\xi|^2)^{-(1+{\sf c}_2)}
 {\bf 1}_{|\xi|<{\sf m}_2},
 \\
 G_4
 &<
 {\sf c}_0
 {\sf R}_1^{-1}
 (T-t)^{{\sf d}_1}
 \eta^\frac{2q}{1-q}
 |\xi|^{\gamma-2-2{\sf c}_2}
 {\bf 1}_{|\xi|>{\sf m}_2}.
 \end{align*}
 \end{lem}
 \begin{proof}
 From \eqref{eq6.12} - \eqref{eq6.13},
 we see that
 \begin{align*}
 h_1[\epsilon]
 &=
 -\lambda^{-\frac{n+2}{2}}
 \lambda\dot\lambda
 (\Lambda_y\epsilon)\chi_{\sf in}
 +
 \lambda^{-\frac{n+2}{2}}
 \lambda^2
 \epsilon
 \pa_t\chi_{\sf in}
 \\
 &\lesssim
 {\sf R}_{\sf in}^{\frac{n}{2}}
 \sigma
 \cdot
 (T-t)^{{\sf d}_1}
 \lambda^{-\frac{n+2}{2}}
 \sigma
 {\bf 1}_{|y|<2{\sf R}_{\sf in}}
 \\
 &\lesssim
 {\sf R}_{\sf in}^{\frac{n}{2}+4}
 \sigma
 \cdot
 (T-t)^{{\sf d}_1}
 \lambda^{-\frac{n+2}{2}}
 \sigma
 (1+|y|^2)^{-2}
 {\bf 1}_{|y|<2{\sf R}_{\sf in}},
 \\
 h_1'[\epsilon]
 &=
 2\lambda^{-\frac{n+2}{2}}
 \nabla_y\epsilon\cdot\nabla_y\chi_{\sf in}
 +
 \lambda^{-\frac{n+2}{2}}
 \epsilon\Delta_y\chi_{\sf in}
 \\
 &\lesssim
 {\sf R}_{\sf in}^{\frac{n}{2}}
 (T-t)^{{\sf d}_1}
 \lambda^{-\frac{n+2}{2}}
 \sigma
 |y|^{-2-(n-\frac{9}{4})}
 {\bf 1}_{{\sf R}_{\sf in}<|y|<2{\sf R}_{\sf in}}.
 \end{align*}
 Since $T=e^{-{\sf R}_{\sf in}}$,
 this implies
 \begin{align*}
 h[\epsilon]
 &=
 -h_1[\epsilon]+h_1'[\epsilon]
 \\
 &\lesssim
 {\sf R}_{\sf in}^{-\frac{1}{8}(2n-9)}
 (T-t)^{{\sf d}_1}
 \lambda^{-\frac{n+2}{2}}
 \sigma
 (1+|y|^2)^{-1-\frac{1}{16}(2n-9)}
 {\bf 1}_{|y|<2{\sf R}_{\sf in}}.
 \end{align*}
 Due to this estimate
 and
 Lemma \ref{Lem7.1} - Lemma \ref{Lem7.3},
 it holds that
 \begin{align*}
 G_1
 &=
 \lambda^{-\frac{n+2}{2}}
 (\lambda\dot\lambda-\sigma)
 (\Lambda_y{\sf Q})
 \chi_1(1-\chi_{\sf in})
 +
 F_1[\tilde v_1,\tilde w_1](1-\chi_{\sf in})
 \chi_{\sf sq}
 {\bf 1}_{|\xi|<1}
 +
 h[\epsilon]
 \\
 & \quad
 +
 g
 {\bf 1}_{|\xi|<1}
 +
 N_1[\epsilon,\tilde v_1,\tilde w_1]
 {\bf 1}_{|y|<{\sf l}_1}
 +
 N_2[\epsilon,\tilde v_1,\tilde w_1]
 {\bf 1}_{|\xi|<(T-t)^{{\sf q}_1}}
 +
 N_4[\epsilon,\tilde v_1,\tilde w_1]
 \\
 &<
 {\sf c}_0
 {\sf R}_{\sf in}^{-{\sf c}_1}
 (T-t)^{{\sf d}_1}
 \lambda^{-\frac{n+2}{2}}
 \sigma
 (1+|y|^2)^{-1-{\sf c}_2}
 {\bf 1}_{|y|<2{\sf R}_{\sf in}}
 \end{align*}
 for some ${\sf c}_0,{\sf c}_1,{\sf c}_2>0$ depending only on $q,n,J$.
 This proves the estimate for $G_1$.
 Furthermore
 we use Lemma \ref{Lem7.2} - Lemma \ref{Lem7.3} again
 to get
 \begin{align*}
 G_3
 &=
 (
 F_1[\tilde v_1,\tilde w_1]
 (1-\chi_{{\sf in}})
 \chi_{{\sf sq}}
 -
 F_2[\tilde w_1]
 \chi_{{\sf sq}}
 )
 {\bf 1}_{1<|\xi|<{\sf m}_2}
 \\
 & \quad
 +
 N_1[\epsilon,\tilde v_1,\tilde w_1]
 {\bf 1}_{|y|>{\sf l}_1}
 {\bf 1}_{|\xi|<1}
 +
 N_2[\epsilon,\tilde v_1,\tilde w_1]
 {\bf 1}_{(T-t)^{{\sf q}_1}<|\xi|<1}
 \\
 &\lesssim
 {\sf R}_1^{-1}
 (T-t)^{{\sf d}_1}
 \eta^\frac{2q}{1-q}
 (1+|\xi|^2)^{\frac{\gamma}{2}-1-{\sf c}_2}
 {\bf 1}_{|\xi|<{\sf m}_2}
 +
 (T-t)^{2{\sf d}_1}
 \eta^\frac{2q}{1-q}
 {\bf 1}_{(T-t)^{{\sf q}_1}<|\xi|<1}
 \\
 &\lesssim
 {\sf m}_2^\gamma
 {\sf R}_1^{-1}
 (T-t)^{{\sf d}_1}
 \eta^\frac{2q}{1-q}
 (1+|\xi|^2)^{-(1+{\sf c}_2)}
 {\bf 1}_{|\xi|<{\sf m}_2}
 +
 T^{{\sf d}_1}
 (T-t)^{{\sf d}_1}
 \eta^\frac{2q}{1-q}
 {\bf 1}_{(T-t)^{{\sf q}_1}<|\xi|<1}.
 \end{align*}
 Since $T^{{\sf d}_1}<{\sf R}_1^{-1}$,
 we obtain the estimate for $G_3$.
 The inequality for $G_4$ is obvious
 from Lemma \ref{Lem7.2} - Lemma \ref{Lem7.3}.
 \end{proof}
 Most of solutions $v_i$ ($i=1,2,3,4,5$) are immediately constructed
 from this lemma.
 We put
 \[
 \bar v_1
 =
 \bar{\sf c}_0
 {\sf R}_1^{-1}
 (T-t)^{{\sf d}_1}
 (1+|y|^2)^{-{\sf c}_2}.
 \]
 We now check that
 $\bar v_1$ gives a comparison function of \eqref{eq7.2}.
 \begin{align*}
 (
 \eta^2
 \pa_t
 -
 L_\xi
 )
 \bar v_1
 &=
 (
 \eta^2\pa_t
 -
 \Delta_\xi
 +
 q{\sf U}^{q-1}(1-\chi_1)
 +
 \eta\dot\eta\Lambda_\xi
 )
 \bar v_1
 \\
 &>
 (
 \eta^2\pa_t
 -
 \Delta_\xi
 +
 \eta\dot\eta\Lambda_\xi
 )
 \bar v_1
 \\
 &>
 (
 -
 C
 \eta^2
 (T-t)^{-1}
 +
 2{\sf c}_2(n-2-2{\sf c}_2)
 (\lambda^{-1}\eta)^2
 (1+|y|^2)^{-1}
 -
 C
 \eta\dot\eta
 )
 \bar v_1
 \\
 &>
 \eta^2
 (T-t)^{-1}
 \{
 2
 {\sf c}_2(n-2-2{\sf c}_2)
 (T-t)
 \lambda^{-2}
 (1+|y|^2)^{-1}
 -
 C
 \}
 \bar v_1.
 \end{align*}
 From this estimate,
 we easily see that
 \begin{align*}
 (
 \pa_t
 -
 L_\xi
 )
 \bar v_1
 &>
 \tfrac{1}{2}
 {\sf c}_2(n-2-2{\sf c}_2)
 \eta^2
 \lambda^{-2}
 \bar v_1
 \qquad
 \text{for } |y|<1,
 \\
 (
 \pa_t
 -
 L_\xi
 )
 \bar v_1
 &>
 \eta^2
 (T-t)^{-1}
 \{
 2
 {\sf c}_2(n-2-2{\sf c}_2)
 |z|^{-2}
 -
 {\sf c}
 \}
 \bar v_1
 \\
 &>
 {\sf c}_2(n-2-2{\sf c}_2)
 \eta^2
 \lambda^{-2}
 |y|^{-2}
 \bar v_1
 \qquad
 \text{for }
 |y|>1,\ |\xi|<4{\sf R}_{\sf mid}.
 \end{align*}
 This together with Lemma \ref{Lem7.4} implies
 $(\pa_t-L_\xi)\bar v_1>\eta^{-\frac{2q}{1-q}}|G_1|$
 for $|\xi|<4{\sf R}_{\sf mid}$.
 Hence
 a comparison argument shows
 \begin{align*}
 |v_1(\xi,t)|
 <
 \bar v_1
 &\lesssim
 {\sf R}_1^{-1}
 (T-t)^{{\sf d}_1}
 (1+|y|^2)^{-{\sf c}_2}
 \qquad
 \text{for }
 |\xi|<4{\sf R}_{\sf mid},\
 t\in(0,T).
 \end{align*}
 This implies
 $|v_1(\xi,t)|
 \lesssim
 {\sf R}_1^{-1}(T-t)^{{\sf d}_1}$
 for $|\xi|<4{\sf R}_{\sf mid}$, $t\in(0,T)$.
 Since $G_1=0$ for $|\xi|>2$,
 by a comparison argument,
 we get
 \[
 |\nabla_\xi v_1(\xi,t)|
 \lesssim
 {\sf R}_1^{-1}
 (T-t)^{{\sf d}_1}
 |\xi|^{-(n-2)}
 \qquad
 \text{for }
 2<|\xi|<4{\sf R}_{\sf in},\
 t\in(0,T).
 \]
 We next construct $v_2$.
 We multiply \eqref{eq7.3} by $v_2$ and integrate by parts.
 Then we get
 \begin{align}\label{eq7.7}
 \nonumber
 \tfrac{\eta^2}{2}
 \tfrac{d}{dt}
 \|v_2\|_{L_\xi^2(B_2)}^2
 &<
 -
 \|\nabla_\xi v_2\|_{L_\xi^2(B_2)}^2
 -
 \eta\dot\eta
 (\Lambda_\xi v_2,v_2)_{L_\xi^2(B_2)}
 +
 \eta^{-\frac{2q}{1-q}}
 (G_2,v_2)_{L_\xi^2(B_2)}
 \\
 &<
 -\tfrac{{\sf j}_1}{2}\|v_2(t)\|_{L_\xi^2(B_2)}^2
 +
 \eta^{-\frac{4q}{1-q}}
 \|G_2\|_{L_\xi^2(B_2)}^2,
 \end{align}
 where ${\sf j}_1>0$ is the first eigenvalue of
 $-\Delta_\xi e=je$ in $B_2$ with zero Dirichlet boundary condition.
 We remark that
 $G_2=-F_2[\tilde w_1]\approx q{\sf U}_\infty^{q-1}w$
 is a singular inhomogeneous term at the origin
 in \eqref{eq7.3}.
 Since
 $G_2\lesssim
 {\sf R}_1^{-1}(T-t)^{{\sf d}_1}
 \eta^{\frac{2q}{1-q}}|\xi|^{-2+\gamma}{\bf 1}_{|\xi|<1}$
 (see Lemma \ref{Lem7.3})
 and
 $|\xi|^{-2}\in L_\xi^2(B_1)$ if $n\geq5$,
 \eqref{eq7.7} implies
 $\|v_2(t)\|_{L_\xi^2(B_2)}
 \lesssim
 {\sf R}_1^{-1}
 (T-t)^{{\sf d}_1}$
 for $t\in(0,T)$.
 Furthermore
 since $|\xi|^{-2+\gamma}\in L_\xi^p(B_1)$ for some $p>\frac{n}{2}$,
 Lemma \ref{Lem3.4} shows that
 $v_2(\xi,t)$ is bounded near the origin and satisfies
 \begin{align}\label{eq7.8}
 |v_2(\xi,t)|
 &\lesssim
 {\sf R}_1^{-1}
 (T-t)^{{\sf d}_1}
 \qquad
 \text{for }
 |\xi|<{2},\ t\in(0,T).
 \end{align}
 Estimates for $v_3$ can be derived in the same way as $v_1$.
 \begin{align*}
 |v_3(\xi,t)|
 &\lesssim
 {\sf m}_2^{\gamma}
 {\sf R}_1^{-1}
 (T-t)^{{\sf d}_1}
 (1+|\xi|^2)^{-{\sf c}_2}
 \qquad
 \text{for }
 |\xi|<4{\sf R}_{\sf mid},\
 t\in(0,T),
 \\
 |\nabla_\xi v_3(\xi,t)|
 &\lesssim
 {\sf m}_2^{\gamma}
 {\sf R}_1^{-1}
 (T-t)^{{\sf d}_1}
 {\sf m}_2^{-{\sf c}_2}
 {\sf m}_2^{n-2}
 |\xi|^{-(n-2)}
 \qquad
 \text{for }
 {\sf m}_2<|\xi|<4{\sf R}_{\sf mid},\
 t\in(0,T).
 \end{align*}
 From the definition of $\gamma$ (see Section \ref{sec_4.2}),
 for any ${\sf k}\in(0,n-2)$,
 there exists ${\sf c}_A({\sf k})>0$ such that
 $(\Delta_\xi-q{\sf U}_\infty(\xi)^{q-1})|\xi|^{\gamma-{\sf k}}
 =-{\sf c}_A({\sf k})|\xi|^{\gamma-{\sf k}-2}$.
 Hence
 form \eqref{eq4.10},
 we can choose ${\sf m}_2>1$ such that
 $(\Delta_\xi-q{\sf U}^{q-1})|\xi|^{\gamma-{\sf k}}
 <-\tfrac{1}{2}{\sf c}_A({\sf k})|\xi|^{\gamma-{\sf k}-2}$
 for $|\xi|>\tfrac{1}{2}{\sf m}_2$.
 From this relation, 
 we observe that
 \begin{align*}
 (\eta^2\pa_t-L_\xi)
 \{
 (T-t)^{{\sf d}_1}
 |\xi|^{\gamma-2{\sf c}_2}
 \}
 &>
 \{
 -
 {\sf d}_1
 \eta^2
 (T-t)^{-1}
 +
 \tfrac{1}{2}
 {\sf c}_A({\sf k})
 \}
 (T-t)^{{\sf d}_1}|\xi|^{\gamma-2-2{\sf c}_1}
 \\
 &>
 \tfrac{1}{4}
 {\sf c}_A({\sf k})
 (T-t)^{{\sf d}_1}|\xi|^{\gamma-2-2{\sf c}_1}
 \qquad
 \text{for }
 \tfrac{1}{2}{\sf m}_2<|\xi|<4{\sf R}_{\sf mid},\
 t\in(0,T).
 \end{align*}
 Hence
 by using
 $\bar v_4=\bar c_0{\sf R}_1^{-1}(T-t)^{{\sf d}_1}|\xi|^{\gamma-2{\sf c}_2}$
 as a comparison function,
 we get from Lemma \ref{Lem7.4} that
 \begin{align}\label{eq7.9}
 |v_4(\xi,t)|
 &\lesssim
 {\sf R}_1^{-1}
 (T-t)^{{\sf d}_1}
 |\xi|^{\gamma-2{\sf c}_2}
 \qquad
 \text{for }
 \tfrac{1}{2}{\sf m}_2<|\xi|<4{\sf R}_{\sf mid},\
 t\in(0,T).
 \end{align}
 From this estimate,
 it holds that
 $|\nabla_yv_4(\xi,t)|
 \lesssim
 {\sf R}_1^{-1}
 (T-t)^{{\sf d}_1}
 {\sf R}_{\sf mid}^{\frac{1}{2}(\gamma-2{\sf c}_2)}$
 for
 $|\xi|=2\sqrt{{\sf R}_{\sf mid}}$,
 $t\in(0,T)$.
 We note that
 $G_4=0$
 for $|\xi|>2\sqrt{{\sf R}_{\sf mid}}$.
 Hence 
 by a comparison argument in $2\sqrt{{\sf R}_{\sf mid}}<|\xi|<4{\sf R}_{\sf mid}$,
 we get
 \begin{align*}
 |\nabla_yv_4(\xi,t)|
 &\lesssim
 {\sf R}_1^{-1}
 (T-t)^{{\sf d}_1}
 {\sf R}_{\sf mid}^{\frac{1}{2}(\gamma-2{\sf c}_2)}
 {\sf R}_{\sf mid}^{\frac{1}{2}(n-2)}
 |\xi|^{-(n-2)}
 \\
 &\lesssim
 {\sf R}_1^{-1}
 (T-t)^{{\sf d}_1}
 {\sf R}_{\sf mid}^{-{\sf c}_2}
 {\sf R}_{\sf mid}^{\frac{1}{2}(n-2)}
 |\xi|^{\gamma-(n-2)}
 \end{align*}
 for
 $2\sqrt{{\sf R}_{\sf mid}}<|\xi|<4{\sf R}_{\sf mid}$,
 $t\in(0,T)$.
 We finally construct $v_5$.
 We easily see from \eqref{eq7.8} - \eqref{eq7.9} that
 \begin{align*}
 G_5
 &=
 2\nabla_\xi v_2\cdot\nabla_\xi\chi_A
 +
 v_2\Delta\chi_A
 +
 2\nabla_\xi v_4\cdot\nabla_\xi\chi_B
 +
 v_4\Delta\chi_B
 \\
 &\lesssim
 {\sf R}_1^{-1}
 (T-t)^{{\sf d}_1}
 {\bf 1}_{1<|\xi|<2}
 +
 {\sf R}_1^{-1}
 (T-t)^{{\sf d}_1}
 {\sf m}_2^{\gamma-2{\sf c}_2-2}
 {\bf 1}_{\frac{1}{2}{\sf m}_2<|\xi|<{\sf m}_2}
 \\
 &\lesssim
 {\sf m}_2^{\gamma}
 {\sf R}_1^{-1}
 (T-t)^{{\sf d}_1}
 {\bf 1}_{1<|\xi|<{\sf m}_2}.
 \end{align*}
 Hence
 applying a comparison argument to \eqref{eq7.6},
 we obtain
 \begin{align*}
 |v_5(\xi,t)|
 &\lesssim
 {\sf m}_2^{\gamma}
 {\sf R}_1^{-1}
 (T-t)^{{\sf d}_1}
 (1+|\xi|^2)^{-\frac{1}{2}(n-3)}
 \qquad
 \text{for }
 |\xi|<4{\sf R}_{\sf mid},\
 t\in(0,T),
 \\
 |\nabla_\xi v_5(\xi,t)|
 &\lesssim
 {\sf m}_2^{\gamma}
 {\sf R}_1^{-1}
 (T-t)^{{\sf d}_1}
 |\xi|^{-(n-3)}
 \qquad
 \text{for }
 2{\sf m}_2<|\xi|<4{\sf R}_{\sf mid},\
 t\in(0,T).
 \end{align*}
 Therefore
 since ${\sf m}_2$ depends only on $n$,
 we conclude
 \begin{align}
 \label{eq7.10}
 |v(\xi,t)|
 &\lesssim
 {\sf R}_1^{-1}
 (T-t)^{{\sf d}_1}
 (1+|\xi|^2)^{\frac{\gamma}{2}-{\sf c}_2}
 \qquad
 \text{for }
 |\xi|<4{\sf R}_{\sf mid},\
 t\in(0,T),
 \\
 \label{eq7.11}
 |\nabla_\xi v(\xi,t)|
 &\lesssim
 {\sf R}_1^{-1}
 (T-t)^{{\sf d}_1}
 {\sf R}_{\sf mid}^{-{\sf c}_2}
 {\sf R}_{\sf mid}^{\frac{1}{2}(n-2)}
 |\xi|^{\gamma-(n-2)}
 \qquad
 \text{for }
 2\sqrt{{\sf R}_{\sf in}}<|\xi|<4{\sf R}_{\sf mid},\
 t\in(0,T).
 \end{align}
 Estimate \eqref{eq7.10} implies $v\in X_2\subset X_2^{(\delta)}$,
 which is the desired result in this section.
 Gradient estimate \eqref{eq7.11} is used in Section \ref{sec_8}.

 \section{In the selfsimilar region}
 \label{sec_8}
 We study the asymptotic behavior of solutions in the selfsimilar region:
 $|x|\sim\sqrt{T-t}$.
 We recall that
 our solution $u(x,t)$ behaves like
 \[
 u(x,t)
 =
 {\sf U}_\infty(x)
 -
 \theta(x)
 -
 \Theta(x,t)
 +
 w(x,t)
 \qquad
 \text{in }
 |x|\sim\sqrt{T-t}.
 \]
 We here look for solutions of
 (see \eqref{eq5.2})
 \begin{align}\label{eq8.1}
 \begin{cases}
 w_t
 =
 \Delta_xw
 -
 q
 {\sf U}_\infty^{q-1}
 w
 +
 (
 F_1[\tilde v_1,\tilde w_1]
 -
 F_2[\tilde w_1]
 )
 (1-\chi_{{\sf sq}})
 \\
 \qquad
 \quad
 +
 k[v]
 +
 (
 g
 +
 N[\epsilon,\tilde v_1,\tilde w_1]
 )
 {\bf 1}_{|\xi|>1}
 &
 \text{in } x\in\R^n,\
 t\in(0,T),
 \\
 w=w_0
 &
 \text{in } x\in\R^n,\
 t=0.
 \end{cases}
 \end{align}
 The initial data $w_0(x)$ is chosen later such that $w(x,t)$
 decays enough as $t\to T$.
 We remark that
 our approach in this section is simpler than
 that of Section 5 in \cite{Seki}.
 In fact,
 we do not need the explicit expression of the heat kernel of
 $w_t=\Delta_xw-q{\sf U}_\infty^{q-1}w$.
 A goal of this section is to construct a solution
 $w(x,t)\in X_1\subset X_1^{(\delta)}$.
 Combining \eqref{eq7.10},
 we obtain $(w,v)\in X\subset X^{\delta}$.
 This proves Theorem \ref{Thm1}.
 As in Section \ref{sec_7},
 we first provide estimates of each term on the right-hand side \eqref{eq8.1}.
 We postpone the proofs of Lemma \ref{Lem8.1} - Lemma \ref{Lem8.3}
 to Appendix.
 Throughout this section,
 \begin{itemize}
 \item
 $(\lambda(t),\epsilon(y,t))$ are solutions obtained
 in Section {\rm\ref{sec_6}} from $(w_1,v_1)\in X^{(\delta)}$,
 \item
 $v(\xi,t)$ is a solution of \eqref{eq7.1} obtained in Section {\rm\ref{sec_7}},
 \item
 $(\tilde w_1,\tilde v_1)\in X$ are extensions of $(w_1,v_1)\in X^{(\delta)}$,
 namely $(\tilde w_1,\tilde v_1)=\mathcal{T}_1(w_1,v_1)$.
 \end{itemize}
 \begin{lem}\label{Lem8.1}
 There exist positive constants
 ${\sf b}_1,{\sf c}_0,{\sf c}_1,{\sf c}_2$
 depending only on $q,n,J$
 {\rm(}independent of
 $\delta$,
 ${\sf d}_1$,
 ${\sf b}$,
 ${\sf R}_{\sf in}$,
 ${\sf R}_{\sf mid}$,
 ${\sf R}_1$,
 ${\sf r}_0$,
 ${\sf r}_3${\rm)}
 such that
 if
 ${\sf b}\in(0,{\sf b}_1)$,
 then it holds that
 \begin{align*}
 \sum_{i=0}^6
 |g_i|
 {\bf 1}_{|\xi|>1}
 +
 \sum_{i=0}^{10}
 |g_i'|
 {\bf 1}_{|\xi|>1}
 &\leq
 {\sf c}_0
 (T-t)^{{\sf c}_1}
 \eta^{\frac{2q}{1-q}}
 |\xi|^{\gamma-2-2{\sf c}_2}
 {\bf 1}_{1<|\xi|<2{\sf l}_2},
 \\
 g_{{\sf out},1}'
 +
 g_{{\sf out},2}'
 +
 g_{{\sf out},3}'
 &\lesssim
 {\bf 1}_{1<|x|<2},
 \\
 g_{{\sf out},4}'
 +
 g_{{\sf out},5}'
 +
 g_{{\sf out},6}'
 +
 g_{{\sf out},7}'
 &\lesssim
 |x|^\gamma
 {\bf 1}_{{\sf r}_3<|x|<2{\sf r}_3}.
 \end{align*}
 \end{lem}

 \begin{lem}\label{Lem8.2}
 There exist positive constants
 ${\sf b}_1,{\sf c}_0,{\sf c}_1,{\sf c}_2$
 depending only on $q,n,J$
 such that
 if
 ${\sf b}\in(0,{\sf b}_1)$,
 then
 \begin{align*}
 N_1[\epsilon,\tilde v_1,\tilde w_1]
 {\bf 1}_{|\xi|>1}
 &\leq
 {\sf c}_0
 (T-t)^{{\sf c}_1}
 \eta^\frac{2q}{1-q}
 |\xi|^{\gamma-2-2{\sf c}_2}
 {\bf 1}_{1<|\xi|<2{\sf l}_2},
 \\
 N_2[\epsilon,\tilde v_1,\tilde w_1]
 {\bf 1}_{|\xi|>1}
 &\leq
 {\sf c}_0
 (T-t)^{2{\sf d}_1}
 \eta^\frac{2q}{1-q}
 |\xi|^{\gamma-2-2{\sf c}_2}
 {\bf 1}_{1<|\xi|<2{\sf l}_2},
 \\
 N_3[\epsilon,\tilde v_1,\tilde w_1]
 {\bf 1}_{|z|<1}
 &\leq
 {\sf c}_0
 (T-t)^{{\sf c}_1}
 \eta^\frac{2q}{1-q}
 |\xi|^{\gamma-2-2{\sf c}_2}
 {\bf 1}_{|\xi|>{\sf l}_2}
 {\bf 1}_{|z|<1},
 \\
 N_3[\epsilon,\tilde v_1,\tilde w_1]
 {\bf 1}_{1<|z|<(T-t)^{-\frac{1}{3}}}
 &\leq
 {\sf c}_0
 (T-t)^{-1+\frac{\gamma}{2}+J+{\sf c}_1}
 |z|^{\gamma+2J}
 {\bf 1}_{1<|z|<(T-t)^{-\frac{1}{3}}},
 \\
 N_3[\epsilon,\tilde v_1,\tilde w_1]
 {\bf 1}_{(T-t)^{-\frac{1}{3}}<|z|<{\sf l}_{\sf out}}
 &\leq
 {\sf c}_0
 (T-t)^{-1+\frac{\gamma}{2}+J+{\sf d}_1+{\sf c}_1}
 |z|^{\gamma+2J+3{\sf d}_1}
 {\bf 1}_{(T-t)^{-\frac{1}{3}}<|z|<{\sf l}_{\sf out}},
 \\
 N_3[\epsilon,\tilde v_1,\tilde w_1]
 {\bf 1}_{|z|>{\sf l}_{\sf out}}
 &\leq
 {\sf c}_0
 {\sf U}_\infty^q
 {\bf 1}_{{\sf l}_{\sf out}\sqrt{T-t}<|x|<2},
 \\
 N_4[\epsilon,\tilde v_1,\tilde w_1]
 {\bf 1}_{|\xi|>1}
 &=
 0,
 \\
 N_5[\tilde w_1]
 &\leq
 {\sf c}_0
 {\sf R}^{-1}
 |x|^{-3}
 {\bf 1}_{|x|>1},
 \\
 N_6[\tilde w_1]
 &\leq
 {\sf c}_0
 {\bf 1}_{1<|x|<4}
 +
 {\sf c}_0
 {\sf R}^{-1}
 |x|^{-1}
 {\bf 1}_{|x|>1}.
 \end{align*}
 \end{lem}

 \begin{lem}\label{Lem8.3}
 There exist positive constants
 ${\sf b}_1,{\sf c}_0,{\sf c}_1,{\sf c}_2$
 depending only on $q,n,J$
 such that
 if
 ${\sf b}\in(0,{\sf b}_1)$,
 then
 \begin{align*}
 F_1[\tilde v_1,\tilde w_1]
 (1-\chi_{\sf sq})
 &\leq
 {\sf c}_0
 (T-t)^{{\sf c}_1}
 \eta^\frac{2q}{1-q}
 |\xi|^{\gamma-2-2{\sf c}_2}
 {\bf 1}_{\sqrt{{\sf R}_{\sf mid}}<|\xi|<2{\sf l}_2},
 \\
 F_2[\tilde w_1]
 (1-\chi_{\sf sq})
 &\leq
 {\sf c}_0
 {\sf R}_1^{-1}
 (T-t)^{{\sf d}_1}
 \eta^\frac{2q}{1-q}
 |\xi|^{\gamma-2-2{\sf c}_2}
 {\bf 1}_{\sqrt{{\sf R}_{\sf mid}}<|\xi|<2{\sf l}_2}.
 \end{align*}
 \end{lem}
 Let ${\sf c}_1^{\min}$ be the minimum value of
 ${\sf c}_1$ obtained in
 Lemma \ref{Lem7.1} - Lemma \ref{Lem7.3}
 and
 Lemma \ref{Lem8.1} - Lemma \ref{Lem8.3}.
 We take
 \begin{itemize}
 \item 
 ${\sf d}_1=\frac{1}{2}{\sf c}_1^{\min}$.
 \end{itemize}
 As mentioned before,
 this ${\sf d}_1$ depends only on $q,n,J$.
 From \eqref{eq7.10} - \eqref{eq7.11},
 we see that
 \begin{align*}
 k_1[v]
 &=
 \eta^\frac{2}{1-q}
 v
 \pa_t\chi_{\sf mid}
 \lesssim
 (T-t)^{-1}
 \eta^\frac{2}{1-q}
 v
 {\bf 1}_{{\sf R}_{\sf mid}<|\xi|<2{\sf R}_{\sf mid}}
 \\
 &=
 \eta^2
 (T-t)^{-1}
 |\xi|^2
 \cdot
 \eta^\frac{2q}{1-q}
 |\xi|^{-2}
 v
 {\bf 1}_{{\sf R}_{\sf mid}<|\xi|<2{\sf R}_{\sf mid}}
 \\
 &=
 |z|^2
 \cdot
 {\sf R}_1^{-1}
 (T-t)^{{\sf d}_1}
 \eta^\frac{2q}{1-q}
 |\xi|^{\gamma-2-2{\sf c}_2}
 {\bf 1}_{{\sf R}_{\sf mid}<|\xi|<2{\sf R}_{\sf mid}},
 \\
 k_1'[v]
 &=
 2\eta^\frac{2q}{1-q}
 (\nabla_\xi v\cdot\nabla_\xi\chi_{\sf mid})
 +
 \eta^\frac{2q}{1-q}
 v(\Delta_\xi\chi_{\sf mid})
 \\
 &\lesssim
 {\sf R}_{\sf mid}^{-{\sf c}_2}
 {\sf R}_1^{-1}
 (T-t)^{{\sf d}_1}
 \eta^\frac{2q}{1-q}
 |\xi|^{\gamma-2-2{\sf c}_2}
 {\bf 1}_{{\sf R}_{\sf mid}<|\xi|<2{\sf R}_{\sf mid}}.
 \end{align*}
 Hence
 it follows that
 \begin{align}\label{eq8.2}
 k[v]
 =
 -k_1[v]+k_1'[v]
 \lesssim
 {\sf R}_{\sf mid}^{-{\sf c}_2}
 {\sf R}_1^{-1}
 (T-t)^{{\sf d}_1}
 \eta^\frac{2q}{1-q}
 |\xi|^{\gamma-2-2{\sf c}_2}
 {\bf 1}_{{\sf R}_{\sf mid}<|\xi|<2{\sf R}_{\sf mid}}.
 \end{align}
 As in Section \ref{sec_7},
 we divide \eqref{eq8.1} into two equations.
 The first equation is given by
 \begin{align}\label{eq8.3}
 \begin{cases}
 \pa_tw_1
 =
 \Delta_xw_1
 -
 q
 {\sf U}_\infty^{q-1}
 w_1
 +
 (
 F_1[\tilde v_1,\tilde w_1]
 -
 F_2[\tilde w_1]
 )
 (1-\chi_{{\sf sq}})
 \\
 \hspace{12mm}
 +
 k[v]
 +
 g
 {\bf 1}_{|\xi|>1}
 +
 (
 N_1[\epsilon,\tilde v_1,\tilde w_1]
 +
 N_2[\epsilon,\tilde v_1,\tilde w_1]
 )
 {\bf 1}_{|\xi|>1}
 \\
 \hspace{12mm}
 +
 N_3[\epsilon,\tilde v_1,\tilde w_1]
 {\bf 1}_{|z|<1}
 {\bf 1}_{|\xi|>1}
 &
 \text{in } |z|<2{\sf r}_0,\
 t\in(0,T),
 \\
 w=0
 & \text{on } |z|=2{\sf r}_0,\
 t\in(0,T),
 \\
 w=0
 &
 \text{in } |z|<2{\sf r}_0,\
 t=0.
 \end{cases}
 \end{align}
 From Lemma \ref{Lem8.1} - Lemma \ref{Lem8.3} and \eqref{eq8.2},
 the right-hand side of \eqref{eq8.3} can be computed as
 \begin{align*}
 (
 F_1[\tilde v_1,\tilde w_1]
 -
 F_2[\tilde w_1]
 )
 (1-\chi_{{\sf sq}})
 &+
 k[v]
 +
 (
 g
 + 
 \sum_{i=1}^2
 N_i
 [\epsilon,\tilde v_1,\tilde w_1]
 +
 N_3
 [\epsilon,\tilde v_1,\tilde w_1]
 {\bf 1}_{|z|<1}
 )
 {\bf 1}_{|\xi|>1}
 \\
 &\lesssim
 {\sf R}_{\sf mid}^{-{\sf c}_1}
 {\sf R}_1^{-1}
 (T-t)^{{\sf d}_1}
 \eta^\frac{2q}{1-q}
 |\xi|^{\gamma-2-2{\sf c}_2}
 {\bf 1}_{|\xi|>1}
 {\bf 1}_{|z|<1}
 \end{align*}
 for some ${\sf c}_1,{\sf c}_2>0$.
 We introduce a comparison function.
 \[
 \bar w_1
 =
 {\sf R}_{\sf mid}^{-{\sf c}_1}
 {\sf R}_1^{-1}
 (T-t)^{{\sf d}_1}
 \eta^\frac{2}{1-q}
 |\xi|^\gamma
 (1+|\xi|^2)^{-{\sf c}_2}.
 \]
 From the definition of $\gamma$,
 for any ${\sf k}\in(0,\frac{n-2}{2})$,
 there exists ${\sf c}_B({\sf k})>0$ such that
 \[
 (\Delta_\xi-q{\sf U}_\infty(\xi)^{q-1})
 |\xi|^{\gamma}
 (1+|\xi|^2)^{-{\sf k}}
 <
 -
 {\sf c}_B({\sf k})
 |\xi|^{\gamma}
 (1+|\xi|^2)^{-(1+{\sf k})}.
 \]
 Hence
 it follows that
 \begin{align*}
 (\pa_t-\Delta_x+q{\sf U}_\infty^{q-1})
 \bar w_1
 &=
 (\pa_t-\eta^{-2}\Delta_\xi+\eta^{-2}q{\sf U}_\infty(\xi)^{q-1})
 \bar w_1
 \\
 &>
 (-C(T-t)^{-1}
 +
 {\sf c}_B({\sf k})
 \eta^{-2}(1+|\xi|^2)^{-1})
 \bar w_1
 \\
 &=
 (T-t)^{-1}
 (-C
 +
 {\sf c}_B({\sf k})
 (T-t)
 \eta^{-2}
 (1+|\xi|^2)^{-1})
 \bar w_1.
 \end{align*}
 This immediately implies
 \begin{align*}
 (\pa_t-\Delta_x+q{\sf U}_\infty^{q-1})
 \bar w_1
 &>
 \tfrac{1}{4}
 {\sf c}_B({\sf k})
 \eta^{-2}
 \bar w_1
 \qquad
 \text{for } |\xi|<1,
 \\
 (\pa_t-\Delta_x+q{\sf U}_\infty^{q-1})
 \bar w_1
 &>
 (T-t)^{-1}
 (-C+{\sf c}_B({\sf k})|z|^{-2})
 \bar w_1
 \qquad
 \text{for }
 |\xi|>1.
 \end{align*}
 From the second inequality,
 there exists ${\sf r}_0'>0$ depending only on $q,n,J$
 such that
 \begin{align*}
 (\pa_t-\Delta_x+q{\sf U}_\infty^{q-1})
 \bar w_1
 &>
 \tfrac{1}{2}
 {\sf c}_B({\sf k})
 \eta^{-2}
 |\xi|^{-2}
 \bar w_1
 \qquad
 \text{for }
 |\xi|>1,\
 |z|<{\sf r}_0'.
 \end{align*}
 We fix ${\sf r}_0\in(0,\frac{1}{2}{\sf r}_0')$.
 A comparison argument shows
 \begin{align}\label{eq8.4}
 |w_1|
 <
 \bar w_1
 &\lesssim
 {\sf R}_{\sf mid}^{-{\sf c}_1}
 {\sf R}_1^{-1}
 (T-t)^{{\sf d}_1}
 \eta^\frac{2}{1-q}
 |\xi|^{\gamma}
 (1+|\xi|^2)^{-{\sf c}_2}
 \end{align}
 for
 $|z|<2{\sf r}_0,\ t\in(0,T)$.
 Since all terms on the right-hand side of \eqref{eq8.3}
 are zero for $|\xi|>4{\sf l}_2$,
 we apply a comparison argument in $|\xi|>4{\sf l}_2$, $|z|<2{\sf r}_0$,
 then
 \begin{align}\label{eq8.5}
 |w_1|
 +
 |\nabla_\xi w_1|
 &\lesssim
 {\sf R}_{\sf mid}^{-{\sf c}_1}
 {\sf R}_1^{-1}
 (T-t)^{{\sf d}_1}
 \eta^\frac{2}{1-q}
 {\sf l}_2^{-{\sf c}_2+(n-3)}
 |\xi|^{\gamma-(n-3)}
 \end{align}
 for
 $|\xi|>4{\sf l}_2$,
 $|z|<2{\sf r}_0$,
 $t\in(0,T)$.
 We next consider the second equation.
 \begin{align}\label{eq8.6}
 \begin{cases}
 \dis
 \pa_tw_2
 =
 \Delta_xw_2
 -
 q
 {\sf U}_\infty^{q-1}
 w_2
 +
 N_3[\epsilon,\tilde v_1,\tilde w_1]
 {\bf 1}_{|z|>1}
 \\
 \dis
 \qquad
 \quad
 +
 (
 N_4[\epsilon,\tilde v_1,\tilde w_1]
 +
 N_5[\tilde w_1]
 +
 N_6[\tilde w_1]
 )
 {\bf 1}_{|\xi|>1}
 \\
 \dis
 \qquad
 \quad
 -
 (T-t)^{-1}
 w_1
 (\Delta_z\chi_0)
 -
 2
 (T-t)^{-\frac{1}{2}}
 \eta^{-1}
 \nabla_\xi w_1
 \cdot
 \nabla_z\chi_0
 &
 \text{in } x\in\R^n,\
 t\in(0,T),
 \\
 w_2=w_0
 &
 \text{in } x\in\R^n,\
 t=0.
 \end{cases}
 \end{align}
 Once $w_2(x,t)$ is constructed,
 $w(x,t)=w_1\chi_0+w_2(x,t)$ gives a solution of \eqref{eq8.1},
 where $\chi_0=\chi(|z|/{\sf r}_0)$.
 We introduce selfsimilar variables.
 \[
 x=z\sqrt{T-t},
 \qquad
 T-t=e^{-\tau}.
 \]
 Equation \eqref{eq8.6} is written in the selfsimilar variables.
 \begin{align}\label{eq8.7}
 \begin{cases}
 \dis
 \pa_\tau
 w_2
 =
 \Delta_z
 w_2
 -
 \tfrac{z}{2}
 \cdot
 \nabla_z
 w_2
 -
 q
 {\sf U}_\infty(z)^{q-1}
 w_2
 +
 (T-t)G
 &
 \text{in } z\in\R^n,\
 \tau\in(-\log T,\infty),
 \\
 w_2=w_0
 &
 \text{in } z\in\R^n,\
 \tau=-\log T,
 \end{cases}
 \end{align}
 where $G$ represents the right-hand side of \eqref{eq8.6}.
 From Lemma \ref{Lem8.1} - Lemma \ref{Lem8.2} and \eqref{eq8.5},
 there exists a positive constant ${\sf c}_1$ depending only on
 $q,n,J$ such that
 \begin{align}
 \nonumber
 G
 &\lesssim
 (T-t)^{-1+\frac{\gamma}{2}+J+{\sf c}_1}
 |z|^{\gamma}
 (1+|z|^2)^J
 {\bf 1}_{|z|<(T-t)^{-\frac{1}{3}}}
 \\
 \nonumber
 & \quad
 +
 (T-t)^{-1+\frac{\gamma}{2}+J+{\sf d}_1+{\sf c}_1}
 |z|^{\gamma+2J+3{\sf d}_1}
 {\bf 1}_{(T-t)^{-\frac{1}{3}}<|z|<{\sf l}_{\sf out}}
 \\
 \label{eq8.8}
 & \quad
 +
 |x|^\frac{2q}{1-q}
 {\bf 1}_{{\sf l}_{\sf out}\sqrt{T-t}<|x|<4}
 +
 {\sf R}_1^{-1}
 |x|^{-1}
 {\bf 1}_{|x|>1}.
 \end{align}
 From this estimate,
 we easily see that
 \begin{align}\label{eq8.9}
 \|G\|_\rho
 \lesssim
 (T-t)^{-1+\frac{\gamma}{2}+J+{\sf c}_1}.
 \end{align}
 Let $\rho(z)=e^{-\frac{|z|^2}{4}}$ and 
 $L_\rho^2(\R^n)$ be the weighted $L^2$ space defined in Section \ref{sec_4.2}.
 Furthermore
 let $\mu_j=\frac{\gamma}{2}+j$ be the $j$th eigenvalue of
 $-(\Delta_z-\frac{z}{2}\cdot\nabla_z-q{\sf U}_\infty(z)^{q-1})e=\mu e$,
 and $e_j(z)$ be the corresponding eigenfunction (see Section \ref{sec_4.2}).
 We now take
 \[
 w_0
 =
 \sum_{j=0}^J
 \alpha_j
 e_j(z)
 \chi_J
 \qquad
 \text{with }
 \chi_J
 =
 \chi(T^\frac{1}{3}|z|).
 \]
 We expand $e_j(z)\chi_J$ as
 \[
 e_j(z)
 \chi_J
 =
 \sum_{i=0}^J
 (e_j\chi_J,e_i)_\rho
 e_i
 +
 e_{j,J}^\bot.
 \]
 We write ${\sf A}_{ji}=(e_j\chi_J,e_i)_\rho$,
 then
 \begin{align}\label{eq8.10}
 w_0
 =
 \sum_{j=0}^J
 \sum_{i=0}^J
 \alpha_j
 {\sf A}_{ji}
 e_i(z)
 +
 \sum_{j=0}^J
 \alpha_j
 e_{j,J}^\bot.
 \end{align}
 We also expand $w_2$ as
 \begin{align*}
 w_2
 =
 \sum_{j=0}^Ja_j(\tau)e_j(z)
 +
 w_2^\bot
 \qquad
 \text{with }
 a_j(\tau)=(w_2(\cdot,\tau),e_j)_\rho.
 \end{align*}
 From \eqref{eq8.7},
 we verify that
 \begin{align*}
 \pa_\tau
 a_j
 =
 -\mu_j
 a_j
 +
 (T-t)
 (G,e_j)_\rho
 =
 -
 (
 \tfrac{\gamma}{2}+j
 )
 a_j
 +
 (T-t)
 (G,e_j)_\rho.
 \end{align*}
 Integrating both sides,
 we get
 \[
 e^{(\frac{\gamma}{2}+j)\tau}a_j(\tau)
 -
 \underbrace{
 e^{(\frac{\gamma}{2}+j)\tau}a_j|_{\tau=-\log T}
 }_{=T^{-(\frac{\gamma}{2}+j)}\sum_{i=0}^J\alpha_i{\sf A}_{ij}}
 =
 \int_{-\log T}^\tau
 e^{(\frac{\gamma}{2}+j)\tau_1}
 (T-t)
 (G,e_j)_\rho
 d\tau_1.
 \]
 We now chose $\alpha_j$ such that
 \[
 \sum_{i=0}^J
 \alpha_i
 {\sf A}_{ij}
 =
 -
 T^{\frac{\gamma}{2}+j}
 \int_{-\log T}^\infty
 e^{(\frac{\gamma}{2}+j)\tau_1}
 (T-t)
 (G,e_j)_\rho
 d\tau_1
 \qquad
 (j=0,1,2,\cdots,J).
 \]
 Since the matrix ${A}_{ij}$ is close to the identity matrix,
 $\alpha_j$ can be uniquely determined.
 Then
 $a_j(\tau)$ is represented as
 \[
 e^{(\frac{\gamma}{2}+j)\tau}
 a_j(\tau)
 =
 -
 \int_\tau^\infty
 e^{(\frac{\gamma}{2}+j)\tau_1}
 (T-t)
 (G,e_j)_\rho
 d\tau_1.
 \]
 Hence
 it follows from \eqref{eq8.9} that
 \begin{align}\label{eq8.11}
 a_j(\tau)
 \lesssim
 e^{-(\frac{\gamma}{2}+j)\tau}
 e^{(j-J-{\sf c}_1)\tau}
 =
 e^{-(\frac{\gamma}{2}+J+{\sf c}_1)\tau}
 =
 (T-t)^{\frac{\gamma}{2}+J+{\sf c}_1}.
 \end{align}
 Furthermore
 \eqref{eq8.9} shows
 \begin{align}\label{eq8.12}
 \nonumber
 \alpha_j
 &=
 -
 \sum_{i=0}^J
 {\sf A}_{ij}^{-1}
 T^{\frac{\gamma}{2}+j}
 \int_{-\log T}^\infty
 e^{(\frac{\gamma}{2}+j)\tau_1}
 (T-t)
 (G,e_i)_\rho
 d\tau_1
 \\
 &\lesssim
 T^{\frac{\gamma}{2}+j}
 \int_0^T
 (T-t)^{-1+J-j+{\sf c}_1}
 dt
 \lesssim
 T^{\frac{\gamma}{2}+J+{\sf c}_1}.
 \end{align}
 We next provide estimates for $w_2^\bot$.
 Since $w_2$ solves \eqref{eq8.7},
 from \eqref{eq8.9},
 we get
 \begin{align*}
 \tfrac{1}{2}
 \pa_\tau
 \|w_2^\bot\|_\rho^2
 &=
 -\|\nabla_zw_2^\bot\|_\rho^2
 -q
 ({\sf U}_\infty(z)^{q-1}w_2^\bot,w_2^\bot)_\rho
 +
 (T-t)
 (G,w_2^\bot)
 \\
 &<
 -
 \mu_{J+1}
 \|w_2^\bot\|_\rho^2
 +
 C
 (T-t)^{\frac{\gamma}{2}+J+{\sf c}_1}
 \|w_2^\bot\|_\rho.
 \end{align*}
 We recall that
 $\mu_{J+1}=\frac{\gamma}{2}+J+1>\frac{\gamma}{2}+J+\frac{1}{2}{\sf c}_1$
 (see Section \ref{sec_4.2}).
 Hence
 it follows that
 \begin{align}\label{eq8.13}
 e^{2(\frac{\gamma}{2}+J+\frac{1}{2}{\sf c}_1)\tau}
 \|w_2^\bot\|_\rho^2
 -
 \underbrace{
 e^{2(\frac{\gamma}{2}+J+\frac{1}{2}{\sf c}_1)\tau}
 \|w_2^\bot\|_\rho^2
 |_{\tau=-\log T}
 }_{=T^{-2(\frac{\gamma}{2}+J+\frac{1}{2}{\sf c}_1)}
 \|w_0^\bot\|_\rho^2}
 \lesssim
 \int_0^t
 (T-t)^{-1+{\sf c}_1}
 dt
 \lesssim
 T^{{\sf c}_1}.
 \end{align}
 We go back to \eqref{eq8.10} to estimate
 $T^{-2(\frac{\gamma}{2}+J+\frac{1}{2}{\sf c}_1)}
 \|w_0^\bot\|_\rho^2$.
 We note that
 \begin{align*}
 e_{j,J}^\bot
 &=
 e_j\chi_J
 -
 \sum_{i=0}^J
 (e_j\chi_J,e_i)e_i
 \\
 &=
 e_j\chi_J
 -
 (e_j\chi_J,e_j)e_j
 -
 \sum_{i\not=j}
 (e_j(\chi_J-1),e_i)e_i
 \\
 &=
 e_j(\chi_J-1)
 +
 (e_j(1-\chi_J),e_j)e_j
 -
 \sum_{i\not=j}
 (e_j(\chi_J-1),e_i)e_i.
 \end{align*}
 Since $1-\chi_J=0$ for $|z|<T^{-\frac{1}{3}}$
 and
 $e_j={\sf E}_j|z|^{\gamma+2j}$ as $|z|\to\infty$
 (see Section \ref{sec_4.2}),
 it holds that
 \[
 \|e_j(1-\chi_J)\|_\rho^2
 =
 \|e_j(1-\chi_J)\|_{L_\rho^2(|z|>T^{-\frac{1}{3})}}^2
 \lesssim
 e^{-{\sf p}}
 \qquad
 \text{with }
 {\sf p}=
 \tfrac{1}{8}T^{-\frac{2}{3}}.
 \]
 This together with \eqref{eq8.12} shows
 \[
 \|w_0^\bot\|_\rho
 =
 \|\sum_{j=0}^J
 \alpha_j
 e_{j,J}^\bot
 \|_\rho
 \lesssim
 \sum_{j=0}^J
 \alpha_j
 \|e_{j,J}^\bot\|_\rho
 \lesssim
 T^{\frac{\gamma}{2}+J+{\sf c}_1}
 e^{-{\sf p}}.
 \]
 Therefore
 \eqref{eq8.13} can be written as
 \begin{align*}
 \|w_2^\bot\|_\rho^2
 &\lesssim
 e^{-2(\frac{\gamma}{2}+J+\frac{1}{2}{\sf c}_1)\tau}
 (
 T^{-2(\frac{\gamma}{2}+J+\frac{1}{2}{\sf c}_1)}
 \|w_0^\bot\|_\rho^2
 +
 T^{{\sf c}_1}
 )
 \\
 &\lesssim
 T^{{\sf c}_1}
 e^{-2(\frac{\gamma}{2}+J+\frac{1}{2}{\sf c}_1)\tau}.
 \end{align*}
 Combining \eqref{eq8.11},
 we conclude
 \begin{align}\label{eq8.14}
 \|w_2\|_\rho
 &\lesssim
 T^{\frac{1}{2}{\sf c}_1}
 (T-t)^{\frac{\gamma}{2}+J+\frac{1}{2}{\sf c}_1}.
 \end{align}
 By using this estimate,
 we provide pointwise estimates for $w_2$ below.
 The proof is divided into four parts.
 \begin{enumerate}[(i)]
 \item $|z|<1$
 \item $1<|z|<{\sf l}_{\sf out}$
 \item ${\sf l}_{\sf out}\sqrt{T-t}<|x|<1$
 \item $|x|>1$
 \end{enumerate}
 We first consider the case (i).
 To verify $w_2\in X_1$,
 it is sufficient to obtain the spacial weight $|z|^\gamma$ as $|z|\to0$.
 Let $e_{J+1}(z)$ be the eigenfunction of
 $-(\Delta_z-\frac{z}{2}\cdot\nabla_z-q{\sf U}_\infty(z)^{q-1})e=\mu e$
 defined in Section \ref{sec_4.2}.
 We can choose ${\sf r}\in(0,1)$ such that $e_{J+1}(z)>0$ for $|z|<{\sf r}$.
 From \eqref{eq8.8}, \eqref{eq8.12} and \eqref{eq8.14},
 we easily check that
 $\bar w_2=
 T^{\frac{1}{2}{\sf c}_1}
 (T-t)^{\frac{\gamma}{2}+J+\frac{1}{2}{\sf c}_1}
 e_{J+1}(z)$
 gives a super solution of \eqref{eq8.7} in $|z|<{\sf r}$.
 Hence
 we have
 \begin{align}\label{eq8.15}
 \nonumber
 |w_2|
 &\lesssim
 \bar w_2
 =
 T^{\frac{1}{2}{\sf c}_1}
 (t-t)^{\frac{\gamma}{2}+J+\frac{1}{2}{\sf c}_1}
 e_{J+1}(z)
 \\
 &\lesssim
 T^{\frac{1}{2}{\sf c}_1}
 (t-t)^{\frac{\gamma}{2}+J+\frac{1}{2}{\sf c}_1}
 |z|^\gamma
 \qquad
 \text{for }
 |z|<{\sf r},\
 t\in(0,T).
 \end{align}
 In the last inequality,
 we use the property of $e_j(z)$ (see Section \ref{sec_4.2}).
 Since we can assume $2{\sf d}_1<{\sf c}_1$,
 \eqref{eq8.15} proves the case (i).
 We next consider the case (ii).
 From Section \ref{sec_5.3},
 we recall that
 ${\sf l}_{\sf out}=L_2(T-t)^{-\frac{1}{2}+{\sf b}_{\sf out}}$
 with 
 ${\sf b}_{\sf out}
 =
 \tfrac{{\sf d}_1}{2(\gamma+2J-\frac{2}{1-q}+3{\sf d}_1)}$
 and
 \begin{align}\label{eq8.16}
 (T-t)^{\frac{\gamma}{2}+J+{\sf d}_1}
 |z|^{\gamma+2J+3{\sf d}_1}
 >
 {\sf U}_\infty(x)
 \qquad
 \Leftrightarrow
 \qquad
 |z|>{\sf l}_{\sf out}.
 \end{align}
 We put
 ${\sf l}_{\sf x}=(T-t)^{{\sf b}_{\sf x}}$
 with ${\sf b}_{\sf x}=\frac{{\sf d}_1}{2(\gamma+2J+3{\sf d}_1)}$.
 A direct computation shows
 \begin{align}\label{eq8.17}
 (T-t)^{\frac{\gamma}{2}+J+{\sf d}_1}
 |z|^{\gamma+2J+3{\sf d}_1}
 >
 1
 \qquad
 \Leftrightarrow
 \qquad
 |x|>{\sf l}_{\sf x}.
 \end{align}
 From the definition of ${\sf l}_{\sf out}$ and ${\sf l}_x$,
 we note that
 ${\sf l}_{\sf out}\sqrt{T-t}<{\sf l}_{\sf x}$.
 We here write \eqref{eq8.8} in the following form.
 \begin{align*}
 G
 &
 {\bf 1}_{|z|>1}
 {\bf 1}_{|x|<{\sf l}_{\sf x}}
 \lesssim
 (T-t)^{-1+\frac{\gamma}{2}+J+{\sf c}_1}
 |z|^{\gamma+2J}
 {\bf 1}_{1<|z|<(T-t)^{-\frac{1}{3}}}
 \\
 &\quad
 +
 (T-t)^{-1+\frac{\gamma}{2}+J+{\sf d}_1+{\sf c}_1}
 |z|^{\gamma+2J+3{\sf d}_1}
 {\bf 1}_{(T-t)^{-\frac{1}{3}}<|z|<{\sf l}_{\sf out}}
 +
 |x|^{\frac{2q}{1-q}}
 {\bf 1}_{{\sf l}_{\sf out}\sqrt{T-t}<|x|<{\sf l}_{\sf x}}
 \\
 &\lesssim
 (T-t)^{-1+\frac{\gamma}{2}+J+{\sf c}_1}
 |z|^{\gamma+2J+3{\sf d}_1}
 {\bf 1}_{1<|z|<{\sf l}_{\sf out}}
 +
 |x|^{-2}
 {\sf U}_\infty
 {\bf 1}_{{\sf l}_{\sf out}\sqrt{T-t}<|x|<{\sf l}_x}.
 \end{align*}
 We observe from \eqref{eq8.16} that
 \begin{align*}
 |x|^{-2}
 {\sf U}_\infty
 &\lesssim
 |x|^{-2}
 (T-t)^{\frac{\gamma}{2}+J+{\sf d}_1}
 |z|^{\gamma+2J+3{\sf d}_1}
 \\
 &\lesssim
 (T-t)^{-2{\sf b}_{\sf out}+\frac{\gamma}{2}+J+{\sf d}_1}
 |z|^{\gamma+2J+3{\sf d}_1}
 \qquad
 \text{for }
 {\sf l}_{\sf out}\sqrt{T-t}<|x|<{\sf l}_x.
 \end{align*}
 Hence
 it follows that
 \begin{align*}
 G
 {\bf 1}_{|z|>1}
 {\bf 1}_{|x|<{\sf l}_{\sf x}}
 &\lesssim
 (T-t)^{-1+\frac{\gamma}{2}+J+{\sf c}_1}
 |z|^{\gamma+2J+3{\sf d}_1}
 {\bf 1}_{|z|>1}
 {\bf 1}_{|x|<{\sf l}_{\sf x}}.
 \end{align*}
 Furthermore
 from \eqref{eq8.8} and \eqref{eq8.17},
 we immediately see that
 \begin{align*}
 G
 {\bf 1}_{|x|>{\sf l}_{\sf x}}
 &\lesssim
 (
 |x|^\frac{2q}{1-q}
 {\bf 1}_{{\sf l}_{\sf out}\sqrt{T-t}<|x|<4}
 +
 {\sf R}_1^{-1}
 |x|^{-1}
 {\bf 1}_{|x|>1}
 )
 {\bf 1}_{|x|>{\sf l}_{\sf x}}
 \\
 &\lesssim
 {\bf 1}_{|x|>{\sf l}_{\sf x}}
 \\
 &\lesssim
 (T-t)^{\frac{\gamma}{2}+J+{\sf d}_1}
 |z|^{\gamma+2J+3{\sf d}_1}
 {\bf 1}_{|x|>{\sf l}_{\sf x}}.
 \end{align*}
 Therefore
 we obtain
 \[
 G{\bf 1}_{|z|>1}
 \lesssim
 (T-t)^{-1+\frac{\gamma}{2}+J+{\sf c}_1}
 |z|^{\gamma+2J+3{\sf d}_1}
 {\bf 1}_{|z|>1}.
 \]
 We put
 $\bar{w}_2
 =
 (T-t)^{\frac{\gamma}{2}+J+\frac{5}{4}{\sf d}_1}
 |z|^{\gamma+2J+3{\sf d}_1}$.
 A direct computation shows
 \begin{align*}
 (
 \pa_\tau-\Delta_z+\tfrac{z}{2}\cdot\nabla_z+q{\sf U}_\infty(z)^{q-1}
 )
 \bar{w}_2
 &=
 (
 -(
 \tfrac{\gamma}{2}+J+\tfrac{5}{4}{\sf d}_1
 )
 +
 \tfrac{1}{2}
 (
 \gamma+2J+3{\sf d}_1
 )
 -
 C
 |z|^{-2}
 )
 \bar{w}_2
 \\
 &=
 (
 \tfrac{1}{4}{\sf d}_1
 -
 C
 |z|^{-2}
 )
 \bar{w}_2.
 \end{align*}
 Here
 we choose ${\sf m}_3>0$ such that
 $\tfrac{1}{4}{\sf d}_1-C|z|^{-2}>\tfrac{1}{8}{\sf d}_1$
 for $|z|>{\sf m}_3$.
 Hence
 it follows that
 \begin{align*}
 (
 \pa_\tau-\Delta_z+\tfrac{z}{2}\cdot\nabla_z+q{\sf U}_\infty(z)^{q-1}
 )
 \bar{w}_2
 >
 \tfrac{1}{8}
 {\sf d}_1
 \bar{w}_2
 \qquad
 \text{for }
 |z|>{\sf m}_3.
 \end{align*}
 Furthermore
 from \eqref{eq8.12},
 it holds that
 \[
 w_0
 =
 \sum_{j=0}^J
 \alpha_j
 e_j
 \chi_J
 \lesssim
 T^{\frac{\gamma}{2}+J+{\sf c}_1}
 |z|^{\gamma+2J}
 {\bf 1}_{|z|<2T^{-\frac{1}{3}}}
 \qquad
 \text{for }
 |z|>1.
 \]
 This implies
 $|w_0|<\bar{w}_2|_{t=0}$ for $|z|>1$.
 Therefore
 a comparison argument shows
 \begin{align}\label{eq8.18}
 |w_2|
 <
 \bar{w}_2
 =
 (T-t)^{\frac{\gamma}{2}+J+\frac{5}{4}{\sf d}_1}
 |z|^{\gamma+2J+3{\sf d}_1}
 \qquad
 \text{for }
 |z|>{\sf m}_3,\
 t\in(0,T).
 \end{align}
 Since $T^{\frac{1}{8}{\sf d}_1}<{\sf R}_1^{-1}$,
 the case (ii) is proved.
 We investigate the case (iii).
 From \eqref{eq8.8}.
 we easily see that
 \begin{align}\label{eq8.19}
 \nonumber
 G
 {\bf 1}_{|z|>{\sf l}_{\sf out}}
 &\lesssim
 |x|^\frac{2q}{1-q}
 {\bf 1}_{{\sf l}_{\sf out}\sqrt{T-t}<|x|<4}
 {\bf 1}_{|z|>{\sf l}_{\sf out}}
 +
 {\sf R}_1^{-1}
 |x|^{-1}
 {\bf 1}_{|x|>1}
 {\bf 1}_{|z|>{\sf l}_{\sf out}}
 \\
 \nonumber
 &\lesssim
 |x|^{\frac{2q}{1-q}}
 {\bf 1}_{|z|>{\sf l}_{\sf out}}
 \\
 \nonumber
 &\lesssim
 ({\sf l}_{\sf out}\sqrt{T-t})^{-2}
 |x|^{\frac{2}{1-q}}
 {\bf 1}_{|z|>{\sf l}_{\sf out}}
 \\
 &\lesssim
 (T-t)^{-2{\sf b}_{\sf out}}
 |x|^{\frac{2}{1-q}}
 {\bf 1}_{|z|>{\sf l}_{\sf out}}.
 \end{align}
 From \eqref{eq8.18} and \eqref{eq8.16},
 we note that
 $|w_2|
 <
 (T-t)^{\frac{1}{4}{\sf d}_1}
 {\sf U}_\infty(x)$
 on
 $|z|={\sf l}_{\sf out}$.
 Furthermore
 since $\frac{1}{3}<\frac{1}{2}-{\sf b}_{\sf out}$,
 it holds that $w_2|_{t=0}=0$
 for $|z|>{\sf l}_{\sf out}|_{t=0}$.
 We choose a comparison function $\bar{w}_2$ as
 \[
 \bar{w}_2
 =
 2T^{\frac{1}{4}{\sf d}_1}
 e^{-2(T-t)^\frac{1}{2}}
 {\sf U}_\infty(x).
 \]
 We here go back to \eqref{eq8.6}.
 We compute
 \begin{align*}
 \left(
 \pa_t-\Delta_x+q{\sf U}_\infty(x)^{q-1}
 \right)
 \bar{w}_2
 &=
 (
 (T-t)^{-\frac{1}{2}}
 -
 (1-q)
 {\sf U}_\infty^{q-1}
 )
 \bar{w}_2
 \\
 &=
 (
 (T-t)^{-\frac{1}{2}}
 -
 (1-q)
 L_1^{q-1}
 |x|^{-2}
 )
 \bar{w}_2
 \\
 &\geq
 (
 (T-t)^{-\frac{1}{2}}
 -
 C(T-t)^{-2{\sf b}_{\sf out}}
 )
 \bar{w}_2
 \\
 &\geq
 \tfrac{1}{2}
 (T-t)^{-\frac{1}{2}}
 \bar{w}_2
 \qquad
 \text{for }
 |z|>{\sf l}_{\sf out}.
 \end{align*}
 Since
 ${\sf b}_{\sf out}<C{\sf d}_1$ (see the definition of ${\sf b}_{\sf out}$),
 we can assume ${\sf b}_{\sf out}<\frac{1}{8}$.
 Hence
 \eqref{eq8.19} can be written as
 \begin{align*}
 G
 {\bf 1}_{|z|>{\sf l}_{\sf out}}
 &<
 C
 (T-t)^{-2{\sf b}_{\sf out}}
 |x|^{\frac{2}{1-q}}
 {\bf 1}_{|z|>{\sf l}_{\sf out}}
 \\
 &<
 (T-t)^{-\frac{1}{4}}
 |x|^{\frac{2}{1-q}}
 {\bf 1}_{|z|>{\sf l}_{\sf out}}
 \\
 &<
 C
 (T-t)^{-\frac{1}{4}}
 T^{-\frac{1}{4}{\sf d}_1}
 \bar{w}_2
 {\bf 1}_{|z|>{\sf l}_{\sf out}}
 \\
 &=
 C
 (T-t)^{\frac{1}{4}}
 T^{-\frac{1}{4}{\sf d}_1}
 \cdot
 \tfrac{1}{2}
 (T-t)^{-\frac{1}{2}}
 \bar{w}_2
 {\bf 1}_{|z|>{\sf l}_{\sf out}}.
 \end{align*}
 Therefore
 from a comparison argument,
 we obtain
 \begin{align}\label{eq8.20}
 |w_2|
 <
 \bar{w}_2
 =
 2T^{\frac{1}{4}{\sf d}_1}
 e^{-2(T-t)^\frac{1}{2}}
 {\sf U}_\infty(x)
 \qquad
 \text{for } |z|>{\sf l}_{\sf out},\
 t\in(0,T).
 \end{align}
 This gives the desired estimate for the case (iii).
 We finally discuss the case (iv).
 It holds from \eqref{eq8.8} that
 \begin{align*}
 G
 {\bf 1}_{|x|>1}
 &\lesssim
 |x|^\frac{2q}{1-q}
 {\bf 1}_{1<|x|<4}
 +
 {\sf R}_1^{-1}
 |x|^{-1}
 {\bf 1}_{|x|>1}
 \lesssim
 |x|^{-1}
 {\bf 1}_{|x|>1}.
 \end{align*}
 From \eqref{eq8.20} and the definition of $w_0(x)$,
 we have
 \begin{align*}
 |w_2|
 &\lesssim
 T^{\frac{1}{4}{\sf d}_1}
 \qquad
 \text{on } |x|=1,
 \\
 w_2|_{t=0}
 &=
 0
 \qquad
 \text{for } |x|>1.
 \end{align*}
 Taking into account this relation,
 we define
 $
 \bar w_2
 =
 2
 T^{\frac{1}{4}{\sf d}_1}
 e^{-2(T-t)^\frac{1}{2}}
 |x|^{-1}$
 as a comparison function.
 Then
 we get
 \begin{align*}
 (
 \pa_t-\Delta_x+q{\sf U}_\infty(x)^{q-1}
 )
 \bar{w}_2
 >
 \pa_t
 \bar{w}_2
 =
 (T-t)^{-\frac{1}{2}}
 \bar{w}_2.
 \end{align*}
 This implies
 $|G|
 {\bf 1}_{|x|>1}
 <
 (
 \pa_t-\Delta_x+q{\sf U}_\infty(x)^{q-1}
 )
 \bar{w}_2$
 for $|x|>1$.
 A comparison argument shows
 \begin{align}\label{eq8.21}
 |w_2|
 <
 \bar{w}_2(x,t)
 \lesssim
 T^{\frac{1}{4}{\sf d}_1}
 |x|^{-1}
 \qquad
 \text{for }
 |x|>1.
 \end{align}
 Combining
 \eqref{eq8.4}, \eqref{eq8.15}, \eqref{eq8.18}, \eqref{eq8.20}
 and \eqref{eq8.21},
 we conclude
 $w(x,t)=w_1\chi_0+w_2\in X_1\subset X_1^{(\delta)}$.
 This completes the proof.

 \appendix
 \def\thesection{\Alph{lem}}
 \renewcommand{\thesection}{\arabic{section}}
 \setcounter{section}{8}
 \section{Appendix}
 \label{sec_9}
 In this section ,
 we provide details of computations $g$, $N$ and $F$.
 We here introduce new parameters related to ${\sf l}_1$.
 We compute
 the position of
 $|y|={\sf l}_1$ (${\sf l}_1=\sigma^{-\frac{1}{n-2}}$)
 in the $\xi$ variable.
 \begin{align*}
 |\xi|
 &=
 \lambda
 \eta^{-1}
 |y|
 =
 \lambda
 \eta^{-1}
 {\sf l}_1
 \\
 &=
 \lambda
 (\lambda^{-\frac{n-2}{2}}\sigma)^{-\frac{1-q}{2}}
 \sigma^{-\frac{1}{n-2}}
 \\
 &=
 \lambda^{1+\frac{(n-2)(1-q)}{4}}
 \sigma^{-\frac{1-q}{2}-\frac{1}{n-2}}
 \\
 &=
 \lambda^{1+\frac{(n-2)(1-q)}{4}}
 ((T-t)^{-1}\lambda^2)^{-\frac{1-q}{2}-\frac{1}{n-2}}
 \\
 &=
 (T-t)^{\frac{1-q}{2}+\frac{1}{n-2}}
 \lambda^{1+\frac{(n-2)(1-q)}{4}-(1-q)-\frac{2}{n-2}}
 \\
 &=
 (T-t)^{\frac{1-q}{2}+\frac{1}{n-2}}
 \lambda^{-\frac{(6-n)(1-q)}{4}+\frac{n-4}{n-2}}.
 \end{align*}
 Therefore
 there exist ${\sf q}_1,{\sf q}_2\in(0,1)$ such that
 \begin{align}\label{eq9.1}
 \lambda^{-1}
 \eta
 (T-t)^{{\sf q}_1}
 =
 (T-t)^{-{\sf q}_2}
 {\sf l}_1.
 \end{align}

 \subsection{List of definition of $g$, $N$, $F$, $h$, $k$}
 We here collect definition of
 $g_i$, $g_i'$, $g_{{\sf out},i}'$, $N_i$, $F_i$,
 $h_1$, $h_1'$, $k_1$ and $k_1'$.
 Since they are error term,
 we omit their signs.
 \begin{enumerate}
 \item
 $g_0=
 \lambda^{-\frac{n+2}{2}}
 \lambda
 \dot\lambda
 (\Lambda_y{\sf Q})
 (\chi_2-\chi_1)$

 \item
 $g_1=
 \lambda^{-\frac{n-2}{2}}
 {\sf Q}
 \dot\chi_2$

 \item
 $g_2=
 (
 \pa_t(
 \lambda^{-\frac{n-2}{2}}\sigma
 T_1
 )
 )
 \chi_1$
 
 \item
 $g_3=
 (
 \pa_t
 (
 \eta^\frac{2}{1-q}
 {\sf U}
 )
 )
 \chi_2
 (1-\chi_1)$

 \item
 $g_4=
 \theta
 \dot\chi_2
 \chi_3$
 
 \item
 $g_5=
 (
 \lambda^{-\frac{n-2}{2}}\sigma T_1
 +
 {\sf U}_{\sf c}
 )
 \dot\chi_1$

 \item
 $g_6=
 (
 -
 \eta^\frac{2}{1-q}{\sf U}
 +
 {\sf U}_\infty
 \chi_4
 +
 \Theta_J
 \chi_3
 )
 \dot\chi_2$
 \item
 $g_0'=
 \eta^{-1}
 \lambda^{-\frac{n}{2}}
 (\nabla_y{\sf Q}\cdot\nabla_\xi\chi_2)
 +
 \eta^{-2}
 \lambda^{-\frac{n-2}{2}}
 {\sf Q}
 (\Delta_\xi\chi_2)$

 \item
 $g_1'=
 2\lambda^{-\frac{n+2}{2}}
 \sigma
 \nabla_yT_1
 \cdot\nabla_y\chi_1$

 \item
 $g_2'=
 2\lambda^{-1}
 \nabla_x{\sf U}_{\sf c}\cdot\nabla_y\chi_1$

 \item
 $g_3'=
 2(\nabla_x\theta\cdot\nabla_x\chi_2)
 \chi_3$

 \item
 $g_4'=
 \theta(\Delta_x\chi_2)
 \chi_3$

 \item
 $g_5'=
 \lambda^{-2}
 (
 \lambda^{-\frac{n-2}{2}}
 \sigma T_1
 +
 {\sf U}_{\sf c}
 )
 (\Delta_y\chi_1)$

 \item
 $g_6'=
 2\eta^{-1}
 \nabla_x(
 \eta^\frac{2}{1-q}{\sf U}
 -
 {\sf U}_\infty
 \chi_4
 -
 \Theta_J
 \chi_3)
 \cdot\nabla_\xi\chi_2$

 \item
 $g_7'=
 \eta^{-2}
 (
 \eta^\frac{2}{1-q}
 {\sf U}
 -
 {\sf U}_\infty
 \chi_4
 -
 \Theta_J
 \chi_3
 )
 (\Delta_\xi\chi_2)$
 \item
 $g_8'=
 \lambda^{-2}
 V
 (
 {\sf U}_{\sf c}
 (1-\chi_1)
 +
 (\theta+\Theta_J)
 (1-\chi_2)
 )
 \chi_{2}$

 \item
 $g_9'=
 \eta^{-2}
 {\sf U}^{q-1}
 (u+\eta^\frac{2}{1-q}{\sf U}-u_1)
 (1-\chi_{1})
 \chi_{2}$

 \item
 $g_{10}'=
 {\sf U}_\infty^{q-1}
 (u+{\sf U}_\infty+(\theta+\Theta_J)\chi_3-u_1)
 (1-\chi_{2})
 \chi_4$
 \item
 $g_{{\sf out},1}'=
 2(\nabla_x{\sf U}_\infty\cdot\nabla_x\chi_4)
 (1-\chi_2)$

 \item
 $g_{{\sf out},2}'=
 {\sf U}_\infty
 (1-\chi_2)
 (\Delta_x\chi_4)$

 \item
 $g_{{\sf out},3}'=
 {\sf M}(t)
 \Delta_x\chi_4$

 \item
 $g_{{\sf out},4}'=
 2(\nabla_x\theta\cdot\nabla_x\chi_3)
 (1-\chi_2)$

 \item
 $g_{{\sf out},5}'=
 \theta
 (1-\chi_2)
 (\Delta_x\chi_3)$

 \item
 $g_{{\sf out},6}'=
 2(\nabla_x\Theta_J\cdot\nabla_x\chi_3)
 (1-\chi_2)$

 \item
 $g_{{\sf out},7}'=
 \Theta_J
 (1-\chi_2)
 (\Delta_x\chi_3)$
 \item
 $N_1[\epsilon,v,w]=
 \{
 f(u)-f(\lambda^{-\frac{n-2}{2}}{\sf Q})
 -
 \lambda^{-2}
 V
 (u-\lambda^{-\frac{n-2}{2}}{\sf Q})
 \}
 \chi_{2}$

 \item
 $N_2[\epsilon,v,w]=
 \{
 f_2(u)
 +
 f_2(\eta^\frac{2}{1-q}{\sf U})
 -
 q\eta^{-2}{\sf U}^{q-1}
 (u+\eta^\frac{2}{1-q}{\sf U})
 \}
 (1-\chi_{1})
 \chi_{2}$

 \item
 $N_3[\epsilon,v,w]=
 \{
 f(u)
 -
 f_2(u)
 -
 f_2({\sf U}_\infty)
 +
 q{\sf U}_\infty^{q-1}
 (u+{\sf U}_\infty)
 -
 (\Delta_x\theta-q{\sf U}_\infty^{q-1}\theta)
 \chi_3
 \}
 (1-\chi_{2})
 \chi_{4}$

 \item
 $N_4[\epsilon,v,w]=
 f_2(u)
 \chi_1$

 \item
 $N_5[w]=
 q{\sf U}_\infty^{q-1}
 w
 (1-\chi_4)$

 \item
 $N_6[w]=
 \{
 \dot{\sf M}
 +
 f(u)
 -
 f_2(u)
 \}
 (1-\chi_4)$
 \item
 $F_1[v,w]=
 \lambda^{-2}
 V
 (\eta^\frac{2}{1-q}v\chi_{\sf mid}+w)
 \chi_2$
 
 \item
 $F_2[w]
 =
 q
 (\eta^{-2}{\sf U}^{q-1}(1-\chi_1)-{\sf U}_\infty^{q-1})
 w
 \chi_2$
 \item
 $h_1[\epsilon]=
 -\lambda^{-\frac{n+2}{2}}
 \lambda\dot\lambda
 (\Lambda_y\epsilon)\chi_{\sf in}
 +
 \lambda^{-\frac{n+2}{2}}
 \lambda^2
 \epsilon
 \pa_t\chi_{\sf in}$

 \item
 $h_1'[\epsilon]
 =
 2\lambda^{-\frac{n+2}{2}}
 \nabla_y\epsilon\cdot\nabla_y\chi_{\sf in}
 +
 \lambda^{-\frac{n+2}{2}}
 \epsilon\Delta_y\chi_{\sf in}$

 \item
 $k_1[v]=
 \eta^\frac{2}{1-q}
 v
 \pa_t\chi_{\sf mid}$

 \item
 $k_1'[v]=
 2\eta^\frac{2q}{1-q}
 (\nabla_\xi v\cdot\nabla_\xi\chi_{\sf mid})
 +
 \eta^\frac{2q}{1-q}
 v(\Delta_\xi\chi_{\sf mid})$
 \end{enumerate}

 \subsection{Assuptions}
 \label{sec_9.2}
 Throughout Section \ref{sec_9.3} - Section \ref{sec_9.5},
 we assume
 \begin{enumerate}[(\text{A}1)]
 \setlength{\leftskip}{5mm}
 \item
 $\lambda(t)\in C^1([0,T])$ and
 $|\lambda\dot\lambda-\sigma|<D_1(T-t)^{{\sf d}_1}\sigma$,
 \item
 $|\epsilon(y,t)|<\sigma$ for
 $(y,t)\in\overline{B}_{{\sf R}_{\sf in}}\times[0,T]$,
 \item
 $(v,w)\in X$.
 \end{enumerate}
 We list several formulas often used in this section.
 \begin{enumerate}[{\rm (}\sf Ei{\rm )}]
 \setlength{\leftskip}{2mm}
 \item $\lambda\dot\lambda=\sigma(1+o)$
 \quad
 (see (A1)),

 \item $\eta^\frac{2}{1-q}=-{\sf A}_1\lambda^{-\frac{n-2}{2}}\sigma$
 \quad
 (see the definition of $\sigma(t)$),

 \item ${\sf l}_1(t)=\sigma(t)^{-\frac{1}{n-2}}$
 \quad
 (see the beginning of Section \ref{sec_5}),

 \item $|T_1(y)|+|y|^2|\nabla_yT_1|\lesssim1$ for $y\in\R^n$
 \quad
 (see \eqref{eq4.1} - \eqref{eq4.2}),

 \item
 $|\lambda^{-\frac{n-2}{2}}\sigma T_1+\eta^\frac{2}{1-q}{\sf U}|
 \lesssim\eta^\frac{2}{1-q}(|y|^{-1}+|\xi|^2)$
 as $|y|\to\infty$, $|\xi|\to0$
 \\[1mm]
 \hspace{65mm}
 (see the definition of $\eta(t)$ and \eqref{eq4.2}),

 \item
 $|\eta^\frac{2}{1-q}{\sf U}(\xi)-{\sf U}_\infty(x)-\Theta_J(x,t)|
 \lesssim
 \eta^\frac{2}{1-q}|\xi|^\gamma(|\xi|^{-{\sf k}_1}+|z|^2)$
 as $|\xi|\to\infty$, $|z|\to0$
 \\[1mm]
 \hspace{90mm}
 (see \eqref{eq4.10} and Section \ref{sec_4.4}).
 \end{enumerate}
 We here establish several estimates
 independent of
 $\delta$,
 ${\sf d}_1$,
 ${\sf b}$,
 ${\sf R}_{\sf in}$,
 ${\sf R}_{\sf mid}$,
 ${\sf R}_1$,
 ${\sf r}_0$,
 ${\sf r}_3$.
 For real numbers $X$ and $Y$,
 we write $X\lesssim Y$
 to denote $|X|\leq C|Y|$
 for some constant $C>0$.
 The constant $C$ depends only on $q,n,J$
 (independent of
 $\delta$, ${\sf b}$, ${\sf d}_1$
 ${\sf R}_{\sf in}$, ${\sf R}_{\sf mid}$, ${\sf R}_1$, ${\sf r}_0$, ${\sf r}_3$).

 \subsection{Proof of Lemma \ref{Lem7.1}\ and Lemma \ref{Lem8.1}\
 (computations of $g$)}
 \label{sec_9.3}
 Lemma \ref{Lem7.1} immediately follows  from Lemma \ref{Lem9.1}.
 To obtain Lemma \ref{Lem8.1} from Lemma \ref{Lem9.1},
 it is enough to apply
 two relations
 $\lambda^{-\frac{n+2}{2}}
 \sigma
 |y|^{-2}
 =
 \eta^\frac{2q}{1-q}
 |\xi|^{-2}$
 and
 $\tfrac{2}{1-q}-2<\gamma<\tfrac{2}{1-q}$
 (see \eqref{eq4.9}).
 \begin{lem}\label{Lem9.1}
 Assume {\rm (A1)}.
 Then
 it holds that
 \begin{align*}
 g_0
 &\lesssim
 \lambda^{-\frac{n+2}{2}}
 \sigma
 |y|^{-(n-2)}
 {\bf 1}_{|y|>{\sf l}_1}
 {\bf 1}_{|\xi|<2{\sf l}_2},
 \\
 g_1
 &\lesssim
 \lambda^{-\frac{n+2}{2}}
 \sigma
 |y|^{-(n-2)}
 {\bf 1}_{{\sf l}_2<|\xi|<2{\sf l}_2},
 \\
 g_2
 & \lesssim
 |\sigma|^\frac{n-4}{n-2}
 \cdot
 \lambda^{-\frac{n+2}{2}}
 \sigma
 (1+|y|^2)^{-1}
 {\bf 1}_{|y|<2{\sf l}_1},
 \\
 g_3
 {\bf 1}_{|\xi|<1}
 &\lesssim
 |z|^2
 \cdot
 \lambda^{-\frac{n+2}{2}}
 \sigma
 |y|^{-2}
 {\bf 1}_{|y|>{\sf l}_1}
 {\bf 1}_{|\xi|<1},
 \\
 g_3
 {\bf 1}_{|\xi|>1}
 &\lesssim
 |z|^2
 \cdot
 \eta^{\frac{2q}{1-q}}
 |\xi|^{\gamma-2}
 {\bf 1}_{1<|\xi|<2{\sf l}_2},
 \\
 g_4
 &\lesssim
 |z|^2
 \eta^{\frac{2(p-q)}{1-q}}
 |\xi|^{\frac{2p}{1-q}-\gamma+2}
 \cdot
 \eta^{\frac{2q}{1-q}}
 |\xi|^{\gamma-2}
 {\bf 1}_{{\sf l}_2<|\xi|<2{\sf l}_2},
 \\
 g_5
 &\lesssim
 |z|^2
 (
 |y|^{-1}
 +
 |\xi|^2
 )
 \cdot
 \lambda^{-\frac{n+2}{2}}
 \sigma
 |y|^{-2}
 {\bf 1}_{{\sf l}_1<|y|<2{\sf l}_1},
 \\
 g_6
 &\lesssim
 |z|^2
 (
 |\xi|^{-{\sf k}_1}
 +
 |z|^2
 )
 \cdot
 \eta^\frac{2q}{1-q}
 |\xi|^{\gamma-2}
 {\bf 1}_{{\sf l}_2<|\xi|<2{\sf l}_2},
 \\
 g_0'
 &\lesssim
 (
 \sigma^{-1}
 {\sf Q})
 \cdot
 \lambda^{-\frac{n+2}{2}}
 \sigma
 |y|^{-2}
 {\bf 1}_{{\sf l}_2<|\xi|<2{\sf l}_2},
 \\
 g_1'
 &\lesssim
 \lambda^{-\frac{n+2}{2}}
 \sigma
 |y|^{-3}
 {\bf 1}_{{\sf l}_1<|y|<2{\sf l}_1},
 \\
 g_2'
 &\lesssim
 |\xi|^2
 \cdot
 \lambda^{-\frac{n+2}{2}}
 \sigma
 |y|^{-2}
 {\bf 1}_{{\sf l}_1<|y|<2{\sf l}_1},
 \\
 g_3'
 &\lesssim
 \eta^\frac{2(p-q)}{1-q}
 |\xi|^{\frac{2p}{1-q}-\gamma+2}
 \cdot
 \eta^\frac{2q}{1-q}
 |\xi|^{\gamma-2}
 {\bf 1}_{{\sf l}_2<|\xi|<2{\sf l}_2},
 \\
 g_4'
 &\lesssim
 \eta^\frac{2(p-q)}{1-q}
 |\xi|^{\frac{2p}{1-q}-\gamma+2}
 \cdot
 \eta^\frac{2q}{1-q}
 |\xi|^{\gamma-2}
 {\bf 1}_{{\sf l}_2<|\xi|<2{\sf l}_2},
\\
 g_5'
 &\lesssim
 (
 |y|^{-1}
 +
 |\xi|^2
 )
 \cdot
 \lambda^{-\frac{n+2}{2}}
 \sigma
 |y|^{-2}
 {\bf 1}_{{\sf l}_1<|y|<2{\sf l}_1},
 \\
 g_6'
 &\lesssim
 (
 |\xi|^{-{\sf k}_1}
 +
 |z|^2
 )
 \cdot
 \eta^\frac{2q}{1-q}
 |\xi|^{\gamma-2}
 {\bf 1}_{{\sf l}_2<|\xi|<2{\sf l}_2},
 \\
 g_7'
 &\lesssim
 (
 |\xi|^{-{\sf k}_1}
 +
 |z|^2
 )
 \cdot
 \eta^\frac{2q}{1-q}
 |\xi|^{\gamma-2}
 {\bf 1}_{{\sf l}_2<|\xi|<2{\sf l}_2},
 \\
 g_8'
 {\bf 1}_{|\xi|<1}
 &\lesssim
 \lambda^{-\frac{n+2}{2}}
 \sigma
 |y|^{-4}
 {\bf 1}_{|y|>{\sf l}_1}
 {\bf 1}_{|\xi|<1},
 \\
 g_8'
 {\bf 1}_{|\xi|>1}
 &\lesssim
 \lambda^2
 \eta^{-2}
 \cdot
 \eta^\frac{2q}{1-q}
 |\xi|^{\gamma-2+(\frac{2}{1-q}-\gamma-2)}
 {\bf 1}_{1<|\xi|<2{\sf l}_2},
 \\
 g_9'
 {\bf 1}_{|\xi|<1}
 &=
 {\sf s}_1+{\sf s}_2,
 \end{align*}
 \begin{align*}
 {\sf s}_1
 {\bf 1}_{|\xi|<(T-t)^{{\sf q}_1}}
 &\lesssim
 (T-t)^{2{\sf q}_1}
 \cdot
 \lambda^{-\frac{n+2}{2}}
 \sigma
 |y|^{-2}
 {\bf 1}_{|y|>{\sf l}_1}
 {\bf 1}_{|\xi|<1},
 \\
 {\sf s}_1
 {\bf 1}_{(T-t)^{{\sf q}_1}<|\xi|<1}
 &\lesssim
 (T-t)^{(n-2){\sf q}_2}
 \cdot
 \lambda^{-\frac{n+2}{2}}
 \sigma
 |y|^{-2}
 {\bf 1}_{|y|>{\sf l}_1}
 {\bf 1}_{|\xi|<1},
 \\
 {\sf s}_2
 &\lesssim
 |\xi|
 \cdot
 \lambda^{-\frac{n+2}{2}}
 \sigma
 |y|^{-2}
 {\bf 1}_{{\sf l}_1<|y|<2{\sf l}_1},
 \\
 g_9'
 {\bf 1}_{|\xi|>1}
 &=
 {\sf s}_3+{\sf s}_4+{\sf s}_5,
 \\
 {\sf s}_3
 &\lesssim
 (T-t)^{(n-2)({\sf q}_1+{\sf q}_2)}
 \cdot
 \lambda^{-\frac{n+2}{2}}
 \sigma
 |y|^{-2}
 {\bf 1}_{1<|\xi|<2{\sf l}_2},
 \\
 {\sf s}_4
 &\lesssim
 (|\xi|^{-{\sf k}_1}+|z|^2)
 \cdot
 \eta^\frac{2q}{1-q}
 |\xi|^{\gamma-2}
 {\bf 1}_{{\sf l}_2<|\xi|<2{\sf l}_2},
 \\
 {\sf s}_5
 &\lesssim
 \eta^{\frac{2(p-q)}{1-q}}
 |\xi|^{\frac{2p}{1-q}-\gamma+2}
 \cdot
 \eta^\frac{2q}{1-q}
 |\xi|^{\gamma-2}
 {\bf 1}_{{\sf l}_2<|\xi|<2{\sf l}_2},
 \\
 g_{10}'
 &=
 {\sf s}_6+{\sf s}_7+{\sf s}_8+{\sf s}_8,
 \\
 {\sf s}_6
 &\lesssim
 |x|^\frac{2q}{1-q}
 {\bf 1}_{1<|x|<2},
 \\
 {\sf s}_7
 &\lesssim
 (T-t)^{(n-2)({\sf q}_1+{\sf q}_2)}
 \cdot
 \lambda^{-\frac{n+2}{2}}
 \sigma
 |y|^{-2}
 {\bf 1}_{{\sf l}_2<|\xi|<2{\sf l}_2},
 \\
 {\sf s}_8
 &\lesssim
 \eta^{\frac{2(p-q)}{1-q}}
 |\xi|^{\frac{2p}{1-q}-\gamma+2}
 \cdot
 \eta^\frac{2q}{1-q}
 |\xi|^{\gamma-2}
 {\bf 1}_{{\sf l}_2<|\xi|<2{\sf l}_2},
 \\
 {\sf s}_9
 &\lesssim
 (|\xi|^{-{\sf k}_1}+|z|^2)
 \cdot
 \eta^\frac{2q}{1-q}
 {\bf 1}_{{\sf l}_2<|\xi|<2{\sf l}_2},
 \\
 g_{{\sf out},1}'
 +
 g_{{\sf out},2}'
 +
 g_{{\sf out},3}'
 &\lesssim
 {\bf 1}_{1<|x|<2},
 \\
 g_{{\sf out},4}'
 +
 g_{{\sf out},5}'
 &\lesssim
 |x|^\frac{2p}{1-q}
 {\bf 1}_{{\sf r}_3<|x|<2{\sf r}_3},
 \\
 g_{{\sf out},6}'
 +
 g_{{\sf out},7}'
 &\lesssim
 |x|^{\gamma+2J-2}
 {\bf 1}_{{\sf r}_3<|x|<2{\sf r}_3}.
 \end{align*}
 \end{lem}

 \begin{proof}
 Formula ({\sf Ei}) implies
 \begin{align*}
 g_0
 &=
 \lambda^{-\frac{n+2}{2}}
 \lambda
 \dot\lambda
 (\Lambda_y{\sf Q})
 (\chi_2-\chi_1)
 \lesssim
 \lambda^{-\frac{n+2}{2}}
 \sigma
 |y|^{-(n-2)}
 {\bf 1}_{|y|>{\sf l}_1}
 {\bf 1}_{|\xi|<2{\sf l}_2},
 \\
 g_1
 &=
 \lambda^{-\frac{n-2}{2}}
 {\sf Q}
 \dot\chi_2
 \lesssim
 \lambda^{-\frac{n+2}{2}}
 \sigma
 |y|^{-(n-2)}
 {\bf 1}_{{\sf l}_2<|\xi|<2{\sf l}_2}.
 \end{align*}
 From ({\sf Eiii}) - ({\sf Eiv}),
 we see that
 \begin{align*}
 g_2
 &=
 (
 \pa_t(
 \lambda^{-\frac{n-2}{2}}\sigma
 T_1
 )
 )
 \chi_1
 \\
 &\lesssim
 (T-t)^{-1}
 \lambda^{-\frac{n-2}{2}}\sigma
 \chi_1
 \\
 &\lesssim
 \lambda^{-\frac{n+2}{2}}
 \sigma^2
 \chi_1
 \\
 &\lesssim
 \sigma
 (1+|y|^2)
 \cdot
 \lambda^{-\frac{n+2}{2}}
 \sigma
 (1+|y|^2)^{-1}
 {\bf 1}_{|y|<2{\sf l}_1}
 \\
 & \lesssim
 |\sigma|^\frac{n-4}{n-2}
 \cdot
 \lambda^{-\frac{n+2}{2}}
 \sigma
 (1+|y|^2)^{-1}
 {\bf 1}_{|y|<2{\sf l}_1}.
 \end{align*}
 We write $g_3=g_3{\bf 1}_{|\xi|<1}+g_3{\bf 1}_{|\xi|>1}$.
 Then from ({\sf Eii}),
 we get
 \begin{align*}
 g_3
 {\bf 1}_{|\xi|<1}
 &=
 (
 \pa_t
 (
 \eta^\frac{2}{1-q}
 {\sf U}
 )
 )
 \chi_2
 (1-\chi_1)
 {\bf 1}_{|\xi|<1}
 \\
 &=
 \eta^{\frac{2}{1-q}-1}
 \dot\eta
 (\Lambda_\xi{\sf U})
 (1-\chi_1)
 {\bf 1}_{|\xi|<1}
 \\
 &\lesssim
 (T-t)^{-1}
 \lambda^2
 \lambda^{-2}
 \eta^{\frac{2}{1-q}}
 |y|^2
 |y|^{-2}
 {\bf 1}_{|y|>{\sf l}_1}
 {\bf 1}_{|\xi|<1}
 \\
 &\lesssim
 |z|^2
 \cdot
 \lambda^{-2}
 \eta^{\frac{2}{1-q}}
 |y|^{-2}
 {\bf 1}_{|y|>{\sf l}_1}
 {\bf 1}_{|\xi|<1}
 \\
 &=
 |z|^2
 \cdot
 \lambda^{-\frac{n+2}{2}}
 \sigma
 |y|^{-2}
 {\bf 1}_{|y|>{\sf l}_1}
 {\bf 1}_{|\xi|<1},
 \\
 g_3
 {\bf 1}_{|\xi|>1}
 &=
 \eta^{\frac{2}{1-q}-1}
 \dot\eta
 (\Lambda_\xi{\sf U})
 \chi_2
 (1-\chi_1)
 {\bf 1}_{|\xi|>1}
 \\
 &\lesssim
 (T-t)^{-1}
 \eta^{\frac{2}{1-q}}
 |\xi|^\gamma
 {\bf 1}_{1<|\xi|<2{\sf l}_2}
 \\
 &=
 (T-t)^{-1}
 |\xi|^2|\xi|^{-2}
 \eta^2
 \eta^{-2}
 \eta^{\frac{2}{1-q}}
 |\xi|^\gamma
 {\bf 1}_{1<|\xi|<2{\sf l}_2}
 \\
 &\lesssim
 |z|^2
 \cdot
 \eta^{\frac{2q}{1-q}}
 |\xi|^{\gamma-2}
 {\bf 1}_{1<|\xi|<2{\sf l}_2}.
 \end{align*}
 We recall that
 $\theta(x)=a_0L_1^{p+1-q}|x|^{\frac{2p}{1-q}+2}$
 (see Section \ref{sec_4.5}).
 Hence
 it holds that
 \begin{align*}
 g_4
 &=
 \theta
 \dot\chi_2
 \chi_3
 \\
 &\lesssim
 (T-t)^{-1}
 |x|^{\frac{2p}{1-q}+2}
 {\bf 1}_{{\sf l}_2<|\xi|<2{\sf l}_2}
 \\
 &=
 |z|^2
 |x|^{\frac{2p}{1-q}}
 {\bf 1}_{{\sf l}_2<|\xi|<2{\sf l}_2}
 \\
 &=
 |z|^2
 \eta^{\frac{2(p-q)}{1-q}}
 |\xi|^{\frac{2p}{1-q}-\gamma+2}
 \cdot
 \eta^{\frac{2q}{1-q}}
 |\xi|^{\gamma-2}
 {\bf 1}_{{\sf l}_2<|\xi|<2{\sf l}_2}.
 \end{align*}
 From ({\sf Ev}) - ({\sf Evi}),
 we have
 \begin{align*}
 g_5
 &=
 (
 \lambda^{-\frac{n-2}{2}}\sigma T_1
 +
 {\sf U}_{\sf c}
 )
 \dot\chi_1
 \\
 &=
 (
 \lambda^{-\frac{n-2}{2}}\sigma T_1
 +
 \eta^\frac{2}{1-q}
 {\sf U}
 )
 \dot\chi_1
 \\
 &\lesssim
 (T-t)^{-1}
 \lambda^{-\frac{n-2}{2}}\sigma
 (
 |y|^{-1}
 +
 |\xi|^2
 )
 {\bf 1}_{{\sf l}_1<|y|<2{\sf l}_1}
 \\
 &=
 \lambda^2
 (T-t)^{-1}
 |y|^2
 (
 |y|^{-1}
 +
 |\xi|^2
 )
 \lambda^{-\frac{n+2}{2}}
 \sigma
 |y|^{-2}
 {\bf 1}_{{\sf l}_1<|y|<2{\sf l}_1}
 \\
 &=
 |z|^2
 (
 |y|^{-1}
 +
 |\xi|^2
 )
 \cdot
 \lambda^{-\frac{n+2}{2}}
 \sigma
 |y|^{-2}
 {\bf 1}_{{\sf l}_1<|y|<2{\sf l}_1},
 \\
 g_6
 &=
 (
 -
 \eta^\frac{2}{1-q}{\sf U}
 +
 {\sf U}_\infty
 \chi_4
 +
 \Theta_J
 \chi_3
 )
 \dot\chi_2
 \\
 &\lesssim
 (T-t)^{-1}
 \eta^\frac{2}{1-q}
 |\xi|^{\gamma}
 (
 |\xi|^{-{\sf k}_1}
 +
 |z|^2
 )
 {\bf 1}_{{\sf l}_2<|\xi|<2{\sf l}_2}
 \\
 &\lesssim
 (T-t)^{-1}
 |\xi|^2
 \eta^2
 (
 |\xi|^{-{\sf k}_1}
 +
 |z|^2
 )
 \cdot
 \eta^\frac{2q}{1-q}
 |\xi|^{\gamma-2}
 {\bf 1}_{{\sf l}_2<|\xi|<2{\sf l}_2}
 \\
 &\lesssim
 |z|^2
 (
 |\xi|^{-{\sf k}_1}
 +
 |z|^2
 )
 \cdot
 \eta^\frac{2q}{1-q}
 |\xi|^{\gamma-2}
 {\bf 1}_{{\sf l}_2<|\xi|<2{\sf l}_2}.
 \end{align*}
 Therefore
 estimates for $g_0,\cdots,g_6$ are derived.
 We apply ({\sf Eii}) and ({\sf Eiv}) again
 to get
 \begin{align*}
 g_0'
 &=
 \eta^{-1}
 \lambda^{-\frac{n}{2}}
 (\nabla_y{\sf Q}\cdot\nabla_\xi\chi_2)
 +
 \eta^{-2}
 \lambda^{-\frac{n-2}{2}}
 {\sf Q}
 (\Delta_\xi\chi_2)
 \\
 &\lesssim
 \eta^{-1}
 \lambda^{-\frac{n}{2}}
 {\sf Q}
 |y|^{-1}
 |\xi|^{-1}
 {\bf 1}_{{\sf l}_2<|\xi|<2{\sf l}_2}
 +
 \eta^{-2}
 \lambda^{-\frac{n-2}{2}}
 {\sf Q}
 |\xi|^{-2}
 {\bf 1}_{{\sf l}_2<|\xi|<2{\sf l}_2}
 \\
 &\lesssim
 (
 \sigma^{-1}
 {\sf Q}
 )
 \cdot
 \lambda^{-\frac{n+2}{2}}
 \sigma
 |y|^{-2}
 {\bf 1}_{{\sf l}_2<|\xi|<2{\sf l}_2},
 \\
 g_1'
 &=
 2\lambda^{-\frac{n+2}{2}}
 \sigma
 \nabla_yT_1
 \cdot\nabla_y\chi_1
 \\
 &\lesssim
 \lambda^{-\frac{n+2}{2}}
 \sigma
 |y|^{-3}
 {\bf 1}_{{\sf l}_1<|y|<2{\sf l}_1},
 \\
 g_2'
 &=
 2\lambda^{-1}
 \nabla_x{\sf U}_{\sf c}\cdot\nabla_y\chi_1
 \\
 &=
 2\lambda^{-1}\eta^{-1}
 \nabla_\xi(\eta^\frac{2}{1-q}{\sf U})\cdot\nabla_y\chi_1
 \\
 &\lesssim
 \lambda^{-1}\eta^{-1}
 \eta^\frac{2}{1-q}
 |\xi|
 |y|^{-1}
 {\bf 1}_{{\sf l}_1<|y|<2{\sf l}_1}
 \\
 &\lesssim
 \lambda
 \eta^{-1}
 |\xi|
 |y|
 \cdot
 \lambda^{-2}
 \eta^\frac{2}{1-q}
 |y|^{-2}
 {\bf 1}_{{\sf l}_1<|y|<2{\sf l}_1}
 \\
 &\lesssim
 |\xi|^2
 \cdot
 \lambda^{-\frac{n+2}{2}}
 \sigma
 |y|^{-2}
 {\bf 1}_{{\sf l}_1<|y|<2{\sf l}_1}.
 \end{align*}
 A direct computation shows that
 \begin{align*}
 g_3'+g_4'
 &=
 2(\nabla_x\theta\cdot\nabla_x\chi_2)
 \chi_3
 +
 \theta(\Delta_x\chi_2)
 \chi_3
 \\
 &\lesssim
 |x|^{\frac{2p}{1-q}}
 {\bf 1}_{{\sf l}_2<|\xi|<2{\sf l}_2}
 \\
 &=
 \eta^{\frac{2p}{1-q}}
 |\xi|^{\frac{2p}{1-q}}
 {\bf 1}_{{\sf l}_2<|\xi|<2{\sf l}_2}
 \\
 &=
 \eta^{\frac{2(p-q)}{1-q}}
 |\xi|^{\frac{2p}{1-q}-\gamma+2}
 \cdot
 \eta^\frac{2q}{1-q}
 |\xi|^{\gamma-2}
 {\bf 1}_{{\sf l}_2<|\xi|<2{\sf l}_2}.
 \end{align*}
 From ({\sf Ev}) - ({\sf Evi}),
 we get
 \begin{align*}
 g_5'
 &=
 \lambda^{-2}
 (
 \lambda^{-\frac{n-2}{2}}
 \sigma T_1
 +
 {\sf U}_{\sf c}
 )
 (\Delta_y\chi_1)
 \\
 &=
 \lambda^{-2}
 (
 \lambda^{-\frac{n-2}{2}}
 \sigma T_1
 +
 \eta^\frac{2}{1-q}
 {\sf U}
 )
 (\Delta_y\chi_1)
 \\
 &\lesssim
 \lambda^{-2}
 \lambda^{-\frac{n-2}{2}}
 \sigma
 (
 |y|^{-1}
 +
 |\xi|^2
 )
 |y|^{-2}
 {\bf 1}_{{\sf l}_1<|y|<2{\sf l}_1}
 \\
 &\lesssim
 (
 |y|^{-1}
 +
 |\xi|^2
 )
 \cdot
 \lambda^{-\frac{n+2}{2}}
 \sigma
 |y|^{-2}
 {\bf 1}_{{\sf l}_1<|y|<2{\sf l}_1},
 \\
 g_6'
 &=
 2\eta^{-1}
 \nabla_x(
 \eta^\frac{2}{1-q}{\sf U}
 -
 {\sf U}_\infty
 \chi_4
 -
 \Theta_J
 \chi_3)
 \cdot\nabla_\xi\chi_2
 \\
 &\lesssim
 \eta^{-1}
 \cdot
 \eta^\frac{2}{1-q}
 |\xi|^{\gamma}
 (
 \eta^{-1}
 |\xi|^{-{\sf k}_1-1}
 +
 (T-t)^{-\frac{1}{2}}
 |z|
 )
 \cdot
 |\xi|^{-1}
 {\bf 1}_{{\sf l}_2<|\xi|<2{\sf l}_2}
 \\
 &\lesssim
 (
 |\xi|^{-{\sf k}_1}
 +
 (T-t)^{-\frac{1}{2}}
 \eta
 |\xi|
 |z|
 )
 \cdot
 \eta^\frac{2q}{1-q}
 |\xi|^{\gamma-2}
 {\bf 1}_{{\sf l}_2<|\xi|<2{\sf l}_2}
 \\
 &\lesssim
 (
 |\xi|^{-{\sf k}_1}
 +
 |z|^2
 )
 \cdot
 \eta^\frac{2q}{1-q}
 |\xi|^{\gamma-2}
 {\bf 1}_{{\sf l}_2<|\xi|<2{\sf l}_2},
 \\
 g_7'
 &=
 \eta^{-2}
 (
 \eta^\frac{2}{1-q}
 {\sf U}
 -
 {\sf U}_\infty
 \chi_4
 -
 \Theta_J
 \chi_3
 )
 (\Delta_\xi\chi_2)
 \\
 &=
 \eta^{-2}
 (
 \eta^\frac{2}{1-q}
 {\sf U}
 -
 {\sf U}_\infty
 -
 \Theta_J
 )
 (\Delta_\xi\chi_2)
 \\
 &\lesssim
 (
 |\xi|^{-{\sf k}_1}
 +
 |z|^2
 )
 \cdot
 \eta^\frac{2q}{1-q}
 |\xi|^{\gamma-2}
 {\bf 1}_{{\sf l}_2<|\xi|<2{\sf l}_2}.
 \end{align*}
 We recall that $V\lesssim|y|^{-4}$ for $|y|>1$.
 Hence
 it holds that
 \begin{align*}
 g_8'
 {\bf 1}_{|\xi|<1}
 &=
 \lambda^{-2}
 V
 (
 {\sf U}_{\sf c}
 (1-\chi_1)
 +
 (\theta+\Theta_J)
 (1-\chi_2)
 )
 \chi_{2}
 {\bf 1}_{|\xi|<1}
 \\
 &=
 \lambda^{-2}
 V
 {\sf U}_{\sf c}
 (1-\chi_1)
 {\bf 1}_{|\xi|<1}
 \\
 &\lesssim
 \lambda^{-2}
 \eta^\frac{2}{1-q}
 |y|^{-4}
 {\bf 1}_{|y|>{\sf l}_1}
 {\bf 1}_{|\xi|<1}
 \\
 &=
 \lambda^{-\frac{n+2}{2}}
 \sigma
 |y|^{-4}
 {\bf 1}_{|y|>{\sf l}_1}
 {\bf 1}_{|\xi|<1},
 \\
 g_8'
 {\bf 1}_{|\xi|>1}
 &=
 \lambda^{-2}
 V
 ({\sf U}_{\sf c}(1-\chi_1)+(\Theta_J+\theta)(1-\chi_2))
 \chi_2
 {\bf 1}_{|\xi|>1}
 \\
 &\lesssim
 \lambda^{-2}
 V
 {\sf U}_\infty
 {\bf 1}_{1<|\xi|<2{\sf l}_2}
 \\
 &\lesssim
 \lambda^{-2}
 |y|^{-4}
 |x|^\frac{2}{1-q}
 {\bf 1}_{1<|\xi|<2{\sf l}_2}
 \\
 &\lesssim
 \lambda^{-2}
 |y|^{-4}
 \eta^\frac{2}{1-q}
 |\xi|^\frac{2}{1-q}
 {\bf 1}_{1<|\xi|<2{\sf l}_2}
 \\
 &=
 \lambda^2
 \eta^{-2}
 \cdot
 \eta^\frac{2q}{1-q}
 |\xi|^{\gamma-2+(\frac{2}{1-q}-\gamma-2)}
 {\bf 1}_{1<|\xi|<2{\sf l}_2}.
 \end{align*}
 We divide $g_9'{\bf 1}_{|\xi|<1}$ into two parts.
 \begin{align*}
 g_9'
 {\bf 1}_{|\xi|<1}
 &=
 \eta^{-2}
 {\sf U}^{q-1}
 (u+\eta^\frac{2}{1-q}{\sf U}-u_1)
 (1-\chi_{1})
 \chi_{2}
 {\bf 1}_{|\xi|<1}
 \\
 &=
 \eta^{-2}
 {\sf U}^{q-1}
 (
 \underbrace{
 \lambda^{-\frac{n-2}{2}}
 {\sf Q}
 }_{={\sf s}_1}
 +
 \underbrace{
 \lambda^{-\frac{n-2}{2}}
 \sigma
 T_1
 \chi_1
 -
 {\sf U}_{\sf c}
 (1-\chi_1)
 +
 \eta^\frac{2}{1-q}
 {\sf U}
 }_{={\sf s}_2}
 )
 (1-\chi_{1})
 {\bf 1}_{|\xi|<1}.
 \end{align*}
 From the definition of ${\sf q}_1,{\sf q}_2$ and ({\sf Eiii}),
 we get
 \begin{align*}
 {\sf s}_1
 {\bf 1}_{|\xi|<(T-t)^{{\sf q}_1}}
 &=
 \eta^{-2}
 {\sf U}^{q-1}
 \lambda^{-\frac{n-2}{2}}
 {\sf Q}
 (1-\chi_{1})
 {\bf 1}_{|\xi|<(T-t)^{{\sf q}_1}}
 \\
 &\lesssim
 \eta^{-2}
 \lambda^{-\frac{n-2}{2}}
 {\sf Q}
 (1-\chi_{1})
 {\bf 1}_{|\xi|<(T-t)^{{\sf q}_1}}
 \\
 &=
 (
 \lambda^2
 \eta^{-2}
 |y|^2
 )
 (
 \sigma^{-1}
 {\sf Q}
 )
 \cdot
 \lambda^{-\frac{n+2}{2}}
 \sigma
 |y|^{-2}
 {\bf 1}_{|y|>{\sf l}_1}
 {\bf 1}_{|\xi|<(T-t)^{{\sf q}_1}}
 \\
 &=
 |\xi|^2
 (
 \sigma^{-1}
 {\sf Q}
 )
 \cdot
 \lambda^{-\frac{n+2}{2}}
 \sigma
 |y|^{-2}
 {\bf 1}_{|y|>{\sf l}_1}
 {\bf 1}_{|\xi|<(T-t)^{{\sf q}_1}}
 \\
 &\lesssim
 (T-t)^{2{\sf q}_1}
 \cdot
 \lambda^{-\frac{n+2}{2}}
 \sigma
 |y|^{-2}
 {\bf 1}_{|y|>{\sf l}_1}
 {\bf 1}_{|\xi|<(T-t)^{{\sf q}_1}},
 \\
 {\sf s}_1
 {\bf 1}_{|\xi|>(T-t)^{{\sf q}_1}}
 &=
 \eta^{-2}
 {\sf U}^{q-1}
 \lambda^{-\frac{n-2}{2}}
 {\sf Q}
 (1-\chi_{1})
 {\bf 1}_{(T-t)^{{\sf q}_1}<|\xi|<1}
 \\
 &\lesssim
 |\xi|^2
 (
 \sigma^{-1}
 {\sf Q}
 )
 \cdot
 \lambda^{-\frac{n+2}{2}}
 \sigma
 |y|^{-2}
 {\bf 1}_{|y|<(T-t)^{-{\sf q}_2}{\sf l}_1}
 {\bf 1}_{|\xi|<1}
 \\
 &\lesssim
 1\cdot
 \sigma^{-1}
 ((T-t)^{-{\sf q}_2}{\sf l}_1)^{-(n-2)}
 \cdot
 \lambda^{-\frac{n+2}{2}}
 \sigma
 |y|^{-2}
 {\bf 1}_{|y|<(T-t)^{-{\sf q}_2}{\sf l}_1}
 {\bf 1}_{|\xi|<1}
 \\
 &\lesssim
 (T-t)^{(n-2){\sf q}_2}
 \cdot
 \lambda^{-\frac{n+2}{2}}
 \sigma
 |y|^{-2}
 {\bf 1}_{(T-t)^{{\sf q}_1}<|\xi|<1}.
 \end{align*}
 From ({\sf Ev}) and ({\sf Eii}),
 we see that
 \begin{align*}
 {\sf s}_2
 &=
 \eta^{-2}
 {\sf U}^{q-1}
 (
 \lambda^{-\frac{n-2}{2}}
 \sigma
 T_1
 \chi_1
 -
 {\sf U}_{\sf c}
 (1-\chi_1)
 +
 \eta^\frac{2}{1-q}
 {\sf U}
 )
 (1-\chi_{1})
 {\bf 1}_{|\xi|<1}
 \\
 &=
 \eta^{-2}
 {\sf U}^{q-1}
 (
 \lambda^{-\frac{n-2}{2}}
 \sigma
 T_1
 +
 \eta^\frac{2}{1-q}
 {\sf U}
 )
 \chi_1
 (1-\chi_{1})
 {\bf 1}_{|\xi|<1}
 \\
 &\lesssim
 \eta^{-2}
 \eta^\frac{2}{1-q}
 (
 |y|^{-1}
 +
 |\xi|^2
 )
 {\bf 1}_{{\sf l}_1<|y|<2{\sf l}_1}
 \\
 &=
 \eta^{-2}
 \lambda^2
 |y|
 \cdot
 \lambda^{-\frac{n+2}{2}}
 \sigma
 |y|^{-2}
 {\bf 1}_{{\sf l}_1<|y|<2{\sf l}_1}
 +
 \eta^{-2}
 \lambda^2
 |\xi|^2
 |y|^2
 \cdot
 \lambda^{-\frac{n+2}{2}}
 \sigma
 |y|^{-2}
 {\bf 1}_{{\sf l}_1<|y|<2{\sf l}_1}
 \\
 &=
 \lambda
 \eta^{-1}
 |\xi|
 \cdot
 \lambda^{-\frac{n+2}{2}}
 \sigma
 |y|^{-2}
 {\bf 1}_{{\sf l}_1<|y|<2{\sf l}_1}
 +
 |\xi|^4
 \cdot
 \lambda^{-\frac{n+2}{2}}
 \sigma
 |y|^{-2}
 {\bf 1}_{{\sf l}_1<|y|<2{\sf l}_1}
 \\
 &\lesssim
 |\xi|
 \cdot
 \lambda^{-\frac{n+2}{2}}
 \sigma
 |y|^{-2}
 {\bf 1}_{{\sf l}_1<|y|<2{\sf l}_1}.
 \end{align*}
 To compute $g_9'{\bf 1}_{|\xi|>1}$,
 we divide it into three parts.
 \begin{align*}
 g_9'
 {\bf 1}_{|\xi|>1}
 &=
 \eta^{-2}
 {\sf U}^{q-1}
 (u+\eta^\frac{2}{1-q}{\sf U}-u_1)
 (1-\chi_{1})
 \chi_{2}
 {\bf 1}_{|\xi|>1}
 \\
 &=
 \eta^{-2}
 {\sf U}^{q-1}
 (
 \lambda^{-\frac{n-2}{2}}
 {\sf Q}
 \chi_2
 -
 {\sf U}_{\sf c}
 -
 (\theta+\Theta_J)
 (1-\chi_2)
 +
 \eta^\frac{2}{1-q}
 {\sf U}
 )
 \chi_2
 {\bf 1}_{|\xi|>1}
 \\
 &=
 \underbrace{
 \eta^{-2}
 {\sf U}^{q-1}
 \lambda^{-\frac{n-2}{2}}
 {\sf Q}
 \chi_2^2
 {\bf 1}_{|\xi|>1}
 }_{={\sf s}_3}
 +
 \underbrace{
 \eta^{-2}
 {\sf U}^{q-1}
 \{
 \eta^\frac{2}{1-q}
 {\sf U}
 -
 {\sf U}_\infty
 -
 \Theta_J
 \}
 \chi_2
 (1-\chi_2)
 }_{={\sf s}_4}
 \\
 & \quad
 -
 \underbrace{
 \eta^{-2}
 {\sf U}^{q-1}
 \theta
 \chi_2
 (1-\chi_2)
 }_{={\sf s}_5}.
 \end{align*}
 A computation of ${\sf s}_3$ follows from that of
 ${\sf s}_1{\bf 1}_{|\xi|>(T-t)^{{\sf q}_1}}$.
 \begin{align*}
 {\sf s}_3
 &\lesssim
 (T-t)^{(n-2)({\sf q}_1+{\sf q}_2)}
 \cdot
 \lambda^{-\frac{n+2}{2}}
 \sigma
 |y|^{-2}
 {\bf 1}_{1<|\xi|<2{\sf l}_2}.
 \end{align*}
 Due to ({\sf Evi}),
 we get
 \begin{align*}
 {\sf s}_4
 &=
 \eta^{-2}
 {\sf U}^{q-1}
 (
 \eta^\frac{2}{1-q}
 {\sf U}
 -
 {\sf U}_\infty
 -
 \Theta_J
 )
 \chi_2
 (1-\chi_2)
 \\
 &\lesssim
 \eta^{-2}
 |\xi|^{-2}
 \cdot
 \eta^\frac{2}{1-q}
 |\xi|^\gamma(|\xi|^{-{\sf k}_1}+|z|^2)
 {\bf 1}_{{\sf l}_2<|\xi|<2{\sf l}_2}
 \\
 &=
 (|\xi|^{-{\sf k}_1}+|z|^2)
 \cdot
 \eta^\frac{2q}{1-q}
 |\xi|^{\gamma-2}
 {\bf 1}_{{\sf l}_2<|\xi|<2{\sf l}_2}.
 \end{align*}
 From the definition of $\theta$,
 we see that
 \begin{align*}
 {\sf s}_5
 &=
 \eta^{-2}
 {\sf U}^{q-1}
 \theta
 \chi_2
 (1-\chi_2)
 \\
 &\lesssim
 {\sf U}_\infty^{q-1}
 \theta
 \chi_2
 (1-\chi_2)
 \\
 &\lesssim
 |x|^\frac{2p}{1-q}
 {\bf 1}_{{\sf l}_2<|\xi|<2{\sf l}_2}
 \\
 &=
 \eta^{\frac{2(p-q)}{1-q}}
 |\xi|^{\frac{2p}{1-q}-\gamma+2}
 \cdot
 \eta^\frac{2q}{1-q}
 |\xi|^{\gamma-2}
 {\bf 1}_{{\sf l}_2<|\xi|<2{\sf l}_2}.
 \end{align*}
 This completes the estimate for $g_9'$.
 We finally discuss $g_{10}'$.
 By using a new cut off function $\chi_{2,{\sf a}}=\chi_2(\cdot/2)$,
 we divide $g_{10}'$ into four parts.
 \begin{align*}
 g_{10}'
 &=
 {\sf U}_\infty^{q-1}
 (u+{\sf U}_\infty+(\theta+\Theta_J)\chi_3-u_1)
 (1-\chi_{2})
 \chi_4
 \\
 &=
 {\sf U}_\infty^{q-1}
 (u+{\sf U}_\infty+(\theta+\Theta_J)\chi_3-u_1)
 (1-\chi_{2,{\sf a}})
 \chi_4
 \\
 &\quad
 +
 {\sf U}_\infty^{q-1}
 (u+{\sf U}_\infty+(\theta+\Theta_J)\chi_3-u_1)
 (\chi_{2,{\sf a}}-\chi_2)
 \chi_4
 \\
 &=
 {\sf U}_\infty^q
 (1-\chi_4)
 \chi_4
 +
 {\sf U}_\infty^{q-1}
 (
 \lambda^{-\frac{n-2}{2}}
 {\sf Q}
 -
 \eta^\frac{2}{1-q}
 {\sf U}
 +
 {\sf U}_\infty
 +
 (\theta+\Theta_J)
 )
 \chi_2
 (\chi_{2,{\sf a}}-\chi_2)
 \\
 &=
 \underbrace{
 {\sf U}_\infty^q
 (1-\chi_4)
 \chi_4
 }_{={\sf s}_6}
 +
 \underbrace{
 {\sf U}_\infty^{q-1}
 \lambda^{-\frac{n-2}{2}}
 {\sf Q}
 \chi_2
 (\chi_{2,{\sf a}}-\chi_2)
 }_{={\sf s}_7}
 +
 \underbrace{
 {\sf U}_\infty^{q-1}
 \theta
 \chi_2
 (\chi_{2,{\sf a}}-\chi_2)
 }_{={\sf s}_8}
 \\
 & \quad
 +
 \underbrace{
 {\sf U}_\infty^{q-1}
 (
 -
 \eta^\frac{2}{1-q}
 {\sf U}
 +
 {\sf U}_\infty
 +
 \Theta_J
 )
 \chi_2
 (\chi_{2,{\sf a}}-\chi_2)
 }_{={\sf s}_9}.
 \end{align*}
 Estimate for ${\sf s}_6$ is obvious.
 From computations of ${\sf s}_3,{\sf s}_4,{\sf s}_5$,
 we can verify that
 \begin{align*}
 {\sf s}_7
 &\lesssim
 (T-t)^{(n-2)({\sf q}_1+{\sf q}_2)}
 \cdot
 \lambda^{-\frac{n+2}{2}}
 \sigma
 |y|^{-2}
 {\bf 1}_{{\sf l}_2<|\xi|<2{\sf l}_2},
 \\
 {\sf s}_8
 &\lesssim
 \eta^{\frac{2(p-q)}{1-q}}
 |\xi|^{\frac{2p}{1-q}-\gamma+2}
 \cdot
 \eta^\frac{2q}{1-q}
 |\xi|^{\gamma-2}
 {\bf 1}_{{\sf l}_2<|\xi|<2{\sf l}_2},
 \\
 {\sf s}_9
 &\lesssim
 (|\xi|^{-{\sf k}_1}+|z|^2)
 \cdot
 \eta^\frac{2q}{1-q}
 {\bf 1}_{{\sf l}_2<|\xi|<2{\sf l}_2}.
 \end{align*}
 We next compute $g_{\sf out}'$.
 We recall that ${\sf M}(t)={\sf M}_0(1+o)$
 with ${\sf M}_0={\sf U}_\infty(x)|_{|x|=1}$.
 Then we easily see that
 \begin{align*}
 g_{{\sf out},1}'
 &=
 2(\nabla_x{\sf U}_\infty\cdot\nabla_x\chi_4)
 (1-\chi_2)
 \lesssim
 {\bf 1}_{1<|x|<2},
 \\
 g_{{\sf out},2}'
 &=
 {\sf U}_\infty
 (1-\chi_2)
 (\Delta_x\chi_4)
 \lesssim
 {\bf 1}_{1<|x|<2},
 \\
 g_{{\sf out},3}'
 &=
 {\sf M}(t)
 \Delta_x\chi_4
 \lesssim
 {\bf 1}_{1<|x|<2},
 \\
 g_{{\sf out},4}'
 &=
 2(\nabla_x\theta\cdot\nabla_x\chi_3)
 (1-\chi_2)
 \lesssim
 |x|^\frac{2p}{1-q}
 {\bf 1}_{{\sf r}_3<|x|<2{\sf r}_3},
 \\
 g_{{\sf out},5}'
 &=
 \theta
 (1-\chi_2)
 (\Delta_x\chi_3)
 \lesssim
 |x|^\frac{2p}{1-q}
 {\bf 1}_{{\sf r}_3<|x|<2{\sf r}_3},
 \\
 g_{{\sf out},6}'
 &=
 2(\nabla_x\Theta_J\cdot\nabla_x\chi_3)
 (1-\chi_2)
 \lesssim
 |x|^{\gamma+2J-2}
 {\bf 1}_{{\sf r}_3<|x|<2{\sf r}_3},
 \\
 g_{{\sf out},7}'
 &=
 \Theta_J
 (1-\chi_2)
 (\Delta_x\chi_3)
 \lesssim
 |x|^{\gamma+2J-2}
 {\bf 1}_{{\sf r}_3<|x|<2{\sf r}_3}.
 \end{align*}
 We completes the proof of Lemma \ref{Lem9.1}.
 \end{proof}

 \subsection{Proof of Lemma \ref{Lem7.2}\ and Lemma \ref{Lem8.2}\
 (computations of $N$)}
 \label{sec_9.4}
 Lemma \ref{Lem7.2} and Lemma \ref{Lem8.2} are obtained
 from Lemma \ref{Lem9.2}.
 \begin{lem}\label{Lem9.2}
 Assume {\rm (A1)} {\rm -} {\rm (A3)}.
 Then
 it holds that
 \begin{align*}
 N_1[\epsilon,v,w]
 {\bf 1}_{|y|<{\sf l}_1}
 &\lesssim
 \sigma^\frac{2}{n-2}
 \cdot
 \lambda^{-\frac{n+2}{2}}
 \sigma
 (1+|y|^2)^{-1}
 {\bf 1}_{|y|<{\sf l}_1},
 \\
 N_1[\epsilon,v,w]
 {\bf 1}_{|y|>{\sf l}_1}
 {\bf 1}_{|\xi|<1}
 &\lesssim
 \eta^\frac{2(p-q)}{1-q}
 \cdot
 \eta^\frac{2q}{1-q}
 {\bf 1}_{|y|>{\sf l}_1}
 {\bf 1}_{|\xi|<1},
 \\
 N_1[\epsilon,v,w]
 {\bf 1}_{|\xi|>1}
 &\lesssim
 \eta^\frac{2(p-q)}{1-q}
 |\xi|^{\frac{2p}{1-q}-\gamma+2}
 \cdot
 \eta^\frac{2q}{1-q}
 |\xi|^{\gamma-2}
 {\bf 1}_{1<|\xi|<2{\sf l}_2},
 \\
 N_2[\epsilon,v,w]
 {\bf 1}_{|y|<2{\sf l}_1}
 &\lesssim
 |\xi|^2
 \cdot
 \lambda^{-\frac{n+2}{2}}
 \sigma
 |y|^{-2}
 {\bf 1}_{{\sf l}_1<|y|<2{\sf l}_1},
 \\
 N_2[\epsilon,v,w]
 {\bf 1}_{|y|>2{\sf l}_1}
 {\bf 1}_{|\xi|<(T-t)^{{\sf q}_1}}
 &\lesssim
 (T-t)^{2{\sf q}_1}
 \cdot
 \lambda^{-\frac{n+2}{2}}
 \sigma
 |y|^{-2}
 {\bf 1}_{|y|>{\sf l}_1}
 {\bf 1}_{|\xi|<(T-t)^{{\sf q}_1}},
 \\
 N_2[\epsilon,v,w]
 {\bf 1}_{(T-t)^{{\sf q}_1}<|\xi|<1}
 &\lesssim
 (T-t)^{2{\sf d}_1}
 \cdot
 \eta^\frac{2q}{1-q}
 {\bf 1}_{(T-t)^{{\sf q}_1}<|\xi|<1},
 \\
 N_2[\epsilon,v,w]
 {\bf 1}_{|\xi|>1}
 &=
 {\sf t}_1
 +
 {\sf t}_2
 +
 {\sf t}_3
 +
 {\sf t}_4,
 \end{align*}
 \begin{align*}
 {\sf t}_1
 &\lesssim
 (T-t)^{2({\sf q}_1+{\sf q}_2)(n-2)}
 \cdot
 \eta^\frac{2q}{1-q}
 |\xi|^{-\frac{2}{1-q}-2}
 {\bf 1}_{1<|\xi|<2{\sf l}_2},
 \\
 {\sf t}_2
 &\lesssim
 (|\xi|^{-{\sf k}_1}+|z|^2)^2
 \cdot
 \eta^\frac{2q}{1-q}
 |\xi|^{\gamma-2-(\frac{2}{1-q}-\gamma)}
 {\bf 1}_{{\sf l}_2<|\xi|<4{\sf l}_2},
 \\
 {\sf t}_3
 &\lesssim
 \eta^\frac{2(p-q)}{1-q}
 |\xi|^{\frac{2p}{1-q}-\gamma+2}
 \cdot
 \eta^\frac{2q}{1-q}
 |\xi|^{\gamma-2}
 {\bf 1}_{1<|\xi|<2{\sf l}_2},
 \\
 {\sf t}_4
 &\lesssim
 (T-t)^{2{\sf d}_1}
 \cdot
 \eta^\frac{2q}{1-q}
 |\xi|^{\gamma-2-(\frac{2}{1-q}-\gamma)}
 {\bf 1}_{1<|\xi|<2{\sf l}_2},
 \\
 N_3[\epsilon,v,w]
 &=
 {\sf t}_7+{\sf t}_8+{\sf t}_9,
 \\
 {\sf t}_7
 &\lesssim
 \eta^\frac{2q}{1-q}
 |\xi|^{\gamma-2-(\frac{2}{1-q}-\gamma)}
 {\bf 1}_{{\sf l}_2<|\xi|<4{\sf l}_2},
 \\
 {\sf t}_8
 &\lesssim
 {\sf U}_\infty^q
 {\bf 1}_{\frac{1}{2}{\sf r}_3<|x|<2},
 \\
 {\sf t}_9
 {\bf 1}_{|z|<1}
 &\lesssim
 \eta^\frac{2q}{1-q}
 |\xi|^{\gamma-2-\frac{2}{1-q}+\gamma}
 {\bf 1}_{|\xi|>2{\sf l}_2}
 {\bf 1}_{|z|<1},
 \\
 {\sf t}_9
 {\bf 1}_{1<|z|<(T-t)^{-\frac{1}{3}}}
 &\lesssim
 (T-t)^{-1+\frac{1}{6}(\gamma+2J-\frac{2}{1-q})}
 \cdot
 (T-t)^{\frac{\gamma}{2}+J}
 |z|^{\gamma+2J}
 {\bf 1}_{1<|z|<(T-t)^{-\frac{1}{3}}},
 \\
 {\sf t}_9
 {\bf 1}_{(T-t)^{-\frac{1}{3}}<|z|<{\sf l}_{\sf out}}
 &\lesssim
 (T-t)^{-\frac{1}{3}-{\sf d}_1}
 \cdot
 (T-t)^{\frac{\gamma}{2}+J+\frac{3}{2}{\sf d}_1}
 |z|^{\gamma+2J+3{\sf d}_1}
 {\bf 1}_{(T-t)^{-\frac{1}{3}}<|z|<{\sf l}_{\sf out}},
 \\
 {\sf t}_9
 {\bf 1}_{|x|>{\sf lout}\sqrt{T-t}}
 &\lesssim
 {\sf U}_\infty^q
 {\bf 1}_{{\sf lout}\sqrt{T-t}<|x|<{\sf r}_3},
 \\
 N_4[\epsilon,v,w]
 {\bf 1}_{|y|<1}
 &\lesssim
 (T-t)
 \lambda^{\frac{(n-2)(1-q)}{2}}
 \cdot
 \lambda^{-\frac{n+2}{2}}
 \sigma
 {\bf 1}_{|y|<1},
 \\
 N_4[\epsilon,v,w]
 {\bf 1}_{|y|>1}
 &\lesssim
 (T-t)
 \cdot
 \lambda^{-\frac{n+2}{2}}
 \sigma
 |y|^{-2}
 {\bf 1}_{1<|y|<2{\sf l}_1},
 \\
 N_5[w]
 &\lesssim
 {\sf R}_1^{-1}
 |x|^{-3}
 {\bf 1}_{|x|>1},
 \\
 N_6[w]
 &\lesssim
 {\bf 1}_{1<|x|<4}
 +
 {\sf R}^{-1}
 |x|^{-1}
 {\bf 1}_{|x|>2}.
 \end{align*}
 \end{lem}
 \begin{proof}
 From (A1) - (A3),
 we note that
 $u_1(x,t)=
 \lambda^{-\frac{n-2}{2}}
 \epsilon(y,t)
 \chi_{\sf in}
 +
 \eta(t)^\frac{2}{1-q}
 v(\xi,t)
 \chi_{\sf mid}
 +
 w(x,t)$
 satisfies
 \begin{align*}
 u_1(x,t)
 \lesssim
 \begin{cases}
 \lambda^{-\frac{n-2}{2}}
 \sigma
 =
 \eta^\frac{2}{1-q}
 &
 \text{for } |\xi|<1,
 \\
 (T-t)^{{\sf d}_1}
 \eta^\frac{2}{1-q}
 |\xi|^\gamma
 &
 \text{for } |\xi|>1,\ |z|<1,
 \\
 (T-t)^{\frac{\gamma}{2}+J+{\sf d}_1}
 |z|^{\gamma+2J+3{\sf d}_1}
 &
 \text{for } 1<|z|<{\sf l}_{\sf out},
 \\
 {\sf U}_\infty(x)
 &
 \text{for } \sqrt{T-t}\cdot{\sf l}_{\sf out}<|x|<2.
 \end{cases}
 \end{align*}
 We recall that
 $\lambda^{-\frac{n-2}{2}}\sigma\lesssim\lambda^{-\frac{n-2}{2}}{\sf Q}$
 for $|y|>{\sf l}_1$.
 Therefore since $\lambda^{-\frac{n-2}{2}}\sigma=\eta^\frac{2}{1-q}$,
 it holds that
 \begin{align*}
 |
 \lambda^{-\frac{n-2}{2}}
 {\sf Q}(y)
 \chi_2
 +
 \lambda^{-\frac{n-2}{2}}
 &
 \sigma
 T_1(y)
 -
 {\sf U}_{\sf c}(x,t)
 (1-\chi_1)
 -
 (\theta(x)+\Theta_J(x,t))
 (1-\chi_2)
 \chi_3
 |
 \\
 &\lesssim
 \begin{cases}
 \lambda^{-\frac{n-2}{2}}
 {\sf Q}(y) & \text{for } |y|<{\sf l}_1,
 \\
 \eta^\frac{2}{1-q}
 {\sf U}(\xi)
 & \text{for } |y|>{\sf l}_1,\ |\xi|<1,
 \\
 {\sf U}_\infty(x)
 & \text{for } |\xi|>1,\ |x|<1,
 \\
 {\sf M}(t)
 & \text{for } |x|>1.
 \end{cases}
 \end{align*}
 From this estimate and ({\sf Eiii}),
 we get
 \begin{align*}
 N_1[\epsilon,v,w]
 {\bf 1}_{|y|<{\sf l}_1}
 &=
 \{
 f(u)-f(\lambda^{-\frac{n-2}{2}}{\sf Q})
 -
 \lambda^{-2}
 V
 (u-\lambda^{-\frac{n-2}{2}}{\sf Q})
 \}
 \chi_2
 {\bf 1}_{|y|<{\sf l}_1}
 \\
 &\lesssim
 f''(\lambda^{-\frac{n-2}{2}}{\sf Q})
 (u-\lambda^{-\frac{n-2}{2}}{\sf Q})^2
 {\bf 1}_{|y|<{\sf l}_1}
 \\
 &\lesssim
 (\lambda^{-\frac{n-2}{2}}{\sf Q})^{p-2}
 (
 (\lambda^{-\frac{n-2}{2}}\sigma)^2
 +
 u_1^2
 )
 {\bf 1}_{|y|<{\sf l}_1}
 \\
 &\lesssim
 (\lambda^{-\frac{n-2}{2}}{\sf Q})^{p-2}
 (\lambda^{-\frac{n-2}{2}}\sigma)^2
 {\bf 1}_{|y|<{\sf l}_1}
 \\
 &\lesssim
 \lambda^{-\frac{(n-2)p}{2}}
 (1+|y|^2)^{-\frac{1}{2}(n-2)(p-2)}
 \sigma^2
 {\bf 1}_{|y|<{\sf l}_1}
 \\
 &=
 (1+|y|^2)^{\frac{1}{2}(n-4)}
 \sigma
 \cdot
 \lambda^{-\frac{n+2}{2}}
 \sigma
 (1+|y|^2)^{-1}
 {\bf 1}_{|y|<{\sf l}_1}
 \\
 &=
 \sigma^\frac{2}{n-2}
 \cdot
 \lambda^{-\frac{n+2}{2}}
 \sigma
 (1+|y|)^{-2}
 {\bf 1}_{|y|<{\sf l}_1},
 \\
 N_1[\epsilon,v,w]
 {\bf 1}_{|y|>{\sf l}_1}
 {\bf 1}_{|\xi|<1}
 &=
 \{
 f(u)-f(\lambda^{-\frac{n-2}{2}}{\sf Q})
 -
 \lambda^{-2}
 V
 (u-\lambda^{-\frac{n-2}{2}}{\sf Q})
 \}
 \chi_2
 {\bf 1}_{|y|>{\sf l}_1}
 {\bf 1}_{|\xi|<1}
 \\
 &\lesssim
 (
 (\eta^\frac{2}{1-q}{\sf U})^p
 +
 u_1^p
 )
 {\bf 1}_{|y|>{\sf l}_1}
 {\bf 1}_{|\xi|<1}
 \\
 &\lesssim
 (\eta^\frac{2}{1-q}{\sf U})^p
 {\bf 1}_{|y|>{\sf l}_1}
 {\bf 1}_{|\xi|<1}
 \\
 &\lesssim
 \eta^\frac{2(p-q)}{1-q}
 \cdot
 \eta^\frac{2q}{1-q}
 {\bf 1}_{|y|>{\sf l}_1}
 {\bf 1}_{|\xi|<1},
 \end{align*}
 \begin{align*}
 N_1[\epsilon,v,w]
 {\bf 1}_{|\xi|>1}
 &=
 \{
 f(u)-f(\lambda^{-\frac{n-2}{2}}{\sf Q})
 -
 \lambda^{-2}
 V
 (u-\lambda^{-\frac{n-2}{2}}{\sf Q})
 \}
 \chi_2
 {\bf 1}_{|\xi|>1}
 \\
 &\lesssim
 (
 {\sf U}_\infty^p
 +
 u_1^p
 )
 {\bf 1}_{1<|\xi|<2{\sf l}_2}
 \\
 &\lesssim
 {\sf U}_\infty^p
 {\bf 1}_{1<|\xi|<4{\sf l}_2}
 \\
 &=
 \eta^\frac{2(p-q)}{1-q}
 |\xi|^{\frac{2p}{1-q}-\gamma+2}
 \cdot
 \eta^\frac{2q}{1-q}
 |\xi|^{\gamma-2}
 {\bf 1}_{1<|\xi|<2{\sf l}_2}.
 \end{align*}
 We apply the Taylor expansion to get
 \begin{align*}
 N_2[\epsilon,v,w]
 &=
 \{
 f_2(u)
 +
 f_2(\eta^\frac{2}{1-q}{\sf U})
 -
 q\eta^{-2}{\sf U}^{q-1}
 (u+\eta^\frac{2}{1-q}{\sf U})
 \}
 (1-\chi_1)
 \chi_2
 \\
 &=
 \tfrac{1}{2}
 f_2''(u_\kappa)
 (u+\eta^\frac{2}{1-q}{\sf U})^2
 (1-\chi_1)
 \chi_2,
 \end{align*}
 where
 $u_\kappa
 =
 \kappa u-(1-\kappa)\eta^\frac{2}{1-q}{\sf U}$
 with
 $\kappa\in(0,1)$.
 From ({\sf Eii}),
 we have
 \begin{align*}
 N_2[\epsilon,v,w]
 {\bf 1}_{|y|<2{\sf l}_1}
 &=
 \tfrac{1}{2}
 f_2''(u_\kappa)
 (u+\eta^\frac{2}{1-q}{\sf U})^2
 (1-\chi_1)
 \chi_2
 {\bf 1}_{|y|<2{\sf l}_1}
 \\
 &=
 \tfrac{1}{2}
 f_2''(u_\kappa)
 (
 \lambda^{-\frac{n-2}{2}}
 {\sf Q}
 +
 \lambda^{-\frac{n-2}{2}}
 \sigma
 T_1
 \chi_1
 +
 \eta^\frac{2}{1-q}
 {\sf U}
 \chi_1
 +
 u_1
 )^2
 {\bf 1}_{{\sf l}_1<|y|<2{\sf l}_1}
 \\
 &\lesssim
 (\eta^\frac{2}{1-q}{\sf U})^{q-2}
 (\eta^\frac{2}{1-q}{\sf U})^2
 {\bf 1}_{{\sf l}_1<|y|<2{\sf l}_1}
 \\
 &\lesssim
 \lambda^2
 \eta^{-2}
 |y|^2
 \cdot
 \lambda^{-\frac{n+2}{2}}
 \sigma
 |y|^{-2}
 {\bf 1}_{{\sf l}_1<|y|<2{\sf l}_1}
 \\
 &=
 |\xi|^2
 \cdot
 \lambda^{-\frac{n+2}{2}}
 \sigma
 |y|^{-2}
 {\bf 1}_{{\sf l}_1<|y|<2{\sf l}_1}.
 \end{align*}
 Furthermore
 from ({\sf Eii}),
 we get
 \begin{align*}
 N_2[\epsilon,v,w]
 {\bf 1}_{|y|>2{\sf l}_1}
 {\bf 1}_{|\xi|<1}
 &=
 \tfrac{1}{2}
 f_2''(u_\kappa)
 (u+\eta^\frac{2}{1-q}{\sf U})^2
 (1-\chi_1)
 \chi_2
 {\bf 1}_{|y|>2{\sf l}_1}
 {\bf 1}_{|\xi|<1}
 \\
 &=
 \tfrac{1}{2}
 f_2''(u_\kappa)
 (
 \lambda^{-\frac{n-2}{2}}{\sf Q}
 +
 u_1
 )^2
 {\bf 1}_{|y|>2{\sf l}_1}
 {\bf 1}_{|\xi|<1}
 \\
 &\lesssim
 \eta^\frac{2q-4}{1-q}
 (
 \lambda^{-\frac{n-2}{2}}
 {\sf Q}
 +
 (T-t)^{{\sf d}_1}
 \eta^\frac{2}{1-q}
 )^2
 {\bf 1}_{|y|>2{\sf l}_1}
 {\bf 1}_{|\xi|<1}
 \\
 &=
 \eta^\frac{2q}{1-q}
 (
 \sigma^{-1}
 {\sf Q}
 +
 (T-t)^{{\sf d}_1}
 )^2
 {\bf 1}_{|y|>2{\sf l}_1}
 {\bf 1}_{|\xi|<1}.
 \end{align*}
 This together with the definition of ${\sf q}_1,{\sf q}_2$ implies
 \begin{align*}
 N_2[\epsilon,v,w]
 {\bf 1}_{|y|>2{\sf l}_1}
 {\bf 1}_{|\xi|<(T-t)^{{\sf q}_1}}
 &\lesssim
 \eta^\frac{2q}{1-q}
 (
 \sigma^{-1}
 {\sf Q}
 +
 (T-t)^{{\sf d}_1}
 )^2
 {\bf 1}_{|y|>2{\sf l}_1}
 {\bf 1}_{|\xi|<(T-t)^{{\sf q}_1}}
 \\
 &=
 \lambda^2
 \eta^{-2}
 |y|^2
 \{
 (
 \sigma^{-1}
 {\sf Q}
 +
 (T-t)^{2{\sf d}_1}
 \}
 \cdot
 \lambda^{-\frac{n+2}{2}}
 \sigma
 |y|^{-2}
 {\bf 1}_{|y|>2{\sf l}_1}
 {\bf 1}_{|\xi|<(T-t)^{{\sf q}_1}}
 \\
 &=
 |\xi|^2
 \{
 (
 \sigma^{-1}
 {\sf Q}
 )^2
 +
 (T-t)^{2{\sf d}_1}
 \}
 \cdot
 \lambda^{-\frac{n+2}{2}}
 \sigma
 |y|^{-2}
 {\bf 1}_{|y|>2{\sf l}_1}
 {\bf 1}_{|\xi|<(T-t)^{{\sf q}_1}}
 \\
 &\lesssim
 |\xi|^2
 \cdot
 \lambda^{-\frac{n+2}{2}}
 \sigma
 |y|^{-2}
 {\bf 1}_{|y|>2{\sf l}_1}
 {\bf 1}_{|\xi|<(T-t)^{{\sf q}_1}}
 \\
 &\lesssim
 (T-t)^{2{\sf q}_1}
 \cdot
 \lambda^{-\frac{n+2}{2}}
 \sigma
 |y|^{-2}
 {\bf 1}_{|y|>2{\sf l}_1}
 {\bf 1}_{|\xi|<(T-t)^{{\sf q}_1}},
 \\
 N_2[\epsilon,v,w]
 {\bf 1}_{(T-t)^{{\sf q}_1}<|\xi|<1}
 &\lesssim
 \eta^\frac{2q}{1-q}
 (
 \sigma^{-1}
 {\sf Q}
 +
 (T-t)^{{\sf d}_1}
 )^2
 {\bf 1}_{(T-t)^{{\sf q}_1}<|\xi|<1}
 \\
 &\lesssim
 \eta^\frac{2q}{1-q}
 \{
 (T-t)^{2(n-2){\sf q}_2}
 +
 (T-t)^{2{\sf d}_1}
 \}
 {\bf 1}_{(T-t)^{{\sf q}_1}<|\xi|<1}
 \\
 &\lesssim
 (T-t)^{2{\sf d}_1}
 \eta^\frac{2q}{1-q}
 {\bf 1}_{(T-t)^{{\sf q}_1}<|\xi|<1}.
 \end{align*}
 In the last line,
 we use ${\sf d}_1<(n-2){\sf q}_2$.
 We divide $N_2[\epsilon,v,w]{\bf 1}_{|\xi|>1}$ into four parts.
 \begin{align*}
 N_2
 [\epsilon,v,w]
 {\bf 1}_{|\xi|>1}
 &=
 \tfrac{1}{2}
 f_2''(u_\kappa)
 (u+\eta^\frac{2}{1-q}{\sf U})^2
 (1-\chi_1)
 \chi_2
 {\bf 1}_{|\xi|>1}
 \\
 &\lesssim
 {\sf U}_\infty^{q-2}
 \{
 \lambda^{-\frac{n-2}{2}}{\sf Q}
 \chi_2
 -
 \eta^\frac{2}{1-q}{\sf U}
 \chi_2
 -
 {\sf U}_\infty
 (1-\chi_2)
 -
 (\Theta_J+\theta)
 (1-\chi_2)
 \\
 & \quad
 +
 \eta^\frac{2}{1-q}{\sf U}
 +
 u_1
 \}^2
 {\bf 1}_{1<|\xi|<2{\sf l}_2}
 \\
 &=
 {\sf U}_\infty^{q-2}
 \{
 \underbrace{
 \lambda^{-\frac{n-2}{2}}
 {\sf Q}
 }_{={\sf t}_1}
 +
 \underbrace{
 (
 \eta^\frac{2}{1-q}{\sf U}
 -
 {\sf U}_\infty
 -
 \Theta_J
 )
 (1-\chi_2)
 }_{={\sf t}_2}
 -
 \underbrace{
 \theta
 (1-\chi_2)
 }_{={\sf t}_3}
 +
 \underbrace{
 u_1
 }_{={\sf t}_4}
 \}^2
 \\
 & \quad
 \times
 {\bf 1}_{1<|\xi|<2{\sf l}_2}.
 \end{align*}
 From the definition of ${\sf q}_1,{\sf q_2}$
 (see \eqref{eq9.1}),
 we note that
 \begin{align}\label{eq9.2}
 (\sigma^{-1}{\sf Q})
 {\bf 1}_{|\xi|>1}
 \lesssim
 (T-t)^{({\sf q}_1+{\sf q}_2)(n-2)}.
 \end{align}
 We compute ${\sf t}_i$ ($i=1,2,3,4$)
 from ({\sf Eii}), ({\sf Evi}) and \eqref{eq9.2}.
 \begin{align*}
 {\sf t}_1
 &=
 {\sf U}_\infty^{q-2}
 (
 \lambda^{-\frac{n-2}{2}}
 {\sf Q}
 )^2
 {\bf 1}_{1<|\xi|<2{\sf l}_2}
 \\
 &\lesssim
 \eta^\frac{2(q-2)}{1-q}
 |\xi|^{\frac{2(q-2)}{1-q}}
 (
 \lambda^{-\frac{n-2}{2}}
 {\sf Q}
 )^2
 {\bf 1}_{1<|\xi|<2{\sf l}_2}
 \\
 &=
 (\sigma^{-1}{\sf Q})^2
 (
 \eta^{-\frac{2}{1-q}}
 \lambda^{-\frac{n-2}{2}}
 \sigma
 )^2
 \cdot
 \eta^\frac{2q}{1-q}
 |\xi|^{-\frac{2}{1-q}-2}
 {\bf 1}_{1<|\xi|<2{\sf l}_2}
 \\
 &=
 (\sigma^{-1}{\sf Q})^2
 \cdot
 \eta^\frac{2q}{1-q}
 |\xi|^{-\frac{2}{1-q}-2}
 {\bf 1}_{1<|\xi|<2{\sf l}_2}
 \\
 &\lesssim
 (T-t)^{2({\sf q}_1+{\sf q}_2)(n-2)}
 \cdot
 \eta^\frac{2q}{1-q}
 |\xi|^{-\frac{2}{1-q}-2}
 {\bf 1}_{1<|\xi|<2{\sf l}_2},
 \\
 {\sf t}_2
 &=
 {\sf U}_\infty^{q-2}
 (\eta^\frac{2}{1-q}{\sf U}-{\sf U}_\infty-\Theta_J)^2
 (1-\chi_2)^2
 {\bf 1}_{1<|\xi|<2{\sf l}_2}
 \\
 &\lesssim
 \eta^\frac{2(q-2)}{1-q}
 |\xi|^{\frac{2(q-2)}{1-q}}
 \cdot
 \eta^\frac{4}{1-q}
 |\xi|^{2\gamma}
 (|\xi|^{-{\sf k}_1}+|z|^2)^2
 {\bf 1}_{{\sf l}_2<|\xi|<4{\sf l}_2}
 \\
 &\lesssim
 (|\xi|^{-{\sf k}_1}+|z|^2)^2
 \cdot
 \eta^\frac{2q}{1-q}
 |\xi|^{\gamma-2-(\frac{2}{1-q}-\gamma)}
 {\bf 1}_{{\sf l}_2<|\xi|<4{\sf l}_2},
 \\
 {\sf t}_3
 &=
 {\sf U}_\infty^{q-2}
 \theta^2
 (1-\chi_2)^2
 {\bf 1}_{1<|\xi|<2{\sf l}_2}
 \\
 &\lesssim
 {\sf U}_\infty^{q-1}
 \theta
 {\bf 1}_{1<|\xi|<2{\sf l}_2}
 \\
 &\lesssim
 \eta^\frac{2(p-q)}{1-q}
 |\xi|^{\frac{2p}{1-q}-\gamma+2}
 \cdot
 \eta^\frac{2q}{1-q}
 |\xi|^{\gamma-2}
 {\bf 1}_{1<|\xi|<2{\sf l}_2},
 \\
 {\sf t}_4
 &=
 {\sf U}_\infty^{q-2}
 u_1^2
 {\bf 1}_{1<|\xi|<2{\sf l}_2}
 \\
 &\lesssim
 \eta^\frac{2(q-2)}{1-q}
 |\xi|^{\frac{2(q-2)}{1-q}}
 \cdot
 (T-t)^{2{\sf d}_1}
 \eta^\frac{4}{1-q}
 |\xi|^{2\gamma}
 {\bf 1}_{1<|\xi|<2{\sf l}_2}
 \\
 &\lesssim
 (T-t)^{2{\sf d}_1}
 \cdot
 \eta^\frac{2q}{1-q}
 |\xi|^{\gamma-2-(\frac{2}{1-q}-\gamma)}
 {\bf 1}_{1<|\xi|<2{\sf l}_2}.
 \end{align*}
 This proves the estimate of $N_2{\bf 1}_{|\xi|>1}$.
 A computation for $N_3$ is postponed to the end of this subsection.
 From ({\sf Ei}) and ({\sf Eiii}),
 we have
 \begin{align*}
 N_4[\epsilon,v,w]
 {\bf 1}_{|y|<1}
 &=
 f_2(u)
 \chi_1
 {\bf 1}_{|y|<1}
 \\
 &\lesssim
 (\lambda^{-\frac{n-2}{2}}{\sf Q})^q
 {\bf 1}_{|y|<1}
 \\
 &\lesssim
 \lambda^{-\frac{(n-2)q}{2}}
 {\bf 1}_{|y|<1}
 \\
 &=
 \lambda^{-\frac{(n-2)q}{2}+\frac{n+2}{2}}
 \sigma^{-1}
 \cdot
 \lambda^{-\frac{n+2}{2}}
 \sigma
 {\bf 1}_{|y|<1}
 \\
 &=
 \lambda^{-\frac{(n-2)q}{2}+\frac{n+2}{2}}
 ((T-t)^{-1}\lambda^2)^{-1}
 \cdot
 \lambda^{-\frac{n+2}{2}}
 \sigma
 {\bf 1}_{|y|<1}
 \\
 &=
 (T-t)
 \lambda^{\frac{(n-2)(1-q)}{2}}
 \cdot
 \lambda^{-\frac{n+2}{2}}
 \sigma
 {\bf 1}_{|y|<1},
 \\
 N_4[\epsilon,v,w]
 {\bf 1}_{|y|>1}
 &=
 f_2(u)
 \chi_1
 {\bf 1}_{|y|>1}
 \\
 &\lesssim
 (\lambda^{-\frac{n-2}{2}}{\sf Q})^q
 {\bf 1}_{1<|y|<2{\sf l}_1}
 +
 u_1^q
 {\bf 1}_{1<|y|<2{\sf l}_1}
 \\
 &\lesssim
 (\lambda^{-\frac{n-2}{2}}{\sf Q})^q
 {\bf 1}_{1<|y|<2{\sf l}_1}
 +
 (\lambda^{-\frac{n-2}{2}}\sigma)^q
 {\bf 1}_{1<|y|<2{\sf l}_1}
 \\
 &\lesssim
 (\lambda^{-\frac{n-2}{2}}{\sf Q})^q
 {\bf 1}_{1<|y|<2{\sf l}_1}
 \\
 &\lesssim
 \lambda^{-\frac{(n-2)q}{2}+\frac{n+2}{2}}
 \sigma^{-1}
 |y|^{-(n-2)q+2}
 \cdot
 \lambda^{-\frac{n+2}{2}}
 \sigma
 |y|^{-2}
 {\bf 1}_{1<|y|<2{\sf l}_1}
 \\
 &=
 (T-t)
 \lambda^{\frac{(n-2)(1-q)}{2}}
 |y|^{-(n-2)q+2}
 \cdot
 \lambda^{-\frac{n+2}{2}}
 \sigma
 |y|^{-2}
 {\bf 1}_{1<|y|<2{\sf l}_1}.
 \end{align*}
 This proves estimates of $N_4{\bf 1}_{|y|>1}$ for the case $-(n-2)q+2\leq0$.
 On the other hand,
 for the case $-(n-2)q+2>0$,
 it holds from ({\sf Eiii}) that
 \begin{align*}
 N_4[\epsilon,v,w]
 {\bf 1}_{|y|>1}
 &\lesssim
 \lambda^{-\frac{(n-2)q}{2}+\frac{n+2}{2}}
 \sigma^{-1}
 |y|^{-(n-2)q+2}
 \cdot
 \lambda^{-\frac{n+2}{2}}
 \sigma
 |y|^{-2}
 {\bf 1}_{1<|y|<2{\sf l}_1}
 \\
 &\lesssim
 \lambda^{-\frac{(n-2)q}{2}+\frac{n+2}{2}}
 \sigma^{-1}
 (\sigma^{-\frac{1}{n-2}})^{-(n-2)q+2}
 \cdot
 \lambda^{-\frac{n+2}{2}}
 \sigma
 |y|^{-2}
 {\bf 1}_{1<|y|<2{\sf l}_1}
 \\
 &=
 \lambda^{-\frac{(n-2)q}{2}+\frac{n+2}{2}}
 \sigma^{-1+q-\frac{2}{n-2}}
 \cdot
 \lambda^{-\frac{n+2}{2}}
 \sigma
 |y|^{-2}
 {\bf 1}_{1<|y|<2{\sf l}_1}
 \\
 &=
 \lambda^{-\frac{(n-2)q}{2}+\frac{n+2}{2}}
 ((T-t)^{-1}\lambda^2)^{-1+q-\frac{2}{n-2}}
 \cdot
 \lambda^{-\frac{n+2}{2}}
 \sigma
 |y|^{-2}
 {\bf 1}_{1<|y|<2{\sf l}_1}
 \\
 &=
 (T-t)^{1-q+\frac{2}{n-2}}
 \lambda^{\frac{(n-2)^2-8}{2(n-2)}+\frac{(6-n)q}{2}}
 \cdot
 \lambda^{-\frac{n+2}{2}}
 \sigma
 |y|^{-2}
 {\bf 1}_{1<|y|<2{\sf l}_1}.
 \end{align*}
 Therefore
 it follows that
 \begin{align*}
 N_4[\epsilon,v,w]
 {\bf 1}_{|y|>1}
 &\lesssim
 (T-t)
 \cdot
 \lambda^{-\frac{n+2}{2}}
 \sigma
 |y|^{-2}
 {\bf 1}_{1<|y|<2{\sf l}_1}
 \qquad
 \text{for any }
 q\in(0,1).
 \end{align*}
 Estimates for $N_5[w]$ is obvious.
 We write $N_6[w]$ as
 \begin{align*}
 N_6[w]
 &=
 \{
 \dot{\sf M}
 +
 f(u)
 -
 f_2(u)
 \}
 (1-\chi_4)
 \\
 &=
 \underbrace{
 \{
 \dot{\sf M}
 +
 f(u)
 -
 f_2(u)
 \}
 (1-\chi_{4,{\sf a}})
 }_{={\sf t}_5}
 +
 \underbrace{
 \{
 \dot{\sf M}
 +
 f(u)
 -
 f_2(u)
 \}
 (\chi_{4,{\sf a}}-\chi_4)
 }_{={\sf t}_6}.
 \end{align*}
 From the definition of ${\sf M}$
 ($\dot {\sf M}=f({\sf M})-f_2({\sf M})$),
 it holds that
 \begin{align*}
 {\sf t}_5
 &=
 \{
 \dot{\sf M}
 +
 f(u)
 -
 f_2(u)
 \}
 (1-\chi_{4,{\sf a}})
 \\
 &=
 (
 f(u)
 -
 f_2(u)
 +
 f({\sf M})
 -
 f_2({\sf M})
 )
 (1-\chi_{4,{\sf a}})
 \\
 &=
 (
 f(-{\sf M}+w)
 -
 f(-{\sf M})
 -
 f_2(-{\sf M}+w)
 +
 f_2(-{\sf M})
 )
 (1-\chi_{4,{\sf a}})
 \\
 &=
 (
 f'(-{\sf M}_\kappa
 )w
 -
 f_2'(-{\sf M}_\kappa) 
 w
 )
 (1-\chi_{4,{\sf a}})
 \\
 &\lesssim
 w
 {\bf 1}_{|x|>2},
 \end{align*}
 where
 ${\sf M}_\kappa=\kappa{\sf M}-(1-\kappa)w$.
 The estimate for ${\sf t}_6$ is obvious.
 \begin{align*}
 {\sf t}_6
 &=
 \{
 \dot{\sf M}
 +
 f(u)
 -
 f_2(u)
 \}
 (\chi_{4,{\sf a}}-\chi_4)
 \lesssim
 {\bf 1}_{1<|x|<4}.
 \end{align*}
 We finally discuss $N_3[\epsilon,v,w]$.
 We put
 \[
 n_3
 =
 f(u)
 -
 f_2(u)
 -
 f_2({\sf U}_\infty)
 +
 q{\sf U}_\infty^{q-1}
 (u+{\sf U}_\infty)
 -
 (\Delta_x\theta-q{\sf U}_\infty^{q-1}\theta)
 \chi_3.
 \]
 We now define two cut off functions.
 \begin{itemize}
 \item
 $\chi_{2,{\sf a}}=\chi_2(\cdot/2)$,
 \item
 $\chi_{3,{\sf e}}=\chi_3(2\cdot)$.
 \end{itemize}
 Then
 $N_3$ can be written as
 \begin{align*}
 N_3
 [\epsilon,v,w]
 &=
 n_3
 (1-\chi_2)
 \chi_4
 \\
 &=
 \underbrace{
 n_3
 (\chi_{2,{\sf a}}-\chi_2)
 }_{={\sf t}_7}
 +
 \underbrace{
 n_3
 (\chi_4-\chi_{3,{\sf e}})
 }_{={\sf t}_8}
 +
 \underbrace{
 n_3
 (1-\chi_{2,{\sf a}})
 \chi_{3,{\sf e}}
 }_{={\sf t}_9}.
 \end{align*}
 We first compute ${\sf t}_7$.
 \begin{align*}
 {\sf t}_7
 &=
 n_3
 (\chi_{2,{\sf a}}-\chi_2)
 \\
 &=
 f(u)
 (\chi_{2,{\sf a}}-\chi_2)
 -
 (\Delta_x\theta-q{\sf U}_\infty^{q-1}\theta)
 (\chi_{2,{\sf a}}-\chi_2)
 \\
 & \quad
 -
 \{
 f_2(u)
 +
 f_2({\sf U}_\infty)
 -
 q{\sf U}_\infty^{q-1}
 (u+{\sf U}_\infty)
 \}
 (\chi_{2,{\sf a}}-\chi_2)
 \\
 &\lesssim
 \{
 {\sf U}_\infty^p
 +
 |x|^\frac{2p}{1-q}
 +
 f_2''({\sf U}_\infty)
 (u+{\sf U}_\infty)^2
 \}
 {\bf 1}_{{\sf l}_2<|\xi|<4{\sf l}_2}
 \\
 &\lesssim
 |x|^\frac{2p}{1-q}
 {\bf 1}_{{\sf l}_2<|\xi|<4{\sf l}_2}
 +
 f_2''({\sf U}_\infty)
 (u+{\sf U}_\infty)^2
 {\bf 1}_{{\sf l}_2<|\xi|<4{\sf l}_2}.
 \end{align*}
 From
 ({\sf Evi}) and \eqref{eq9.2},
 we have
 \begin{align*}
 u
 +
 {\sf U}_\infty
 &=
 \lambda^{-\frac{n-2}{2}}
 {\sf Q}
 -
 \eta^{\frac{2}{1-q}}
 {\sf U}
 \chi_2
 +
 {\sf U}_\infty
 \chi_2
 -
 (\theta+\Theta_J)
 (1-\chi_2)
 \\
 &=
 \lambda^{-\frac{n-2}{2}}
 {\sf Q}
 -
 (
 \eta^{\frac{2}{1-q}}
 {\sf U}
 -
 {\sf U}_\infty
 )
 \chi_2
 -
 (\theta+\Theta_J)
 (1-\chi_2)
 \\
 &\lesssim
 \lambda^{-\frac{n-2}{2}}
 \sigma
 \cdot
 (
 \sigma^{-1}
 {\sf Q}
 )
 +
 \Theta_J
 \chi_2
 +
 \Theta_J
 (1-\chi_2)
 \\
 &\lesssim
 \eta^\frac{2}{1-q}
 (T-t)^{({\sf q}_1+{\sf q}_2)(n-2)}
 +
 \eta^\frac{2}{1-q}
 |\xi|^\gamma
 \\
 &\lesssim
 \eta^\frac{2}{1-q}
 |\xi|^\gamma
 \qquad
 \text{for }
 {\sf l}_2<|\xi|<4{\sf l}_2.
 \end{align*}
 We recall that
 ${\sf l}_2(t)=(T-t)^{{\sf b}}$.
 Since
 \begin{align*}
 |x|^\frac{2p}{1-q}
 &=
 |x|^\frac{2(p-q+2)}{1-q}
 \eta^{-\frac{4}{1-q}}
 |\xi|^{-2\gamma}
 \cdot
 |x|^\frac{2(q-2)}{1-q}
 \eta^\frac{4}{1-q}
 |\xi|^{2\gamma}
 \\
 &=
 \eta^\frac{2(p-q)}{1-q}
 |\xi|^{\frac{2(p-q+2)}{1-q}-2\gamma}
 \cdot
 |x|^\frac{2(q-2)}{1-q}
 \eta^\frac{4}{1-q}
 |\xi|^{2\gamma},
 \end{align*}
 there exists ${\sf b}_1>0$ depending only on $q,n$
 such that
 if $0<{\sf b}<{\sf b}_1$
 \begin{align*}
 |x|^\frac{2p}{1-q}
 <
 |x|^\frac{2(q-2)}{1-q}
 \eta^\frac{4}{1-q}
 |\xi|^{2\gamma}
 \qquad
 \text{for }
 {\sf l}_2<|\xi|<4{\sf l}_2.
 \end{align*}
 Due to these estimates,
 we obtain
 \begin{align*}
 {\sf t}_7
 &\lesssim
 |x|^\frac{2p}{1-q}
 {\bf 1}_{{\sf l}_2<|\xi|<4{\sf l}_2}
 +
 f_2''({\sf U}_\infty)
 (u+{\sf U}_\infty)^2
 {\bf 1}_{{\sf l}_2<|\xi|<4{\sf l}_2}
 \\
 &\lesssim
 |x|^\frac{2(q-2)}{1-q}
 \eta^\frac{4}{1-q}
 |\xi|^{2\gamma}
 {\bf 1}_{{\sf l}_2<|\xi|<4{\sf l}_2}
 \\
 &\lesssim
 \eta^\frac{2q}{1-q}
 |\xi|^{\gamma-2-\frac{2}{1-q}+\gamma}
 {\bf 1}_{{\sf l}_2<|\xi|<4{\sf l}_2}.
 \end{align*}
 We note that
 $u,\theta\lesssim{\sf U}_\infty$
 for $\tfrac{1}{2}{\sf r}_3<|x|<2$.
 Hence
 it follows that
 \begin{align*}
 {\sf t}_8
 &=
 n_3
 (\chi_4-\chi_{3,{\sf e}})
 \lesssim
 {\sf U}_\infty^q
 {\bf 1}_{\tfrac{1}{2}{\sf r}_3<|x|<2}.
 \end{align*}
 We finally compute ${\sf t}_9$.
 From the definition of $\chi_{2,{\sf a}},\chi_{3,{\sf e}}$,
 we note that
 \[
 u
 =
 -
 {\sf U}_\infty
 -
 (\theta+\Theta_J)
 +
 w
 \qquad
 \text{when }
 (1-\chi_{2,{\sf a}})
 \chi_{3,{\sf e}}
 \not=0.
 \]
 For simplicity,
 we write
 \begin{align*}
 \Theta_{\sf a}
 &=
 \Theta_J-w,
 \\
 P_1\theta
 &=
 \Delta_x\theta-q{\sf U}_\infty^{q-1}\theta.
 \end{align*}
 Then
 ${\sf t}_9$
 is expressed as
 \begin{align*}
 {\sf t}_9
 &=
 \{
 f({\sf U}_\infty+\Theta_{\sf a}+\theta)
 -
 f_2({\sf U}_\infty+\Theta_{\sf a}+\theta)
 +
 f_2({\sf U}_\infty)
 +
 {\sf U}_\infty^{q-1}
 (\Theta_{\sf a}+\theta)
 +
 P_1\theta
 \}
 \\
 & \quad
 \times
 (1-\chi_{2,{\sf a}})
 \chi_{3,{\sf e}}.
 \end{align*}
 From the Taylor expansion,
 we see that
 \begin{align*}
 {\sf t}_9
 &=
 \underbrace{
 \left\{
 \sum_{i=0}^N
 \tfrac{f^{(i)}({\sf U}_\infty)}{i!}
 (\Theta_{\sf a}+\theta)^i
 -
 \sum_{i=2}^N
 \tfrac{f_2^{(i)}({\sf U}_\infty)}{i!}
 (\Theta_{\sf a}+\theta)^i
 +
 P_1\theta
 \right\}
 (1-\chi_{2,{\sf a}})
 \chi_{3,{\sf e}}
 }_{={\sf t}_{9,1}}
 \\
 & \quad
 +
 \left(
 \underbrace{
 \tfrac{f^{(N+1)}({\sf U}_{\infty}^{(\kappa)})}{(N+1)!}
 }_{={\sf t}_{9,2}}
 -
 \underbrace{
 \tfrac{f_2^{(N+1)}({\sf U}_{\infty}^{(\kappa)})}{(N+1)!}
 }_{={\sf t}_{9,3}}
 \right)
 (\Theta_{\sf a}+\theta)^{N+1}
 (1-\chi_{2,{\sf a}})
 \chi_{3,{\sf e}}.
 \end{align*}
 Here
 we write
 ${\sf U}_{\infty}^{(\kappa)}
 =\kappa{\sf U}_\infty+(1-\kappa)(\Theta_{\sf a}+\theta)$.
 We recall that $\theta(x)$ satisfies
 (see Section \ref{sec_4.5})
 \begin{align*}
 P_1\theta
 &=
 \Delta\theta
 -
 q{\sf U}_\infty^{q-1}
 \theta
 \\
 &=
 -
 f({\sf U}_\infty)
 -
 \sum_{i=1}^N
 \tfrac{f^{(i)}({\sf U}_\infty)}{i!}
 (\theta-\theta_L)^i
 +
 \sum_{i=2}^N
 \tfrac{f_2^{(i)}({\sf U}_\infty)}{i!}
 (\theta-\theta_L)^i.
 \end{align*}
 By using this relation,
 we divide ${\sf t}_{9,1}$ into three parts.
 \begin{align*}
 {\sf t}_{9,1}
 &=
 \left
 \{
 \sum_{i=0}^N
 \tfrac{f^{(i)}({\sf U}_\infty)}{i!}
 (\Theta_{\sf a}+\theta)^i
 -
 \sum_{i=2}^N
 \tfrac{f_2^{(i)}({\sf U}_\infty)}{i!}
 (\Theta_{\sf a}+\theta)^i
 +
 P_1\theta
 \right
 \}
 (1-\chi_{2,{\sf a}})
 \chi_{3,{\sf e}}
 \\
 &=
 \underbrace{
 f'({\sf U}_\infty)
 (\Theta_{\sf a}+\theta_L)
 (1-\chi_{2,{\sf a}})
 \chi_{3,{\sf e}}
 }_{={\sf z}_1}
 \\
 & \quad
 +
 \underbrace{
 \sum_{i=2}^N
 \tfrac{f^{(i)}({\sf U}_\infty)}{i!}
 (
 (\Theta_{\sf a}+\theta)^i
 -
 (\theta-\theta_L)^i
 )
 (1-\chi_{2,{\sf a}})
 \chi_{3,{\sf e}}
 }_{={\sf z}_2}
 \\
 & \quad
 -
 \underbrace{
 \sum_{i=2}^N
 \tfrac{f_2^{(i)}({\sf U}_\infty)}{i!}
 (
 (\Theta_{\sf a}+\theta)^i
 -
 (\theta-\theta_L)^i
 )
 (1-\chi_{2,{\sf a}})
 \chi_{3,{\sf e}}
 }_{={\sf z}_3}.
 \end{align*}
 We note that
 $\theta_L(x)\ll\Theta_J(x,t)$ for $|x|<{\sf r}_3$.
 Hence
 we get
 \begin{align*}
 {\sf z}_2
 &=
 \sum_{i=2}^N
 \tfrac{f^{(i)}({\sf U}_\infty)}{i!}
 (
 (\Theta_{\sf a}+\theta)^i
 -
 (\theta-\theta_L)^i
 )
 (1-\chi_{2,{\sf a}})
 \chi_{3,{\sf e}}
 \\
 &=
 \sum_{i=2}^N
 f^{(i)}({\sf U}_\infty)
 (
 \sum_{l=0}^{i-1}
 c_{i,l}
 \theta^{l}
 \Theta_{\sf a}^{i-l}
 +
 \sum_{l=0}^{i-1}
 d_{i,l}
 \theta^l
 \theta_L^{i-l}
 )
 {\bf 1}_{|\xi|>2{\sf l}_2}
 {\bf 1}_{|x|<{\sf r}_3}
 \\
 &=
 \sum_{i=2}^N
 \sum_{l=0}^{i-1}
 (
 c_{i,l}
 {\sf U}_\infty^{p-i}
 \theta^l
 \Theta_{\sf a}^{i-l}
 +
 d_{i,l}
 {\sf U}_\infty^{p-i}
 \theta^l
 \theta_L^{i-l}
 )
 {\bf 1}_{|\xi|>2{\sf l}_2}
 {\bf 1}_{|x|<{\sf r}_3}
 \\
 &\lesssim
 \sum_{i=2}^N
 \sum_{l=0}^{i-1}
 {\sf U}_\infty^{p-i}
 \theta^l
 \Theta_{\sf a}^{i-l}
 {\bf 1}_{|\xi|>2{\sf l}_2}
 \chi_{3,{\sf e}}
 \\
 &=
 {\sf U}_\infty^{p}
 \sum_{i=2}^N
 \sum_{l=0}^{i-1}
 \{
 {\sf U}_\infty^{-1}
 \theta
 \}^l
 \{
 {\sf U}_\infty^{-1}
 \Theta_{\sf a}
 \}^{i-l}
 {\bf 1}_{|\xi|>2{\sf l}_2}
 {\bf 1}_{|x|<{\sf r}_3}
 \\
 &\lesssim
 {\sf U}_\infty^{p}
 \sum_{i=2}^N
 \left(
 \{
 {\sf U}_\infty^{-1}
 \Theta_{\sf a}
 \}^i
 +
 \sum_{l=1}^{i-1}
 \{
 {\sf U}_\infty^{-1}
 \theta
 \}^l
 \{
 {\sf U}_\infty^{-1}
 \Theta_{\sf a}
 \}^{i-l}
 \right)
 {\bf 1}_{|\xi|>2{\sf l}_2}
 {\bf 1}_{|x|<{\sf r}_3}.
 \end{align*}
 Since
 $|\Theta_J|+|\theta|+|w|<{\sf U}_\infty$
 for $|\xi|\gg1$, $|x|<{\sf r}_3$,
 we obtain
 \begin{align*}
 {\sf z}_2
 &\lesssim
 {\sf U}_\infty^{p}
 (
 \{
 {\sf U}_\infty^{-1}
 \Theta_{\sf a}
 \}^2
 +
 \{
 {\sf U}_\infty^{-1}
 \theta
 \}
 \{
 {\sf U}_\infty^{-1}
 \Theta_{\sf a}
 \}
 )
 {\bf 1}_{|\xi|>2{\sf l}_2}
 {\bf 1}_{|x|<{\sf r}_3}
 \\
 &=
 {\sf U}_\infty^{p-2}
 (
 \Theta_{\sf a}
 +
 \theta
 )
 \Theta_{\sf a}
 {\bf 1}_{|\xi|>2{\sf l}_2}
 {\bf 1}_{|x|<{\sf r}_3}.
 \end{align*}
 In the exact same manner,
 we can verify that
 \begin{align*}
 {\sf z}_3
 &\lesssim
 {\sf U}_\infty^{q-2}
 (\Theta_{\sf a}+\theta)
 \Theta_{\sf a}
 {\bf 1}_{|\xi|>2{\sf l}_2}
 {\bf 1}_{|x|<{\sf r}_3}.
 \end{align*}
 Combining these estimates,
 we obtain
 \begin{align*}
 {\sf t}_{9,1}
 &\lesssim
 {\sf z}_1
 +
 {\sf z}_2
 +
 {\sf z}_3
 \\
 &\lesssim
 {\sf U}_\infty^{p-1}
 \Theta_{\sf a}
 {\bf 1}_{|\xi|>2{\sf l}_2}
 {\bf 1}_{|x|<{\sf r}_3}
 +
 (
 {\sf U}_\infty^{p-2}
 +
 {\sf U}_\infty^{q-2}
 )
 (\Theta_{\sf a}+\theta)
 \Theta_{\sf a}
 {\bf 1}_{|\xi|>2{\sf l}_2}
 {\bf 1}_{|x|<{\sf r}_3}
 \\
 &\lesssim
 {\sf U}_\infty^{p-1}
 \Theta_{\sf a}
 {\bf 1}_{|\xi|>2{\sf l}_2}
 {\bf 1}_{|x|<{\sf r}_3}
 +
 {\sf U}_\infty^{q-2}
 (\Theta_{\sf a}+\theta)
 \Theta_{\sf a}
 {\bf 1}_{|\xi|>2{\sf l}_2}
 {\bf 1}_{|x|<{\sf r}_3}
 \\
 &\lesssim
 |x|^\frac{2(p-1)}{1-q}
 \Theta_{\sf a}
 {\bf 1}_{|\xi|>2{\sf l}_2}
 {\bf 1}_{|x|<{\sf r}_3}
 +
 {\sf U}_\infty^{q-2}
 \Theta_{\sf a}^2
 +
 |x|^\frac{2(p-1)}{1-q}
 \Theta_{\sf a}
 {\bf 1}_{|\xi|>2{\sf l}_2}
 {\bf 1}_{|x|<{\sf r}_3}.
 \end{align*}
 We next compute ${\sf t}_{9,2}$.
 We note that
 ${\sf U}_{\infty}^{(\kappa)}
 =\kappa{\sf U}_\infty+(1-\kappa)(\Theta_{\sf a}+\theta)
 =\kappa{\sf U}_\infty(1+o)$
 for $|\xi|\gg1$, $|x|<{\sf r}_3$.
 From this fact,
 we have
 \begin{align*}
 {\sf t}_{9,2}
 &=
 \tfrac{f^{(N+1)}({\sf U}_{\infty}^{(\kappa)})}{(N+1)!}
 (\Theta_{\sf a}+\theta)^{N+1}
 (1-\chi_{2,{\sf a}})
 \chi_{3,{\sf e}}
 \\
 &\lesssim
 {\sf U}_\infty^{p-N-1}
 \cdot
 (\Theta_{\sf a}^{N+1}+\theta^{N+1})
 {\bf 1}_{|\xi|>2{\sf l}_2}
 {\bf 1}_{|x|<{\sf r}_3}
 \\
 &=
 {\sf U}_\infty^p
 (
 ({\sf U}_\infty^{-1}\Theta_{\sf a})^{N+1}
 +
 ({\sf U}_\infty^{-1}\theta)^{N+1}
 )
 {\bf 1}_{|\xi|>2{\sf l}_2}
 {\bf 1}_{|x|<{\sf r}_3}
 \\
 &\lesssim
 {\sf U}_\infty^p
 (
 (
 {\sf U}_\infty^{-1}
 \Theta_{\sf a}
 )^2
 +
 ({\sf U}_\infty^{-1}\theta)^{N+1}
 )
 {\bf 1}_{|\xi|>2{\sf l}_2}
 {\bf 1}_{|x|<{\sf r}_3}.
 \end{align*}
 We now take $N$ large enough such that
 \[
 ({\sf U}_\infty^{-1}\theta)^{N+1}
 <
 (
 {\sf U}_\infty^{-1}
 \Theta_{\sf a}
 )^2
 \qquad
 \text{for }
 |x|<{\sf r}_3.
 \]
 Then
 it follows that
 \[
 {\sf t}_{9,2}
 \lesssim
 {\sf U}_\infty^{p-2}
 \Theta_{\sf a}^2
 {\bf 1}_{|\xi|>2{\sf l}_2}
 {\bf 1}_{|x|<{\sf r}_3}.
 \]
 From the same computation,
 we obtain
 ${\sf t}_{9,3}
 \lesssim
 {\sf U}_\infty^{q-2}
 \Theta_{\sf a}^2
 {\bf 1}_{|\xi|>2{\sf l}_2}
 {\bf 1}_{|x|<{\sf r}_3}$.
 As a consequence,
 we conclude
 \begin{align*}
 {\sf t}_9
 &=
 {\sf t}_{9,1}
 +
 {\sf t}_{9,2}
 +
 {\sf t}_{9,3}
 \\
 &\lesssim
 |x|^\frac{2(p-1)}{1-q}
 \Theta_{\sf a}
 {\bf 1}_{|\xi|>2{\sf l}_2}
 {\bf 1}_{|x|<{\sf r}_3}
 +
 {\sf U}_\infty^{q-2}
 \Theta_{\sf a}^2
 {\bf 1}_{|\xi|>2{\sf l}_2}
 {\bf 1}_{|x|<{\sf r}_3}.
 \end{align*}
 From the definition of $X_1$,
 we can verify that
 \begin{align*}
 \Theta_{\sf a}
 &\lesssim
 \Theta_J
 \lesssim
 \eta^\frac{2}{1-q}
 |\xi|^\gamma
 \qquad
 \text{for } |z|<1,
 \\
 \Theta_{\sf a}
 &\lesssim
 \Theta_J
 \lesssim
 |x|^{\gamma+2J}
 \qquad
 \text{for } 1<|z|<(T-t)^{-\frac{1}{3}},
 \\
 \Theta_{\sf a}
 &\lesssim
 \Theta_J
 +
 {\sf R}_1^{-1}
 (T-t)^{\frac{\gamma}{2}+J+{\sf d}_1}
 |z|^{\gamma+2J+3{\sf d}_1}
 \\
 &\lesssim
 (T-t)^{\frac{\gamma}{2}+J+{\sf d}_1}
 |z|^{\gamma+2J+3{\sf d}_1}
 \\
 &=
 (T-t)^{-\frac{1}{2}{\sf d}_1}
 |x|^{\gamma+2J+3{\sf d}_1}
 \qquad
 \text{for } (T-t)^{-\frac{1}{3}}<|z|<{\sf l}_{\sf out},
 \\
 \Theta_{\sf a}
 &\lesssim
 \Theta_J
 +
 {\sf R}_1^{-1}
 {\sf U}_\infty
 \lesssim
 {\sf U}_\infty
 \qquad
 \text{for } {\sf l}_{\sf out}\sqrt{T-t}<|x|<{\sf r}_3.
 \end{align*}
 Therefore
 we get
 \begin{align*}
 {\sf U}_\infty^{q-2}
 \Theta_{\sf a}^2
 &\lesssim
 {\sf U}_\infty^{q-2}
 \Theta_J^2
 \\
 &\lesssim
 \eta^\frac{2q}{1-q}
 |\xi|^{\gamma-2-\frac{2}{1-q}+\gamma}
 \qquad
 \text{for } |z|<1,
 \\
 {\sf U}_\infty^{q-2}
 \Theta_{\sf a}^2
 &\lesssim
 |x|^{-2-\frac{2}{1-q}}
 \Theta_J
 \cdot
 \Theta_J
 \\
 &\lesssim
 (T-t)^{-1}
 |z|^{-2}
 \cdot
 |x|^{\gamma+2J-\frac{2}{1-q}}
 \cdot
 \Theta_J
 \\
 &\lesssim
 (T-t)^{-1}
 (T-t)^{\frac{1}{6}(\gamma+2J-\frac{2}{1-q})}
 \Theta_J
 \\
 &\lesssim
 (T-t)^{-1+\frac{1}{6}(\gamma+2J-\frac{2}{1-q})}
 \Theta_J
 \qquad
 \text{for } 1<|z|<(T-t)^{-\frac{1}{3}},
 \\
 {\sf U}_\infty^{q-2}
 \Theta_{\sf a}^2
 &\lesssim
 |x|^{-2-\frac{2}{1-q}}
 \Theta_{\sf a}
 \cdot
 \Theta_{\sf a}
 \\
 &\lesssim
 |x|^{-2}
 \cdot
 (T-t)^{-\frac{1}{2}{\sf d}_1}
 |x|^{\gamma+2J-\frac{2}{1-q}+3{\sf d}_1}
 \cdot
 (T-t)^{-\frac{1}{2}{\sf d}_1}
 |x|^{\gamma+2J+3{\sf d}_1}
 \\
 &\lesssim
 |x|^{-2}
 \cdot
 (T-t)^{-\frac{1}{2}{\sf d}_1}
 \cdot
 (T-t)^{-\frac{1}{2}{\sf d}_1}
 |x|^{\gamma+2J+3{\sf d}_1}
 \\
 &\lesssim
 (T-t)^{-\frac{1}{3}}
 (T-t)^{-{\sf d}_1}
 |x|^{\gamma+2J+3{\sf d}_1}
 \\
 &\lesssim
 (T-t)^{-\frac{1}{3}-{\sf d}_1}
 |x|^{\gamma+2J+3{\sf d}_1}
 \qquad
 \text{for } (T-t)^{-\frac{1}{3}}<|z|<{\sf l}_{\sf out},
 \\
 {\sf U}_\infty^{q-2}
 \Theta_{\sf a}^2
 &\lesssim
 {\sf U}_\infty^q
 \qquad
 \text{for } {\sf l}_{\sf out}\sqrt{T-t}<|x|<{\sf r}_3.
 \end{align*}
 Furthermore
 we see that
 \begin{align*}
 |x|^\frac{2(p-1)}{1-q}
 \Theta_{\sf a}
 &\lesssim
 |x|^\frac{2(p-1)}{1-q}
 \Theta_J
 \\
 &=
 |x|^\frac{2(p-1)}{1-q}
 \eta^\frac{2}{1-q}
 |\xi|^\gamma
 \\
 &=
 |x|^\frac{2(p-1)}{1-q}
 \eta^2
 |\xi|^2
 \cdot
 \eta^\frac{2q}{1-q}
 |\xi|^{\gamma-2}
 \\
 &\lesssim
 |x|^{\frac{2(p-1)}{1-q}+2}
 \cdot
 \eta^\frac{2q}{1-q}
 |\xi|^{\gamma-2} \qquad
 \text{for } |z|<1,
 \\
 |x|^\frac{2(p-1)}{1-q}
 \Theta_{\sf a}
 &\lesssim
 |x|^\frac{2(p-1)}{1-q}
 \Theta_J
 \\
 &\lesssim
 (T-t)^\frac{p-1}{3(1-q)}
 \Theta_J
 \qquad
 \text{for } 1<|z|<(T-t)^{-\frac{1}{3}},
 \\
 |x|^\frac{2(p-1)}{1-q}
 \Theta_{\sf a}
 &\lesssim
 |x|^\frac{2(p-1)}{1-q}
 \cdot
 (T-t)^{-\frac{1}{2}{\sf d}_1}
 |x|^{\gamma+2J-\frac{2}{1-q}+3{\sf d}_1}
 \\
 &\lesssim
 (T-t)^{-\frac{1}{2}{\sf d}_1}
 |x|^{\gamma+2J-\frac{2}{1-q}+3{\sf d}_1}
 \qquad
 \text{for } (T-t)^{-\frac{1}{3}}<|z|<{\sf l}_{\sf out},
 \\
 |x|^\frac{2(p-1)}{1-q}
 \Theta_{\sf a}
 &\lesssim
 |x|^\frac{2(p-1)}{1-q}
 {\sf U}_\infty
 \qquad
 \text{for } {\sf l}_{\sf out}\sqrt{T-t}<|x|<{\sf r}_3.
 \end{align*}
 The proof is completed.
 \end{proof}

 \subsection{Proof of Lemma \ref{Lem7.3}\ and Lemma \ref{Lem8.3}\
 (computations of $F$)}
 \label{sec_9.5}
 We here provide estimates for $F_1,F_2$.
 These estimates prove Lemma \ref{Lem7.3} and Lemma \ref{Lem8.3}.
 We recall that
 \begin{itemize}
 \item
 $F_1[v,w]=
 \lambda^{-2}
 V
 (\eta^\frac{2}{1-q}v\chi_{\sf mid}+w)
 \chi_2$,
 
 \item
 $F_2[w]
 =
 q
 (\eta^{-2}{\sf U}^{q-1}(1-\chi_1)-{\sf U}_\infty^{q-1})
 w
 \chi_2$.
 \end{itemize}
 \begin{lem}\label{Lem_A.3}
 Assume {\rm(A1)} - {\rm(A3)}.
 Then it holds that
 \begin{align*}
 F_1[v,w]
 (1-\chi_{\sf in})
 {\bf 1}_{|\xi|<1}
 &\lesssim
 (T-t)^{{\sf d}_1}
 \lambda^{-\frac{n+2}{2}}
 \sigma
 |y|^{-4}
 {\bf 1}_{|y|>{\sf R}_{\sf in}}
 {\bf 1}_{|\xi|<1},
 \\
 F_1[v,w]
 {\bf 1}_{|\xi|>1}
 &\lesssim
 (\lambda^2\eta^{-2})
 \cdot
 \eta^\frac{2q}{1-q}
 |\xi|^{\gamma-4}
 {\bf 1}_{1<|\xi|<2{\sf l}_2},
 \\
 F_2[w]
 {\bf 1}_{|\xi|<1}
 &\leq
 {\sf R}_1^{-1}
 (T-t)^{{\sf d}_1}
 \eta^\frac{2q}{1-q}
 |\xi|^{\gamma-2}
 {\bf 1}_{|\xi|<1},
 \\
 F_2[w]
 {\bf 1}_{|\xi|>1}
 &\leq
 {\sf R}_1^{-1}
 (T-t)^{{\sf d}_1}
 \eta^\frac{2q}{1-q}
 |\xi|^{\gamma-2-(\frac{2}{1-q}-\gamma)}
 {\bf 1}_{1<|\xi|<2{\sf l}_2}.
 \end{align*}
 \end{lem}
 \begin{proof}
 Since the first two estimates are obvious,
 we omit their proofs.
 We recall that
 ${\sf U}_\infty<\eta^{-2}{\sf U}$
 and
 ${\sf U}(\xi)={\sf U}_\infty(\xi)+O(|\xi|^\gamma)$
 as $|\xi|\to\infty$.
 Therefore
 we get
 \begin{align*}
 F_2[w]
 {\bf 1}_{|\xi|<1}
 &=
 q
 (
 \eta^{-2}
 {\sf U}^{q-1}
 (1-\chi_1)
 -
 {\sf U}_\infty^{q-1}
 )
 w
 \chi_2
 \\
 &\lesssim
 {\sf U}_\infty^{q-1}
 w
 (1-\chi_{1})
 {\bf 1}_{|\xi|<1}
 \\
 &\lesssim
 {\sf R}_1^{-1}
 (T-t)^{{\sf d}_1}
 \eta^\frac{2q}{1-q}
 |\xi|^{\gamma-2}
 {\bf 1}_{|y|>{\sf l}_1}
 {\bf 1}_{|\xi|<1},
 \\
 F_2[w]
 {\bf 1}_{|\xi|>1}
 &=
 q
 (
 \eta^{-2}
 {\sf U}^{q-1}
 -
 {\sf U}_\infty^{q-1}
 )
 w
 \chi_2
 {\bf 1}_{|\xi|>1}
 \\
 &\lesssim
 \eta^{-2}
 {\sf U}_\infty^{q-2}
 |{\sf U}_\infty(\xi)-{\sf U}(\xi)|
 \cdot
 w
 {\bf 1}_{1<|\xi|<2{\sf l}_2}
 \\
 &\lesssim
 {\sf R}_1^{-1}
 (T-t)^{{\sf d}_1}
 \eta^\frac{2q}{1-q}
 |\xi|^{\gamma-2-(\frac{2}{1-q}-\gamma)}
 {\bf 1}_{1<|\xi|<2{\sf l}_2}.
 \end{align*}
 The proof is completed.
 \end{proof}

\section*{Acknowledgement}
The author is partly supported by
Grant-in-Aid for Young Scientists (B) No. 26800065.

 
\end{document}